\documentclass[11pt,reqno]{amsart}
\usepackage{indentfirst}
\usepackage{graphics,amsmath,amssymb,amsthm,mathrsfs,bm}
\usepackage{enumitem}
\usepackage{amsfonts}
\usepackage[leqno]{amsmath}
\usepackage{amsmath}
\usepackage{amssymb}
\usepackage{multicol}
\usepackage{stmaryrd}
\usepackage{setspace}
\usepackage{bbm}
\usepackage{cite}
\usepackage{epsfig}
\usepackage{color}
\usepackage{graphics}
\usepackage{graphicx}
\usepackage{epstopdf}
\usepackage{multicol,graphics}
\usepackage{mathtools}
\usepackage{indentfirst}
\usepackage{float}
\usepackage{hyperref}
\hypersetup{pdfstartview=Fit, pdfborder={0 0 0},breaklinks=true}

\newtheorem{theorem}{Theorem}[section]
\newtheorem{lemma}[theorem]{Lemma}
\newtheorem{corollary}[theorem]{Corollary}

\newtheorem{definition}[theorem]{Definition}
\newtheorem{remark}[theorem]{Remark}
\newtheorem{proposition}[theorem]{Proposition}
\numberwithin{equation}{section}

\allowdisplaybreaks[4]
\makeatletter
\def\@textbottom{\vskip \z@ \@plus 30pt}
\let\@texttop\relax
\makeatother
\begin{document}
\title[Principal spectral theory and asymptotic analysis]{ Principal spectral theory and asymptotic analysis for time-periodic cooperative systems with temporally nonlocal dispersal}

\author[Wu, Li, Sun and Vo]{Hao Wu$^{1}$, Wan-Tong Li$^{1,*}$, Jian-Wen Sun$^{1}$ and Hoang-Hung Vo$^{2}$}
\thanks{
$^1$School of Mathematics and Statistics, Lanzhou University, Lanzhou, Gansu 730000, China.\\
$^2$Faculty of Mathematics and Applications, Saigon University, Ho Chi Minh City, Vietnam.\\
$^*${\sf Corresponding author} (wtli@lzu.edu.cn)}
\date{\today}

\begin{abstract}
This paper investigates the principal spectral theory and the asymptotic behavior of the principal spectrum point for a class of time-periodic cooperative systems with nonlocal dispersal operators, incorporating both coupled and uncoupled nonlocal terms. By applying the theory of resolvent positive operators and their perturbations, we first establish criteria for the existence of the principal eigenvalue. We then construct sequences of smooth upper and lower approximating matrix-valued functions, each of whose corresponding operators satisfies the principal eigenvalue existence condition. This approximation framework allows the principal spectrum point to effectively substitute for the principal eigenvalue in characterizing the global dynamics of the nonlinear system. Moreover, it facilitates the study of the asymptotic behavior of the principal spectrum point with respect to parameters under fairly general assumptions. Subsequently, for systems with both coupled and uncoupled nonlocal terms, we analyze the asymptotic behavior of the principal spectrum point in terms of the dispersal rate, dispersal range, and frequency. Finally, we illustrate the applicability of our theoretical results through a Zika virus model and a stem cell model.

		

\end{abstract}

\maketitle

\tableofcontents

\section{ Introduction }
This paper is concerned with the principal spectral theory of the following time-periodic nonlocal operator 
\begin{equation*}
\mathcal{L}[\varphi](x,t) = -\tau\varphi_{t}(x,t)+D(x,t)\mathcal{P}[\varphi](x,t)+ A(x,t)\varphi(x,t), \enspace (x,t)\in\bar{\Omega}\times\mathbb{R}.
\end{equation*}
Here, $\Omega \subset \mathbb{R}^{N}$ is a smooth bounded domain, 
$\tau>0$ represents the frequency,
$$\varphi = (\varphi_1, \varphi_2, \cdots, \varphi_l)^{T}, \quad D(x,t) = (d_{ij}(x,t))_{l \times l}$$
and $A(x,t) = (a_{ij}(x,t))_{l \times l}$ are matrix-valued functions,
$\mathcal{P} = \text{diag} (\mathcal{P}_{1}, \mathcal{P}_{2},\cdots,\mathcal{P}_{l}  )$ with $\mathcal{P}_{i}$ being defined by
\begin{equation*}
\mathcal{P}_{i}[\varphi_{i}](x,t) = \int_{\Omega} k_{i}(x,y,t)\varphi_{i}(y,t)dy, \quad 1\leq i \leq l.
\end{equation*} 
The coefficients $d_{ij},a_{ij}$ and $k_{i}$ for $(1\leq i,j\leq l)$ are assumed throughout to satisfy the following conditions.
\begin{itemize}[leftmargin= 0.9cm]
\item[(H1)] $d_{ij}\in C(\bar{\Omega} \times \mathbb{R})$, $d_{ij}\geq 0$, $d_{ii}>0$ and $d_{ij}(x,t) = d_{ij}(x,t+1)$ for all $(x,t) \in \bar{\Omega}\times \mathbb{R}$.

\item[(H2)] $a_{ij}\in C(\bar{\Omega} \times \mathbb{R})$,   $a_{ij} \geq 0 \text{ for all } i\neq j$
and $a_{ij}(x,t) = a_{ij}(x,t+1)$ for all $(x,t) \in \bar{\Omega}\times \mathbb{R}$.

\item[(H3)] $k_{i}\in C\left(\mathbb{R}^{N} \!\times\! \mathbb{R}^{N}\! \times\! \mathbb{R}\right)$ is nonnegative,
$k_{i}(x,x,t)>0$ and
$$k_{i}(x,y,t) = k_{i}(x,y,t+1)
\enspace \text{for all} \enspace
(x,y,t) \in \mathbb{R}^{N} \! \times \! \mathbb{R}^{N} \!\times\! \mathbb{R}.$$
\end{itemize}

We say that the operator $\mathcal{L}$ is strongly coupled if the matrix-valued function $D$ is not diagonal. Otherwise, it is called weakly coupled.
In the weakly coupled case, if
$
\int_{\mathbb{R}^{N}} k_{i}(x,y,t)dy = 1
\enspace \text{for all} \enspace
(x,t)\in\bar{\Omega}\times \mathbb{R},
$
and the diagonal elements satisfy $a_{ii} = - d_{ii}\int_{\Omega}k_{i}(y,x,t)dy + \tilde{a}_{ii}$ or $a_{ii} = -d_{ii} + \tilde{a}_{ii}$ with $\tilde{a}$ independent of $d_{ii}$, then we call $\mathcal{L}$ a nonlocal dispersal operator with Neumann or Dirichlet boundary condition, respectively.
We refer the reader to \cite{AMRM} for the interpretation of Dirichlet/Neumann boundary conditions in nonlocal dispersal equation.
\subsection{Background and main contributions}
\indent

The spectral theory of linear operators constitutes a fundamental tool for analyzing the dynamics of evolution equations. For classical local dispersal equations, the principal eigenvalue of the linearized operator characterizes the stability of the system's equilibrium point and its long-time dynamical behavior.
Substantial progress has been made in this area, see the books \cite{CC2003,LL2022}.
Now it is well understood that nonlocal dispersal operators can better capture long-range
dispersal of species including humans (Andreu-Vaillo et al. \cite{AMRM} and Fife \cite{Fife}).
However, due to the lack of compactness, the principal
eigenvalue of nonlocal dispersal operators may not exist, see Coville \cite{Co10}, Hutson et al. \cite{HMMV03},
and Shen and Zhang \cite{SZ}.
Therefore, it is necessary to
establish conditions for its existence and to find alternative quantities that can substitute for the principal eigenvalue when it does not exist.

When $l=1$, i.e., in the scalar case, the operator $\mathcal{L}$ has been well studied.
For the one-dimensional case, Huston et al.\cite{HMMV03} were the first to obtain a result on the existence of a principal eigenvalue.
Subsequently, by introducing the concepts of the principal spectrum point or the generalized principal eigenvalue, Coville \cite{Co10}, Shen and Zhang\cite{SZ}, and Li et al. \cite{LCW} employed different methods to establish criteria for the existence of the principal eigenvalue in the general time-independent setting.
The existence theory for the principal eigenvalue of time-periodic nonlocal operators was established by Hutson, Shen and Vickers \cite{HSV08} and Rawal and Shen \cite{RS}.
Building on these existence results, an approximation method was developed in \cite{Co10, RS, SX1}, which constructs a sequence of operators that each possess a principal eigenvalue and converge to the original operator.
This method guarantees that, even if the principal eigenvalue may not exist, the principal spectrum point or the generalized principal eigenvalue, which coincide in many standard settings, can still serve as a threshold parameter for determining the dynamics of the nonlinear system
(see, e.g., \cite{BCV2,Co10,RS,ShenVo,SLLY}),

To overcome the difficulty in obtaining explicit expressions for the principal spectrum point or the generalized principal eigenvalue, many studies have focused on their asymptotic behavior with respect to parameters, such as dispersal rate, dispersal range, and
frequency.
For dispersal rate and dispersal range, Shen and Xie \cite{SX} investigated the time-independent case under  Dirichlet, Neumann, and periodic boundary conditions.
In the time-periodic setting, Shen and Vo \cite{ShenVo} studied the Dirichlet boundary condition, while Vo \cite{Vo1} investigated the Neumann condition. Su et al. \cite{SLLY} considered the influence of frequency.

A central problem in the study of the dispersal range of nonlocal operators is whether the solutions and the principal eigenvalue of nonlocal problems converge to those of the corresponding local dispersal problems.
This convergence was first established for the solution of linear scalar equations by the seminal works of
\cite{CER} and \cite{CERW}.
Following this, Berestycki et al. \cite{BCV1} proved that the generalized principal eigenvalue of the nonlocal dispersal operator with Dirichlet boundary condition converges to the principal eigenvalue of local dispersal operators.
Su et al. \cite{SLY} established a parallel result for Neumann boundary conditions.
For the time-periodic case, Shen and Xie \cite{SX} obtained comprehensive results for Dirichlet, Neumann, and periodic boundary conditions, showing that the solutions to the initial value problems, positive periodic solutions, and the principal spectrum point of the nonlocal dispersal equations all converge to those of the corresponding local equations.
For more studies on the convergence of nonlocal to local dispersal equations, we refer the reader to \cite{KR2,DN,SS,SunVo}.

It is natural to ask whether similar existence results and asymptotic behavior hold for cooperative systems.
In the weakly coupled case (i.e., $D$ is a diagonal matrix), Bao and Shen \cite{BSS17}, Liang et al. \cite{LZZ}, and Feng et al. \cite{Feng}, using various methods, extended parts of the existence theory from scalar operators to irreducible weakly coupled cooperative systems in the time-periodic setting.
Motivated by stem cell regeneration models proposed by Lei \cite{Lei1, Lei2}, Su et al. \cite{SLLW, SWZ} established the existence theory for time-independent cooperative systems that are strongly coupled (i.e., $D$ is non-diagonal).
However, to the best of our knowledge, the time-periodic case and the asymptotic behavior of such strongly coupled systems with nonlocal dispersal remain open.

On the other hand, the asymptotic behavior of the principal spectrum point for weakly coupled systems still presents several open problems.
Lin and Wang \cite{LinWang}  established the asymptotic behavior with respect to the dispersal rate for weakly coupled operators under both Dirichlet and Neumann boundary conditions.
For symmetric weakly coupled operators, Ninh and Vo \cite{NV2} investigated the asymptotic behavior with respect to the dispersal range with Dirichlet boundary condition.
Subsequently, Feng et al.\cite{Feng} studied the asymptotic behavior with respect to the dispersal rate, dispersal range, and frequency for weakly coupled operators in the time-periodic setting with Neumann boundary conditions.
However, their studies \cite{NV2,Feng} rely on the strong assumption that the principal eigenvalue exists.
Moreover, their analysis of the asymptotic behavior with respect to the dispersal range does not include the case of convergence to the principal eigenvalue of the corresponding local dispersal operator.
Furthermore, the asymptotic behavior of the principal spectrum point with Dirichlet boundary condition in a time-periodic setting has not been considered.

This work aims to address the open problems outlined above. The main contributions of this paper are as follows.

\textbf{(i)} Inspired by the work of \cite{Feng} and \cite{KR}, and utilizing the theory of resolvent positive operators and its perturbation theory developed by Thieme \cite{Th1,Th2}, we studied the principal spectrum point of the operator $\mathcal{L}$.
We obtained sufficient conditions that guarantees that the spectrum point is the principal eigenvalue. Furthermore, under the assumption that either $D$ or $A$ is irreducible, we can demonstrate the algebraic simplicity of the principal eigenvalue and the positivity of every component of its corresponding eigenfunction, which extends the conclusions in \cite{SWZ,SLLW} for strongly coupled systems to time-periodic operators.

\textbf{(ii)} We developed a new approximation method for the operator $\mathcal{L}$.
Recently, Wu et al. \cite{WFLS} and Wang and Zhang \cite{WZ} established the approximation methods for the operator $\mathcal{L}$ in the time-independent case. Subsequently, Wang and Zhang \cite{ZL1} extended these results to the weakly coupled time-periodic setting. However, the elements of the matrix-valued sequence obtained by this approximation method lack smoothness, which makes it difficult to use when discussing the asymptotic behavior of the principal spectrum point. On the other hand, compared to the scalar equation, the principal spectrum point for cooperative systems no longer possesses Lipschitz continuity with respect to the coupling matrix $A$, so the smoothing strategy used in \cite{WFLS} cannot be applied either.
In this paper, for both weak and strong coupling, we resolve these problems by constructing smooth upper and lower approximating sequences for the operator $\mathcal{L}$.
Using this result, we are able to remove the strong assumptions on the existence of the principal eigenvalue in \cite{NV2,Feng}, as well as the smoothness assumption on the coefficients in \cite{Feng}, and it facilitates the analysis of asymptotic behavior in the present paper.
Moreover, this result ensures that the principal spectrum point coincides with the generalized principal eigenvalue, and it can serve as a threshold parameter for determining the behavior of the nonlinear system, regardless of whether the principal eigenvalue exists.

\textbf{(iii)} We investigated the asymptotic behavior of the principal spectrum point of operator $\mathcal{L}$ with respect to the dispersal rate, dispersal range, and frequency in two settings: the weakly coupled case with Dirichlet boundary conditions and the strongly coupled case. In particular, for the case with Dirichlet boundary conditions, we show that the principal spectrum point associated with the nonlocal dispersal system converges to the principal eigenvalue of the corresponding local dispersal system as the dispersal range small enough. Furthermore, we also prove that both the solutions to the initial value problem and the positive time-periodic solution of the nonlinear nonlocal dispersal system also converge to those of the corresponding local dispersal system.
Moreover, in previous works on the asymptotic behavior \cite{Feng,SX,ShenVo,LinWang,NV2}, it is commonly assumed that the dispersal rate is constant and the kernel function depends only on a single spatial variable and time-independent. However, in applications it is more realistic to allow these quantities to vary in time. Therefore, we seek to derive, as far as possible, analogous results under more general, time-dependent assumptions.

\textbf{(iv)} We apply our theory to investigate the dynamic behavior of the Zika virus model and the stem cell regeneration model, which formally correspond to weakly coupled and strongly coupled systems, respectively.

\subsection{Notations and basic definitions}
\indent

Before stating our main results, we first introduce some notations and basic definitions.
Throughout this paper, we shall use the following notations.

\begin{itemize}[leftmargin= 0.8cm]
\item For $x=(x_1, x_2, \dots, x_l)^T \in \mathbb{R}^l$, we denote $|x| = \sum_{i=1}^l |x_i|$.

\item The nonnegative and strongly positive vector set of $\mathbb{R}^l$ are defined, respectively, as:
\begin{align*}
(\mathbb{R}^l)^+ &= \{x=(x_1, x_2, \dots, x_l)^T \mid x_i \in \mathbb{R}, x_i \geq 0, i=1, 2, \dots, l\},
\\
(\mathbb{R}^l)^{++} &= \{x=(x_1, x_2, \dots, x_l)^T \mid x_i \in \mathbb{R}, x_i > 0, i=1, 2, \dots, l\}.
\end{align*}
\item Let $X = C(\bar{\Omega}, \mathbb{R}^l)$ be the space of continuous functions. We define its subsets as:
\begin{align*}
X^+ &= \{u \in X \mid u(x) \in (\mathbb{R}^l)^+, x \in \bar{\Omega}\},
\\
X^{++} &= \{u \in X \mid u(x) \in (\mathbb{R}^l)^{++}, x \in \bar{\Omega}\}.
\end{align*}
\item Let $\mathcal{X} = \{u \in C(\bar{\Omega} \times \mathbb{R}, \mathbb{R}^l) \mid u(x, t+1) = u(x, t), (x, t) \in \bar{\Omega} \times \mathbb{R}\}$ be the space of 1-periodic functions in $t$. We define its subsets as:
\begin{align*}
\mathcal{X}^+ &= \{u \in \mathcal{X} \mid u(x, t) \in (\mathbb{R}^l)^+, (x, t) \in \bar{\Omega} \times \mathbb{R}\},
\\
\mathcal{X}^{++} &= \{u \in \mathcal{X} \mid u(x, t) \in (\mathbb{R}^l)^{++}, (x, t) \in \bar{\Omega} \times \mathbb{R}\}.
\end{align*}

\item Let $\mathcal{X}_1 = \{u \in C^{0,1}(\bar{\Omega} \times \mathbb{R}, \mathbb{R}^l) \mid u(x, t+1) = u(x, t), (x, t) \in \bar{\Omega} \times \mathbb{R}\}$, where $C^{0,1}(\bar{\Omega} \times \mathbb{R}, \mathbb{R}^l)$ denotes the class of functions that are continuous in $x$ and $C^1$ in $t$. We define its subsets as:
\begin{align*}
\mathcal{X}_1^+ &= \{u \in \mathcal{X}_1 \mid u(x, t) \in (\mathbb{R}^l)^+, (x, t) \in \bar{\Omega} \times \mathbb{R}\},
\\
\mathcal{X}_1^{++} &= \{u \in \mathcal{X}_1 \mid u(x, t) \in (\mathbb{R}^l)^{++}, (x, t) \in \bar{\Omega} \times \mathbb{R}\}.
\end{align*}

\item For $\varphi$, $\phi \in \mathbb{X}$, we write
$$
\begin{array}{r c l}
\varphi \geq \phi & \text{if} & \varphi - \phi \in
\mathbb{X}^{+},
\\
\varphi > \phi & \text{if} & \varphi - \phi \in \mathbb{X}^{+} \backslash \{ 0 \},
\\
\varphi \gg \phi & \text{if} & \varphi - \phi \in \mathbb{X}^{++} .
\end{array}
$$
where $\mathbb{X} = \mathbb{R}^{l},X,\mathcal{X},\mathcal{X}_{1}$.

\item Let $\mathbb{S}=\{1,2,\cdots,l\}$.

\item For $x\in\mathbb{R}^{N}$, $B_{r}(x)$ denotes the open ball of radius $r$ centered at the point $x$.

\item We denote by $I$ the identity operator.

\end{itemize}

Clearly, $ \mathbb{X}^{+} $ is the positive cone of $\mathbb{X}$ and $\mathbb{X}^{++}$ is the interior of $\mathbb{X}^{+}$, where $\mathbb{X}=X,\mathcal{X},\mathcal{X}_{1}$.
With these notations, the operator $\mathcal{L}$ is then considered as an unbounded linear operator on the space $\mathcal{X}$ with domain $\mathcal{X}_{1}$, namely,
\begin{equation*}
\mathcal{L} : \mathcal{X}_{1}\left( \subset \mathcal{X} \right) \to \mathcal{X}.
\end{equation*}

\begin{definition}{\rm
Let
\begin{equation*}
s(\mathcal{L}) = \sup\{\text{Re} \lambda \mid \lambda \in \sigma(\mathcal{L})\}.
\end{equation*}
We call $s(\mathcal{L})$ the principal spectrum point of  $\mathcal{L}$ .
If  $s(\mathcal{L})$  is an isolated eigenvalue of  $\mathcal{L}$  with a positive eigenfunction  $\phi$  (i.e.  $\phi \in \mathcal{X}^+ \backslash \{ 0 \}$), then  $s(\mathcal{L})$  is called the principal eigenvalue of $ \mathcal{L} $  or it is said that  $\mathcal{L}$  has a principal eigenvalue.
}
\end{definition}

It follows from \cite[Proposition 2.4]{LZZ1} and \cite[Theorem 1.4]{baihe} that the following holds.
\begin{lemma}\label{LEM1.2}
For any given $x\in\bar{\Omega}$, the eigenvalue problem
\begin{equation}\label{1-1}
\left\{
\begin{aligned}
&\!-\!\tau\frac{d\varphi(t)}{dt}+A(x,t)\varphi(t)=\lambda\varphi(t), &t\in\mathbb{R},\\
&\varphi(t)=\varphi(t+1), &t\in\mathbb{R},
\end{aligned}
\right.
\end{equation}
admits a principal eigenvalue $\lambda_{A}(x)$ with a positive eigenfunction $\varphi_{A}(x,t)$. Moreover, $\lambda_{A}(x)$ and $\varphi_{A}(x,t)$ are as smooth in $x$ as $A(x,t)$ in $x$, and when $A(x,t)\equiv A(x)$, $\lambda_{A}(x)$ is the largest real part of the eigenvalues of the matrix $A(x)$.
\end{lemma}

\subsection{Main results}
\indent
\newline
\newline
\indent
Recall that a matrix $C=(c_{ij})_{l\times l}$ is said to be irreducible, if for any nonempty and proper subset $\mathbb{I}$ of $\mathbb{S}=\{1,2,\cdots,l\}$, there exist $i\in \mathbb{I}$ and $j\in \mathbb{S}\setminus\mathbb{I}$ such that $c_{ij}\neq 0$.
For matrix-valued function $C(x,t)$ defined on $\bar\Omega\times\mathbb{R}$, We say that matrix-valued function $C$ is irreducible, if for any $(x,t)\in\bar{\Omega}\times \mathbb{R}$ the matrix is irreducible.

To ensure the principal eigenfunction is strictly positive, we introduce the following assumption.
\begin{itemize}[leftmargin=0.9cm]
\item[$ {\rm (\tilde{H}2) }$] (H2) holds and either $D$ is irreducible or $A$ is irreducible.
\end{itemize}
Define operator
$\mathcal{M}: \mathcal{X}\to \mathcal{X}$ and $\mathcal{N} : \mathcal{X}_{1} \to \mathcal{X}$ by
\begin{align*}
\mathcal{M}[\varphi](x,t)= D(x,t)\mathcal{P}[\varphi](x,t),
\enspace
\mathcal{N}[\varphi](x,t)= -\tau\varphi_{t}(x,t)+A(x,t)\varphi(x,t).
\end{align*}

\begin{theorem}[\textbf{Existence of the principal eigenvalue}]\label{THM1.3}
Assume that (H1)-(H3) hold. Then
$s(\mathcal{L})$ is the principal eigenvalue of $\mathcal{L}$ if one of the following holds:
\begin{itemize}[leftmargin=0.9cm]
\item[(i)] $s(\mathcal{L})>s(\mathcal{N})$.
\item[(ii)] There exists $\alpha > \max_{x\in\bar{\Omega}} \lambda_{A}(x)$ and $\phi\in\mathcal{X}^{+}\backslash \{0\}$ such that
$$
\mathcal{M}(\alpha I - \mathcal{N})^{-1}\phi \geq \phi.
$$
\end{itemize}
Moreover, if $(\tilde{H}2)$ holds, then $s(\mathcal{L})$ is algebraically simple, with an eigenfunction in $\mathcal{X}^{++}$.
Conversely, if $\lambda$ is an eigenvalue of $\mathcal{L}$ with an eigenfunction $u \in \mathcal{X}^{++}$, then $\lambda=s(\mathcal{L})$, and both $(i)$ and $(ii)$ hold.
\end{theorem}
We then obtain a non-locally-integrable condition to guarantee the existence of the principal eigenvalue.
\begin{theorem}\label{THM1.4}
Assume that (H1)-(H3) hold. Then
$s(\mathcal{L})$ is the principal eigenvalue of $\mathcal{L}$ if
there exists $\Omega_{0}\subset\Omega$ such that
\begin{equation*}
\frac{1}{ \max_{\bar{\Omega}}\lambda_{A}(x)- \lambda_{A}(x)}\not\in L^{1}(\Omega_{0}).
\end{equation*}
\end{theorem}
Next, we establish the approximation method for time-periodic cooperative systems.
To highlight the dependence on $A$, we write $s(\mathcal{L})$ as $s(\mathcal{L}(A))$.
\begin{theorem}[\textbf{Approximating the principal spectrum point}]\label{THM1.5}
Assume that (H1)-(H3) hold. Then there exist two sequences of smooth matrix-valued functions $\{A^{k}_{{\pm}}\}\in C^{\infty}(\bar{\Omega}\times\mathbb{R},\mathbb{R}^{l\times l})$  satisfying (H2) such that
\begin{equation*}
s(\mathcal{L}(A^{k}_{-}))
\leq
s(\mathcal{L}(A))
\leq s(\mathcal{L}(A^{k}_{+}))
\end{equation*}
and
\begin{equation*}
\lim_{k\to\infty} \left| s(\mathcal{L}(A_{\pm}^{k})) - s(\mathcal{L}(A)) \right|=0.
\end{equation*}
In particular, $s(\mathcal{L}(A_{\pm}^{k}))$ is the principal eigenvalue of the following eigenvalue problem
\begin{equation*}
\left\{
\begin{aligned}
& -\tau\varphi_{t}(x,t)+D(x,t)\mathcal{P}[\varphi](x,t)+ A^{k}_{\pm}(x,t)\varphi(x,t)  =  \lambda \varphi(x,t), & (x,t) \in \bar{\Omega} \times \mathbb{R}, \\
&\varphi(x,t+1) = \varphi(x,t), & (x,t) \in \bar{\Omega} \times \mathbb{R}.
\end{aligned}
	\right.
\end{equation*}
\end{theorem}

Using the approximation scheme established above, we can show that, regardless of whether the principal spectrum point is the principal eigenvalue or not, it is equivalent to the two generalized principal eigenvalues introduced below.
\begin{equation*}
\lambda_{p}(\mathcal{L}) = \sup \{\lambda\in\mathbb{R} \,|\, \exists \varphi\in\mathcal{X}^{++}_{1}\text{ such that }
-\mathcal{L}[\varphi] + \lambda\varphi   \leq 0 \text{ in } \bar{\Omega}\times\mathbb{R}  \},
\end{equation*}
and
\begin{equation*}
\lambda_{p}^{\prime}(\mathcal{L}) =\inf \{\lambda\in\mathbb{R} \,|\, \exists \varphi\in\mathcal{X}^{++}_{1}\text{ such that }
-\mathcal{L}[\varphi] + \lambda\varphi   \geq 0 \text{ in } \bar{\Omega}\times\mathbb{R}  \}.
\end{equation*}
Moreover, $\lambda_{p}(\mathcal{L})$ and $\lambda_{p}^{\prime}(\mathcal{L})$ can be expressed equivalently by the following variational form:
\begin{equation*}
\lambda_{p}(\mathcal{L}) =
\sup_{\varphi \in \mathcal{X}^{++}_{1}}
\inf_{  (x,t)\in\bar{\Omega} \times \mathbb{R},i\in\mathbb{S} }
\frac{\mathcal{L}_{i}[\varphi](x,t)}{\varphi_{i}(x,t)},
\quad
\lambda_{p}^{\prime}(\mathcal{L}) =
\inf_{\varphi \in \mathcal{X}^{++}_{1}}
\sup_{  (x,t)\in\bar{\Omega} \times \mathbb{R},i\in\mathbb{S} }
\frac{\mathcal{L}_{i}[\varphi](x,t)}{\varphi_{i}(x,t)},
\end{equation*}
where $\mathcal{L}_{i}[\varphi]$ is the $i-$th component of $\mathcal{L}[\varphi]$.

\begin{theorem}\label{THM1.6}
Suppose $(H1)$, $(\tilde{H}2)$ and $(H3)$ hold. Then we have
\begin{equation*}
s(\mathcal{L})=\lambda_{p}(\mathcal{L})=\lambda_{p}^{\prime}(\mathcal{L}).
\end{equation*}
\end{theorem}

In what follows, we investigate the asymptotic properties of the principal spectrum point/ generalized principal eigenvalue with respect to the dispersal rate, the dispersal range, and the frequency.

To more clearly separate the ``inflow"  and ``outflow"  effects in the dispersal process and to introduce the concepts of the dispersal rate and dispersal range.
We consider the following operator:
\begin{equation*}
L_{\tau,D,\sigma,m}[\varphi](x,t) = -\tau\varphi_{t}(x,t)+D(x,t)\mathcal{P}[\varphi](x,t) - D_{0}(x,t)\varphi(x,t) + A(x,t)\varphi(x,t),
\end{equation*}
where
$D_{0}(x,t) = \sigma^{-m} \text{diag}\left( d_{1}(x,t), d_{2}(x,t), \cdots, d_{l}(x,t) \right) $
with $\sigma>0$ and $m\geq 0$, and (H1)-(H3) still hold.
We first introduce two structural conditions relating  $D$ and $D_{0}$, and $k_i$.

\begin{itemize}[leftmargin=0.9cm]

\item[$\mathrm{(\tilde{H}1)}$]
(H1) holds and there exists a nonnegative constant matrix $C=(c_{ij})_{l\times l}$
such that
$D(x,t)= C D_{0}(x,t)$.

\item[$\mathrm{(\tilde{H}3)}$]
(H3) holds,
$k_{i}(x,y,t) = k_{i,\sigma}(x-y,t) =\frac{1}{\sigma^{N}} k_{i}\left(\frac{x-y}{\sigma} ,t \right)$,
and $\int_{\mathbb{R}^{N}} k_{i}(x,y,t) dy =1$ for $(x,t)\in \bar{\Omega}\times \mathbb{R}$.
\end{itemize}
Here, $d_{i}$ $(1\leq i \leq l)$ represent the dispersal rate, $\sigma$ characterizes the dispersal range, and $m\geq 0$ is the cost parameter.
Clearly, ${\rm (\tilde{H}1) }$ ensures that functions $d_{i}(x,t)$ $(1\leq i \leq l)$ are positive, continuous, and $1-$periodic in $t$.

For simplicity of notation, we sometimes write $L_{\tau,D,\sigma,m}$ simply as $L$.
When $L$ satisfies the following condition, $L$ is a nonlocal dispersal operator with Dirichlet boundary condition.
\begin{itemize} [leftmargin=0.8cm]
\item[(D)] $\mathrm{(\tilde{H}1)}$ and $\mathrm{(\tilde{H}3)}$ hold, and matrix $C$ is the identity matrix.
\end{itemize}

\begin{remark}\label{REMARK1}{\rm
Observe that, by replacing $A(x,t)$ with
\begin{equation*}
B(x,t) = -D_{0}(x,t) + A(x,t)
\end{equation*}
in Theorems \ref{THM1.3}, \ref{THM1.4}, \ref{THM1.5} and \ref{THM1.6},
the corresponding results for the operator $L$ follow immediately.
}
\end{remark}

Let $\bar{D} = \max_{1\leq i \leq l, (x,t)\in\bar{\Omega}\times\mathbb{R}}d_{i}(x,t)$
and
$\underline{D} = \min_{1\leq i \leq l, (x,t)\in\bar{\Omega}\times\mathbb{R}}d_{i}(x,t)$. The following result then holds.

\begin{theorem}[\textbf{Asymptotic behavior with respect to the dispersal rate}]\label{THM1.7}
Assume that $(\tilde{H}1)$, $(H2)$ and $(H3) $ hold. Then we have
\begin{equation*}
\lim_{ \bar{D} \to 0^{+}} s(L_{\tau,D,\sigma,m}) = \max_{x \in \bar{\Omega}} \lambda_{A}(x).
\end{equation*}
In addition, if $(\tilde{H}2)$, $(D)$ hold and $k_{i}(\cdot,t) \equiv k_{i}(\cdot)$ for $1 \leq i \leq l$, then
\begin{equation*}
\lim_{ \underline{D} \to +\infty} s(L_{\tau,D,\sigma,m}) = -\infty.
\end{equation*}
\end{theorem}

\begin{remark}{\rm
In Theorem \ref{THM1.7}, condition ${ (\tilde{H}2) }$ ensures the validity of Theorem \ref{THM1.6}. In addition, note that if ${  (\tilde{H}2) }$ and (D) hold simultaneously, then the matrix-valued function $A$ is irreducible.
}
\end{remark}

Let $\lambda^{local}$ be the principal eigenvalue of the following local dispersal eigenvalue problem with Dirichlet boundary condition.
\begin{equation}\label{1-2}
\left\{
\begin{aligned}
& -\tau\varphi_{t}(x,t)+D_{r}(x,t)R[\varphi](x,t)+A(x,t)\varphi(x,t) = \lambda\varphi(x,t) &\text{in }& \Omega\times \mathbb{R},\\
&\varphi(x,t)=0  &\text{on }& \partial\Omega\times \mathbb{R},\\
&\varphi(x,t)=\varphi(x,t+1)  &\text{in }& \Omega,
\end{aligned}
\right.
\end{equation}
where $R[\varphi]=(\Delta \varphi_{1}, \Delta \varphi_{2}, \cdots , \Delta \varphi_{l})$
and
$D_{r}(x,t)=diag\{ d_{r,1}(x,t), d_{r,2}(x,t),\cdots,d_{r,l}(x,t) \}$ with
\begin{equation*}
d_{r,i}(x,t)=\frac{d_{i}(x,t)}{2N}\int_{\mathbb{R}^{N}}k_{i}(z,t)|z|^{2}dz.
\end{equation*}
Subsequently, we demonstrate the dependence of $s(L_{\tau,D,\sigma,m})$ on the dispersal range. 
\begin{theorem}[\textbf{Asymptotic behavior with respect to the dispersal range}]\label{THM1.8}
Assume that $\mathrm{ (\tilde{H}1)},\mathrm{ ({H}2)}$, $\mathrm{ (\tilde{H}3)}$ hold and $k_{i}(\cdot,t) (1\leq i \leq l)$ are compactly supported. Then we have
\begin{equation*}
\lim_{\sigma\to +\infty} s(L_{\tau,D,\sigma,m}) =
\left\{
\begin{aligned}
& \max_{x\in\bar{\Omega}} \lambda_{A_{0}}(x), &\text{  if  }\;m=0,\\
& \max_{x\in\bar{\Omega}} \lambda_{A}(x), &\text{  if  }\; m>0,
\end{aligned}
\right.
\end{equation*}
where $A_{0} = -D_{0} + A$.
Assume further that ${ \rm (\tilde{H}2), (D)}$ hold and $k_{i}(\cdot,t) (1\leq i \leq l)$ are symmetric (in the sense that $k_{i}(z,t)=k_{i}(z^{\prime},t)$ whenever $|z|= |z^{\prime}|$).  Then the following hold.
\begin{itemize}
\item[(i)]  If $m\in [0,2)$ and each $d_{i}(x,t)\equiv d_{i}$ is a constant, then 
\begin{equation*}
\lim_{\sigma\to 0^{+}} s(L_{\tau,D,\sigma,m}) =
\max_{x\in\bar{\Omega}} \lambda_{A}(x).
\end{equation*} 

\item[(ii)] If $m=2$ and $d_{i}, a_{ij} \in C^{\alpha,\frac{\alpha}{2}}(\bar{\Omega} \times \mathbb{R})$, then \begin{equation*}
\lim_{\sigma\to 0^{+}} s(L_{\tau,D,\sigma,m}) = \lambda^{local}.
\end{equation*} 
\end{itemize}  
\end{theorem} 

\begin{remark}{\rm
In Theorem \ref{THM1.8}(ii)(a), the condition that $d_{i}$  is constant serves to guarantee that the choice of approximating sequence in Theorem \ref{THM1.5} is independent of  $d_{i}$.
If we assume that $s(L)$ is the principal eigenvalue of $L$, and that $d_{i}(x,t),a_{ij}(x,t)\in C^{4,0}(\bar{\Omega}\times \mathbb{R})$, then the parallel conclusions in Theorem \ref{THM1.8}(ii)(a) can be derived without using Theorem \ref{THM1.5}.
The requirement $d_{i},a_{ij}\in C^{\alpha,\frac{\alpha}{2}}(\bar{\Omega} \times \mathbb{R})$ in Theorem \ref{THM1.8}(ii)(b)  ensures the well-posedness of the eigenvalue problem for the local dispersal system.
}
\end{remark}

Next, we study the dependence of the principal spectrum point/generalized principal eigenvalue on the frequency.
We first establish a monotonicity result. In what follows, we allow the dispersal rate to depend on the frequency.
In addition, we assume that $D$ is independent of $(x,t)$ and impose symmetry conditions on $L$ so that it admits a well-defined adjoint operator $L^{*}$ with the same structure.
To this end, we introduce the following conditions:
\begin{itemize}[leftmargin=0.9cm]

\item[(F1)]
${\rm (H1) }$ holds.
Moreover, $D$ and $D_{0}$ are independent of $(x,t)$ and may depend on $\tau$, i.e., $D(x,t)=D(\tau)$ and $D_{0}(x,t)=D_{0}(\tau)$. The matrix $D(\tau)$ is symmetric, its off-diagonal entries are constants and its diagonal coincides with the outflow matrix $D_{0}(\tau)$, namely,
\begin{equation*}
\operatorname{diag} D(\tau)
 =  D_{0}(\tau).
\end{equation*}
In addition, $d_{ii}(\tau) \in C^{1}(0,\infty)$ for $ 1\le i\le l$.

\item[(F2)]
For all $1\leq i \leq l$,
$k_{i}(x,y,t) = k(x,y,t)$ and $k(x,y,t) = k(y,x,t)$ for $(x,y,t)\in \bar{\Omega}\times\bar{\Omega}\times\mathbb{R}$, and
$A(x,t)$ is symmetric for all $(x,t)\in\bar{\Omega}\times\mathbb{R}$.
\end{itemize}

Define the temporal averages
$\hat{\mathcal{P}} = \text{diag} (\hat{\mathcal{P}}_{1},  \cdots, \hat{\mathcal{P}}_{l}) : X \to X$  with
$$\hat{\mathcal{P}}_{i}[u_{i}](x) = \int_{\Omega} \int_{0}^{1} k_{i}(x,y,t)dt u_{i}(y) dy $$
and
$\hat{A}(x) = (\hat{a}_{ij}(x))_{l\times l}$ with $ \hat{a}_{ij}(x) = \int_{0}^{1} a_{ij}(x,t)dt$.
For simplicity, we use $'$ to denote differentiation with respect to $\tau$ in what follows.

\begin{theorem}[\textbf{Monotonicity with respect to frequency}]\label{THM1.9}
Assume that ${\rm (H2),(\tilde{H}3), (F1), (F2)}$ hold and $d_{ii}^{\prime}(\tau)\geq 0$, $ \left(\frac{d_{ii}(\tau)}{\tau} \right)^{\prime} \leq 0$ for  $1\leq i \leq l$.
Then $s(L_{\tau,D,\sigma,m})$ is nonincreasing in $\tau$.
Assume further that $s(L_{\tau,D,\sigma,m})$ is the principal eigenvalue, then the following conclusions hold.
\begin{itemize}[leftmargin=0.8cm]

\item[$(i)$] If there exists $1 \leq k\leq l$ such that $d_{kk}^{\prime}(\tau)> 0$, then $s^{\prime}(L_{\tau,D,\sigma,m})<0$.

\item[$(ii)$]
If $d_{ii}^{\prime}(\tau)=0$, $\left( \frac{d_{ii}(\tau)}{\tau} \right)^{\prime}<0$ for $1\leq i \leq l$ and the eigenvalue problem
\begin{equation*}
D \hat{\mathcal{P}}[\phi](x) - D_{0}\phi(x) + \hat{A}(x)\phi(x) = \lambda \phi(x)
\enspace \text{in} \enspace \bar{\Omega}
\end{equation*}
admits a principal eigenvalue with principal eigenfunction $\phi$, then $s^{\prime}(L_{\tau,D,\sigma,m}) = 0$ if and only if there exists some $1-$periodic function $r(t)\in C(\mathbb{R})$ with $\int_{0}^{1}r(t)dt=0$ satisfying
\begin{equation*}
r(t)\phi(x) = \left( D \mathcal{P}[\phi](x,t) - D\hat{\mathcal{P}}[\phi](x) \right) +   \left(A(x,t) - \hat{A}(x) \right)\phi(x)
\enspace  \text{in} \enspace \bar{\Omega}\times[0,1].
\end{equation*}
\end{itemize}
\end{theorem}

\begin{theorem}[\textbf{Asymptotic behavior with respect to frequency}]\label{THM1.10}
Suppose ${\rm (H1),(\tilde{H}2),(\tilde{H}3)}$ hold. Then the following conclusions hold:
\begin{itemize}[leftmargin=0.8cm]
\item[(i)]
Assume that $k_{i}$ and $a_{ij}$ belong to $C^{0,1}(\bar{\Omega} \times \mathbb{R})$ and $d_{ii}(\tau) \to b_{ii} $ as $\tau\to 0^{+}$ for some positive constant $b_{ii}$. If $s(N_{1}(t))$ is the principal eigenvalue of $N_{1}(t)$, then
\begin{equation*}
\lim_{\tau \to 0^{+}} s(L_{\tau,D,\sigma,m}) = \int_{0}^{1} s(N_{1}(t)) dt.
\end{equation*}
Here, $N_{1}(t)$ is defined by
\begin{equation*}
N_{1}(t)[u](x)= B\mathcal{P}[u](x)-B_{0}u(x)+A(x,t)u(x),\quad u\in X,
\end{equation*}
where $B= (b_{ij})_{l\times l}$ with $b_{ij}=d_{ij}$ for $i\neq j$ and $B_{0}= diag\{b_{11},b_{22},\cdots,b_{ll}\}$.
\item[(ii)] Assume that the conditions of Theorem \ref{THM1.9} hold and $d_{ii}(\tau) \to h_{ii} $ as $\tau\to \infty$ for some positive constant $h_{ii}$, then
\begin{equation*}
\lim_{\tau \to +\infty} s(L_{\tau,D,\sigma,m}) = s^{*},
\end{equation*}
where
\begin{equation*}
s^{*} \geq \min_{1\leq i \leq l} \left\{ h_{ii} \left( \frac{1}{|\Omega|} \int_{0}^{1} \int_{\Omega} \int_{\Omega} k(x-y,t)dydxdt - 1 \right) + \frac{1}{|\Omega|} \int_{0}^{1} \int_{\Omega} a_{ii}(x,t) dx dt \right\} .
\end{equation*}

%
%
\item[(iii)]
Assume that (D) holds, $k_{i} (1\leq i \leq l)$ independent of $t$ and $\min_{1\leq i \leq l}d_{ii}(\tau)\to +\infty$, then
\begin{equation*}
\lim_{\tau \to +\infty} s(L_{\tau,D,\sigma,m}) = -\infty.
\end{equation*}
\end{itemize}
\end{theorem}

\begin{remark}{\rm
We now summarize the asymptotic behaviors obtained under Dirichlet boundary conditions that differ from the corresponding results for Neumann boundary conditions in \cite{Feng}.
\begin{itemize}[leftmargin=0.78cm]
\item[(i)] For large dispersal rates, under Dirichlet boundary conditions, the principal spectrum point converges to $-\infty$ (see Theorem \ref{THM1.7}), whereas under Neumann boundary conditions it converges to the principal eigenvalue of the ordinary differential system corresponding to the spatially averaged matrix valued function $A$ (see \cite[Theorem 4.1 (i)]{Feng}).

\item[(ii)] For large dispersal ranges, under Neumann boundary conditions the principal spectrum point converges to $\sup_{x\in\bar{\Omega}} \lambda_{A}(x)$ (see \cite[Theorem 4.2 (i)]{Feng}); under Dirichlet conditions its limit depends on the value of
$m$ (see Theorem \ref{THM1.8}).

\item[(iii)]  Regarding the monotonicity with respect to the frequency, under Dirichlet boundary conditions, if there exists $1 \le k \le l$ such that $d_{kk}'(\tau) > 0$, then $\frac{\partial s(L)}{\partial \tau} < 0$ (see Theorem \ref{THM1.9}). In contrast, in the Neumann case it may happen that $\frac{\partial s(L)}{\partial \tau} = 0$ (see \cite[Theorem 3.2 (i)]{Feng}).

\item[(iv)]  As for the asymptotic behavior with respect to the frequency, under Dirichlet boundary conditions we have that the principal spectrum point converges to $-\infty$ as $\tau\to+\infty$ whenever $d_{ii}(\tau)\to +\infty$ (see Theorem \ref{THM1.10}); whereas, in the Neumann case, if in addition $\frac{d_{ii}(\tau)}{\tau} \to q_{i}$ as $\tau\to+\infty$  for some nonnegative constant $q_{i}$, then the principal spectrum point converges to the largest eigenvalue of the space--time average of the matrix-valued function $A$ (see \cite[Theorem 3.6 (ii)]{Feng}).
\end{itemize}
}
\end{remark}

\section{Principal spectral theory}
\subsection{The existence of the principal eigenvalue}
\indent

In this section, we establish the existence, uniqueness, and variational characterization of the principal eigenvalue.

Inspired by the works of \cite{KR} and \cite{Feng}, we employ the theory of resolvent positive operators to derive a necessary and sufficient condition for the existence of the principal eigenvalue. To this end, we first recall some results on the theory of resolvent positive operators, the readers can refer to Thieme \cite{Th1, Th2} for details.
Let $Z$ denote a Banach space and $Z^{+}$ be a closed convex cone that is normal and generating.
Denote the interior of $Z^{+}$ by $Z^{++}$.
A bounded linear operator $A : Z \to Z$ is called positive if $Ax \in Z^{+}$ for all $x \in Z^{+}$ and $A$ is not the $0$ operator, and strongly positive if $Ax \in Z^{++} $ for all $x\in Z^{+} \backslash \{ 0 \}$.

For any closed operator $A$, we define the spectral bound of $A$ by
\begin{equation*}
s(A) = \sup \{ Re \lambda \in \mathbb{R} : \lambda \in \sigma(A) \},
\end{equation*}
the real spectral bound of $A$ by
\begin{equation*}
s_{\mathbb{R}}(A) = \sup \{ \lambda \in \mathbb{R} : \lambda \in \sigma(A) \},
\end{equation*}
and the spectral radius of $A$ by
\begin{equation*}
r(A) = \sup \{  |\lambda|  : \lambda \in \sigma(A) \}.
\end{equation*}
Moreover, recalling the definition of the principal spectrum point, we know that for operator $\mathcal{L}$, it is precisely the spectral bound.

\begin{definition}{\rm
A closed operator $A$ in $Z$ is said to be resolvent positive if the resolvent set of $A$, $\rho(A)$, contains a ray $(\omega , \infty)$ and $(\lambda I - A)^{-1}$ is a positive operator for all $\lambda > \omega$.
}
\end{definition}

\begin{definition}{\rm
The operator $A = B + C$ is called a positive perturbation of $B$ if $B$ is a resolvent positive operator and $C : \text{dom}(B) \to Z$ is a positive linear operator.
}
\end{definition}

Define
$$
F_{\lambda} = C(\lambda I - B)^{-1}, \, \lambda > s(B).
$$
\begin{definition}{\rm
The operator $C :$ dom$(B) \to Z$ is called a compact perturbator of $B$ and $A = B+C$ a compact perturbation of $B$ if
\begin{equation*}
(\lambda I - B)^{-1} F_{\lambda} :
\overline{\text{dom}(B)} \to \overline{\text{dom}(B)}
\enspace\text{is compact for some}\enspace \lambda>s(B)
\end{equation*}
and
\begin{equation*}
(\lambda I - B)^{-1} (F_{\lambda})^{2} :
Z \to Z
\enspace\text{is compact for some}\enspace \lambda>s(B).
\end{equation*}
C is called an essentially compact perturbator of $B$ and $A = B + C$ an essentially compact perturbation of $B$ if there is some $n \in \mathbb{N}$ such that $(\lambda I - B)^{-1}(F_{\lambda})^{n} $ is compact for all $\lambda>s(B)$.
}
\end{definition}

\begin{theorem}[{\cite[Theorem 3.6]{Th2}}]\label{THM2.4}
Let $A=B+C$ be a positive perturbation of $B$.
Then $r(F_{\lambda})$ is a decreasing convex function of $\lambda>s(B)$, and exactly one of the following three cases holds:
\begin{itemize}[leftmargin=1cm]
\item[$(i)$] If $r(F_{\lambda})\geq 1$ for all $\lambda>s(B)$, then $A$ is not resolvent positive.

\item[$(ii)$] If $r(F_{\lambda})< 1$ for all $\lambda>s(B)$, then $A$ is resolvent positive and $s(A) = s(B)$.

\item[$(iii)$] If there exists $v>\lambda>s(B)$ such that $r(F_{v})<1\leq r(F_{\lambda})$, then $A$ is resolvent positive and $s(B)<s(A)<\infty$; further $s=s(A)$ is characterized by $r(F_{s})=1$.
\end{itemize}
\end{theorem}

\begin{theorem}[{\cite[Theorem 4.7]{Th1}}]\label{THM2.5}
Assume that $C$ is an essentially compact perturbator of $B$. Moreover assume that there exist $\lambda_{2}>\lambda_{1}>s(B)$ such that $r(F_{\lambda_{1}})\geq 1 >r(F_{\lambda_{2}})$. Then $s(B)<s(A)$ and the following statement holds:

\begin{itemize}[leftmargin=1cm]
\item [$(i)$] s(A) is an eigenvalue of $A$ associated with positive eigenfunctions of $A$ and $A^{*}$, has finite algebraic multiplicity, and is a pole of the resolvent of $A$. If $C$ is a compact perturbator of $B$, then all spectral values $\lambda$ of $A$ with $Re \lambda \in (s(B),s(A) ]$ are poles of the resolvent of $A$ and are eigenvalues of $A$ with finite algebraic multiplicity.

\item[$(ii)$] 1 is an eigenvalue of $F_{s(A)}$ and is associated with an eigenfunction $\omega \in Z$ of $F_{s(A)}$ such that $(\lambda I - B)^{-1}\omega \in Z^{+}$.
Actually $s(A)$ is the largest $\lambda \in \mathbb{R}$ for which 1 is an eigenvalue of $F_{\lambda}$.

\end{itemize}
\end{theorem}

Recall that operators
$\mathcal{M}: \mathcal{X}\to \mathcal{X}$ and $\mathcal{N} : \mathcal{X}_{1} \to \mathcal{X}$ by
\begin{align*}
\mathcal{M}[\varphi](x,t)= D(x,t)\mathcal{P}[\varphi](x,t),
\enspace
\mathcal{N}[\varphi](x,t)= -\tau\varphi_{t}(x,t)+A(x,t)\varphi(x,t).
\end{align*}
By \cite[Proposition 2.9]{Feng}, we have the following result.
\begin{lemma}\label{LEM2.6}
The resolvent operator $(\eta I - \mathcal{N})^{-1}$ exists when $Re \eta > \sup_{x\in\bar{\Omega}} \lambda_{A}(x)$.
Moreover, $\mathcal{N}$ is a resolvent positive operator and
$s(\mathcal{N}) = \max_{x\in\bar{\Omega}} \lambda_{A}(x)$.
\end{lemma}
Hence, we can obtain that
the operator $\mathcal{L} = \mathcal{M} + \mathcal{N}$ is a positive perturbation of $\mathcal{N}$.
\begin{lemma}
The operator $\mathcal{M}$ is a compact perturbator of $\mathcal{N}$ and also an essentially compact perturbator of $\mathcal{N}$.
Thus the operator $\mathcal{L}$ is a compact perturbation and also an essentially compact perturbation of $\mathcal{N}$.
\end{lemma}
\begin{proof}
Notice that $(\eta I - \mathcal{N})^{-1}$ is a bounded linear operator for $\eta > s(\mathcal{N})$.
Hence, to prove the desired results, it is sufficient to prove that $\mathcal{M}(\eta I - \mathcal{N})^{-1}$ is a compact operator for $\eta > s(\mathcal{N})$.
Choose any bounded sequence $\{ u_{n} \} \subset
\mathcal{X}$.
It then follows from the boundedness of $(\eta I - \mathcal{N})^{-1}$ that sequence $\{ (\eta I - \mathcal{N})^{-1}u_{n} \} $ is bounded.
Let $w_n=(\eta I-N)^{-1}u_n$. Since $\{w_n\}$ is bounded in $\mathcal{X}$ and the kernels $k_i(x,y,t)$ are uniformly continuous in $(x,t)$ on the compact set $\bar\Omega\times\bar\Omega\times[0,1]$, it follows that for every $\varepsilon>0$ there exists $\delta>0$ such that
\[
|M[w_n](x,t)-M[w_n](x',t')|<\varepsilon
\]
whenever $|(x,t)-(x',t')|<\delta$, uniformly in $n$. Hence, $\{M[w_n]\}$ is uniformly bounded and equicontinuous in $\mathcal{X}$, and by the Arzel\`a--Ascoli theorem $M(\eta I-N)^{-1}$ is compact.
\end{proof}

Let $\{U_{\lambda}(t,s);t \geq s \}$ and $\{V_{\lambda}(t,s); t\geq s \}$, respectively, be the evolution operator on $X$ determined by
\begin{equation*}
\tau\varphi_{t}=A(x,t)\varphi(x,t)-\lambda\varphi(x,t),\quad
(x,t)\in \bar\Omega\times\mathbb{R}
\end{equation*}
and
\begin{equation*}
\tau\varphi_{t}=D(x,t)\mathcal{P}[\varphi](x,t)+A(x,t)\varphi(x,t)-\lambda\varphi(x,t), \quad
(x,t)\in \bar\Omega\times\mathbb{R}.
\end{equation*}
For the existence of the evolution operator, we refer the reader to \cite[Chapter 5]{Pazy}.
\begin{lemma}\label{LEM2.8}
$U_{\lambda}(t,s)$ and $V_{\lambda}(t,s)$ are positive on $X$ for any $t\geq s$ and $\lambda\in\mathbb{R}$. Moreover, if ${ \rm (\tilde{H}2) }$ holds,
$V_{\lambda}(t,s)$ is strongly positive on $X$ for any $t> s$ and $\lambda\in\mathbb{R}$.
\end{lemma}
\begin{proof}
Following similar arguments as in the proof of \cite[Proposition 2.1]{SZ}, we can obtain $U_{\lambda}(t,s)$ and $V_{\lambda}(t,s)$ are positive on $X$ for any $t\geq s$ and $\lambda\in\mathbb{R}$.

Next, we assume that ${ \rm (\tilde{H}2) }$ holds.
Let $u^{s}\in X^{+}\backslash \{ 0\}$ and $u(x,t) = V_{\lambda}(t,s)u^{s}(x)$.
Next, we prove $u(x,t)>0$ for all $(x,t)\in\bar{\Omega}\times (s,\infty)$.
Since $u^{s} \in X^{+}\backslash \{ 0\}$, there exists $i_{0} \in \mathbb{S}$ and $x_{0}\in\Omega$ such that $u^{s}_{i_{0}}(x_{0})>0$.
Note that $u_{i_{0}}$ satisfies
\begin{align*}
\tau \partial_{t} u_{i_{0}}(x,t) &=
\sum_{j=1}^{l} d_{i_{0}j}(x,t)\int_{\Omega}k_{j}(x,y,t)u_{j}(y,t)dy + \sum_{j=1}^{l}a_{i_{0}j}(x,t)u_{j}(x,t) - \lambda u_{i_{0}}(x,t) \\
&\geq
d_{i_{0}i_{0}}(x,t)\int_{\Omega}k_{i_{0}}(x,y,t)u_{i_{0}}(y,t)dy + a_{i_{0}i_{0}}(x,t)u_{i_{0}}(x,t) - \lambda u_{i_{0}}(x,t)
\end{align*}
for all $(x,t)\in\bar{\Omega}\times (s,\infty)$.
Let
$v(x,t) = e^{\frac{K}{\tau}(t-s)} u_{i_{0}}(x,t)$, it follows that
\begin{equation*}
\tau \partial_{t} v(x,t) \geq d_{i_{0}i_{0}}(x,t) \int_{\Omega} k_{i_{0}}(x,y,t)v(y,t)dy + (K + a_{i_{0}i_{0}}(x,t) - \lambda) v(x,t)
\end{equation*}
for all $(x,t)\in\bar{\Omega}\times (s,\infty)$.
By choosing $K$ large enough, we get
\begin{align*}
\tau v(x_{0},t) \geq &\tau u_{i_{0}}^{s}(x_{0}) + \int_{s}^{t} d_{i_{0}i_{0}}(x_{0},z)\int_{\Omega}k_{i_{0}}(x_{0},y,z)v(y,z)dydz \\
&+ \int_{s}^{t} (K + a_{i_{0}i_{0}}(x_{0},z) - \lambda) v(x_{0},z) dz > 0
\end{align*}
for all $t>s$. Moreover, since $k_{i_{0}}(x,x,t)>0$, there exists $r>0$ such that for all $|x-x_{0}|\leq r$,
$\int_{\Omega}k_{i_{0}}(x,y,t)v(y,t)dy>0$.
Hence, we can obtain that $v(x,t)>0$ for all $|x-x_{0}|\leq r$ and $t>s$.
By repeating the process, we get $v(x,t)>0$ for all $\bar{\Omega}\times (s,\infty)$. This implies that
\begin{equation*}
u_{i_{0}}(x,t)>0
\enspace \text{for all} \enspace (x,t)\in\bar{\Omega}\times (s,\infty).
\end{equation*}

Next, we assume that $D$ is irreducible.
It follows that there exists $i_{k} \in \mathbb{S}\backslash\{ i_{0} \}$ such that $d_{i_{k}i_{0}}>0$. Note that $u_{i_{k}}$ satisfies
\begin{align*}
\tau \partial_{t} u_{i_{k}}(x,t) &=
\sum_{j=1}^{l} d_{i_{k}j}(x,t)\int_{\Omega}k_{j}(x,y,t)u_{j}(y,t)dy + \sum_{j=1}^{l}a_{i_{k}j}(x,t)u_{j}(x,t) - \lambda u_{i_{k}}(x,t) \\
&\geq
d_{i_{k}i_{0}}(x,t)\int_{\Omega}k_{i_{0}}(x,y,t)u_{i_{0}}(y,t)dy + a_{i_{k}i_{k}}(x,t)u_{i_{k}}(x,t) - \lambda u_{i_{k}}(x,t).
\end{align*}
It follows that $u_{i_{k}}(x,t)>0$ for all $\bar{\Omega}\times (s,\infty)$.
Using the irreducibility of $D$ again, there exits $i_{q}\in \mathbb{S} \backslash \{ i_{0},i_{k}\}$ such that $d_{i_{q},i_{0}}>0$ or $d_{i_{q},i_{k}} >0$. Similarly, we can prove $u_{i_{q}}(x,t)>0$ for all $\bar{\Omega}\times (s,\infty)$. By repeating the process, we get $u(x,t)\gg 0$ for all $\bar{\Omega}\times (s,\infty)$.

If $A$ is irreducible, we can get the desired conclusion similarly.
\end{proof}

\begin{corollary}\label{COR2.9}
Suppose $(\tilde{H}2)$ holds.
Let $\alpha$ be a constant. Assume that $u \in \mathcal{X}_{1}^{+} \backslash \{ 0 \}$ satisfies
\begin{equation*}
(\mathcal{L} - \alpha I)[u](x,t)\leq 0
\enspace \text{for all} \enspace (x,t)\in \bar{\Omega}\times \mathbb{R}.
\end{equation*}
then we have $ u(x,t) \gg 0 $ for all $(x,t)\in \bar{\Omega}\times \mathbb{R}$.
\end{corollary}

For $\lambda\in\mathbb{R}$, define operator $\mathcal{Q}_{\lambda} : X \to X$ by
\begin{equation*}
\mathcal{Q}_{\lambda} \varphi = e^{-\frac{\lambda}{\tau}} V_{0}(1,0)\varphi.
\end{equation*}

\begin{proposition}\label{PP2.10}
There exists $\lambda_{0}\in\mathbb{R}$ such that
\begin{equation*}
r(\mathcal{Q}_{\lambda_{0}})=r(e^{-\frac{\lambda_{0}}{\tau}}V_{0}(1,0) )=1.
\end{equation*}
Moreover, the operator $\mathcal{L}$ is  resolvent positive and $s(\mathcal{L})=\lambda_{0} = \tau \ln r(V_{0}(1,0))$.
\end{proposition}

\begin{proof}
The idea of the proof follows from \cite[Proposition 3.6]{DKR} and \cite[Proposition 2.10]{Feng}. For completeness and readers' convenience, we provide a detailed and modified proof.

Consider the resolvent equation
\begin{equation*}
\varphi=(\lambda I- \mathcal{L})^{-1}\phi\enspace\text{for}\enspace\phi\in\mathcal{X} \text{ and }\lambda\in\rho(\mathcal{L}).
\end{equation*}
It follows that
\begin{equation*}
\tau \partial_{t} \varphi(x,t) = D(x,t)\mathcal{P}[\varphi](x,t) + A(x,t)\varphi(x,t)-\lambda\varphi(x,t) + \phi(x,t)
\enspace \text{for all} \enspace (x,t)\in \bar{\Omega} \times \mathbb{R}.
\end{equation*}
Using the variation of constant formula, we get
\begin{equation*}
\varphi(x,t) = e^{-\frac{\lambda}{\tau}t}V_{0}(t,0)\varphi(x,0)
+ \int_{0}^{t} e^{-\frac{\lambda}{\tau}(t-s)}V_{0}(t,s)\frac{\phi(x,s)}{\tau} ds
\enspace \text{for all} \enspace (x,t)\in \bar{\Omega} \times \mathbb{R}.
\end{equation*}
Since $\varphi(x,0)=\varphi(x,1)$ for all $x\in\bar\Omega$, we get
\begin{equation*}
\left( I-\mathcal{Q}_{\lambda} \right)\varphi(x,0)=\int_{0}^{1}e^{-\frac{\lambda}{\tau}(1-s)}V_{0}(1,s)\frac{\phi(x,s)}{\tau} ds
\enspace \text{for all} \enspace x\in \bar{\Omega}.
\end{equation*}
Hence, if $1\in\rho(\mathcal{Q}_{\lambda})$, we have
\begin{equation}\label{2-1}
\begin{aligned}
\left[ (\lambda I - \mathcal{L})^{-1}\phi \right](x,t)
=&
e^{-\frac{\lambda}{\tau} t}V_{0}(t,0) (I-\mathcal{Q}_{\lambda})^{-1}
\left( \int_{0}^{1} e^{-\frac{\lambda}{\tau}(1-s)} V_{0}(1,s) \frac{\phi(x,s)}{\tau} ds \right)\\
&+ \int_{0}^{t} e^{-\frac{\lambda}{\tau}(t-s)}V_{0}(t,s)\frac{\phi(x,s)}{\tau} ds.
\end{aligned}
\end{equation}
It follows that $\lambda\in\rho(\mathcal{L})$ if and only if $1\in\rho(\mathcal{Q}_{\lambda})$.

Let $d_{ij}^{*}(x) = \min_{t\in [0,1]} d_{ij}(x,t)$, $k^{*}_{i}(x,y) = \min_{t\in [0,1]} k_{i}(x,y,t)$, and $a_{ij}^{*} = \min_{(x,t)\in \bar\Omega\times\mathbb{R}}a_{ij}(x,t)$.
Set $D^{*}(x) = (d_{ij}^{*}(x))_{l\times l}$ and $A^{*}= (a_{ij}^{*})_{l\times l}$.
Define operators $\mathcal{P}^{*}$ by replacing function $k_{i}$ with  $k_{i}^{*}$ in the definition of $\mathcal{P}$.
Denote the semigroup generated by
$ \frac{1}{\tau} \left(D^{*}\mathcal{P}^{*}+A^{*} \right)$ by $W(t)$.
It follows that $V_{0}(t,s) \geq W(t-s)$ in the positive operator sense. Note that $A^{*}$ is a constant matrix, by \cite[Corollary 1.3]{SWZ}, there exists a principal eigenpair
$(\lambda^{*},\varphi^{*})$ with $\lambda^{*} \in \mathbb{R}$ satisfying
\begin{equation*}
D^{*}(x)\mathcal{P}^{*}[\varphi^{*}](x)+A^{*}\varphi^{*}(x) =\lambda^{*}\varphi^{*}(x)
\enspace \text{in} \enspace\bar{\Omega}.
\end{equation*}
Moreover, by \cite[Theorem 1.5]{SWZ}, there holds
\begin{equation}\label{2-2}
s(D^{*}\mathcal{P}^{*}+A^{*} ) = \lambda^{*}.
\end{equation}
On the other hand, using the spectral mapping theorem (see \cite[Lemma 5.8]{Th2}), we have
\begin{equation*}
e^{ t \sigma( \frac{D^{*}\mathcal{P}^{*} + A^{*}}{\tau}  ) } = \sigma\left(W(t)\right)\backslash\{0\}\enspace\text{for all}\enspace t>0.
\end{equation*}
Since $W(t)$ is a positive operator, we have $r(W(t))\in\sigma\left(W(t)\right)$ for all $t>0$ by \cite[Proposition 4.1.1]{MNP}.
This together with \eqref{2-2} implies that
\begin{equation*}
r(W(t))
= e^{ s(\frac{D^{*}\mathcal{P}^{*} + A^{*}}{\tau}) t }
= e^{ \frac{\lambda^{*}}{\tau} t} \enspace\text{for all}\enspace t>0.
\end{equation*}
Therefore, we obtain
\begin{equation*}
r(\mathcal{Q}_{\lambda^{*}})
= r\left(e^{ -\frac{\lambda^{*}}{\tau} } V_{0}(1,0)\right)
\geq r\left(e^{ -\frac{\lambda^{*}}{\tau} } W(1)\right)
= 1.
\end{equation*}
On the other hand, we can easily check that $\lim_{\lambda\to\infty} r(\mathcal{Q}_{\lambda})=0$ and $r(\mathcal{Q}_{\lambda})$ is strictly decreasing with respect to $\lambda\in\mathbb{R}$ by the definition of $\mathcal{Q}_{\lambda}$.
Consequently, there exists a unique $\lambda_{0}\in\mathbb{R}$ such that $r(\mathcal{Q}_{\lambda_{0}})=1$.
It then follows that $r(\mathcal{Q}_{\lambda}) < r(\mathcal{Q}_{\lambda_{0}}) =1 $
and
$(I - \mathcal{Q}_{\lambda} )^{-1}$ exists for $\lambda>\lambda_{0}$.
In addition, since $\mathcal{Q}_{\lambda}$ is positive, $1 = r(\mathcal{Q}_{\lambda_{0}}) \in \sigma (\mathcal{Q}_{\lambda_{0}})$.
This implies that $1 \in \rho(\mathcal{Q}_{\lambda})$ for $\lambda>\lambda_{0}$.
Hence, $\rho(\mathcal{L})$ contains the ray $(\lambda_{0},+\infty)$ and $\lambda_{0} = s_{\mathbb{R}}(\mathcal{L})$.
Moreover, $(\lambda I-\mathcal{L})^{-1}$ is positive for $\lambda>\lambda_{0}$ by \eqref{2-1}.
Therefore, $\mathcal{L}$ is a resolvent positive operator.
Note that $\mathcal{X}$ is a Banach space with a normal and generating cone $\mathcal{X}^{+}$.
It then follows from \cite[Theorem 3.5]{Th1} that $s(\mathcal{L})=s_{\mathbb{R}}(\mathcal{L})=\lambda_{0}$. The proof is complete.
\end{proof}

\begin{proof}[Proof of Theorem \ref{THM1.3}.]
The proof is divided into three steps. 

\textbf{Step 1.} In this step, we prove that if (i) or (ii) holds, then $s(\mathcal{L})$ is the principal eigenvalue.
Suppose (i) holds.
Since $\mathcal{L}$ is a resolvent positive operator and $s(\mathcal{L})>s(\mathcal{N})$, we can obtain that the case $(iii)$ in Theorem \ref{THM2.4} will happen.
It then follows from Theorem \ref{THM2.5} that $s(\mathcal{L})$ is the principal eigenvalue of $\mathcal{L}$.

Then we assume that (ii) holds.
Let $\psi=(\alpha I - \mathcal{N})^{-1}\phi$, it follows that $\mathcal{L}\psi\geq\alpha\psi$ and $\psi(x,0)=\psi(x,1)$ for all $x\in\bar{\Omega}$.
Using the positivity of the evolution operator $V_{\alpha}(1,0)$ (see Lemma \ref{LEM2.8}), we have $V_{\alpha}(1,0)\psi(x,0) \geq \psi(x,1)$.
It follows that $V^{n}_{\alpha}(1,0)\psi(x,0)\geq\psi(x,0)$ for $n\geq 1$.
This together with the Gelfand's formula implies that $r(V_{\alpha}(1,0))\geq 1$.
Moreover, it is easy to check that
\begin{equation*}
e^{-\frac{\alpha}{\tau}}r(V_{0}(1,0))=r(e^{-\frac{\alpha}{\tau}}V_{0}(1,0))=r(V_{\alpha}(1,0))\geq 1.
\end{equation*}
In addition, by Proposition \ref{PP2.10}, we have
$e^{- \frac{s(\mathcal{L})}{\tau} } r(V_{0}(1,0))=1$. As a result, we get
\begin{equation*}
s(\mathcal{L})\geq \alpha >  \max_{x\in\bar{\Omega}} \lambda_{A}(x) = s(\mathcal{N}).
\end{equation*}
It then follows from (i) that $s(\mathcal{L})$ is the principal eigenvalue of $\mathcal{L}$.

\textbf{Step 2.} In this step, we prove that if ${ \rm (\tilde{H}2) }$ holds, then $s(\mathcal{L})$ is algebraically simple and admits an eigenfunction in $\mathcal{X}^{++}$.
Let $\varphi$ be the corresponding principal eigenfunction of $s(\mathcal{L})$.
Since ${ \rm (\tilde{H}2) }$ holds, we derive from Lemma \ref{LEM2.8} that $\varphi \in \mathcal{X}_{1}^{++}$.
Next, we prove that $s(\mathcal{L})$ is algebraically simple, that is,
dim$\left(\cup_{k=1}^{\infty} \text{ker}(\mathcal{L} - s(\mathcal{L})I)^{k}\right) = 1$.
It suffices to show
dim$\left(\text{ker}(\mathcal{L}-s(\mathcal{L})I)^{2}\right)=1$.
We first prove that dim$\left(\text{ker}(\mathcal{L}-s(\mathcal{L})I)\right)=1$.
Let $\phi\in N(\mathcal{L}-s(\mathcal{L})I)$, it follows that
\begin{equation*}
-\tau\phi_{t}(x,t) + D(x,t)\mathcal{P}[\phi](x,t) + A(x,t)\phi(x,t) = s(\mathcal{L})\phi(x,t),\quad (x,t)\in \bar\Omega\times\mathbb{R}.
\end{equation*}
Replacing $\phi$ by $-\phi$ if necessary, we may assume
$\sup_{(x,t)\in\bar\Omega\times[0,1]}\phi(x,t)>0$.
Let
$$
\delta^{*}_{1} = \sup\{ \delta \,|\, \varphi-\delta \phi \geq 0 \text{ for all } (x,t)\in \bar\Omega\times\mathbb{R}  \}.
$$
Then $\delta_{1}^{*} \in (0,\infty)$.
By the definition of $\delta^{*}_{1}$, we have $V_{1} = \varphi-\delta^{*}_{1}\phi \geq 0$.
If $V_{1}(x,t)=0$ for all $(x,t)\in \bar{\Omega}\times \mathbb{R}$, then the proof is done.
If there exists $(x_{0},t_{0}) \in \bar{\Omega}\times [0,1]$  such that $V_{1}(x_{0},t_{0})>0$.
Then, using $(\mathcal{L} - s(\mathcal{L})I)V_{1}=0$ and Corollary \ref{COR2.9}, we can obtain that
$V_{1}(x,t)>0$ for all $(x,t)\in \bar{\Omega}\times [0,1]$.
This implies that there exists $\delta_{0}>0$ such that
$\varphi - (\delta^{*}_{1}+\delta_{0})\phi \geq 0$, which contradicts the definition of $\delta^{*}_{1}$.
Hence, $\phi = \frac{1}{\delta^{*}_{1}} \varphi$.

Now suppose $\psi\in N(\mathcal{L}-s(\mathcal{L})I)^{2}$. Then
$(\mathcal{L}-s(\mathcal{L})I)\psi=c\varphi$ for a constant $c$.
Next, we prove that $c=0$. If not, we divide both sides by $c$ and still denote $\frac{1}{c}\psi$ by $\psi$. It follows that $(\mathcal{L}-s(\mathcal{L})I)\psi=\varphi\geq 0$. Similarly, we define
$$\delta^{*}_{2} = \sup\{ \delta \,|\, \varphi-\delta \psi \geq 0 \text{ for } (x,t)\in \bar\Omega\times[0,1]  \} $$
and $V_{2}=\varphi-\delta^{*}_{2}\psi$.
Note that $(\mathcal{L} - s(\mathcal{L})I )V_{2} \leq 0$.
By the same arguments as above, we obtain that $\psi$ is a constant multiple of $\varphi$.
Hence, $(\mathcal{L}-s(\mathcal{L})I)\psi=0$. This step is complete.

\textbf{Step 3.} In this step, we prove that if $\lambda$ is an eigenvalue of $\mathcal{L}$ with an eigenfunction $u \in \mathcal{X}^{++}$, then $\lambda=s(\mathcal{L})$, and both (i) and (ii) hold.
Clearly, $\lambda \leq s(\mathcal{L})$.
In what follows, we prove $\lambda \geq s(\mathcal{L})$.
By Proposition \ref{PP2.10}, we have $s(\mathcal{L}) = \tau \ln r(V_{0}(1,0))$.
It then follows from Gelfand's formula that
\begin{equation}\label{2-3}
\frac{s(\mathcal{L})}{\tau} = \lim_{t\to\infty} \frac{1}{t} \ln \| V_{0}(t,0) \|.
\end{equation}
Note that $V_{0}(t,0)u(x,0) = e^{\frac{\lambda}{\tau}t}u(x,t)$.
For any $v_{0} \in X^{+}\backslash \{ 0 \}$, choose a constant $C_{0}$ such that $v_{0}(x) \leq C_{0}u(x,0)$ for all $x\in\bar{\Omega}$.
By Lemma \ref{LEM2.8}, we get
\begin{equation*}
V_{0}(t,0)v_{0}(\cdot) \leq C_{0}V_{0}(t,0)u(\cdot,0) = C_{0}e^{\frac{\lambda}{\tau}t}u(\cdot,t)
\enspace\text{for all}\enspace t>0.
\end{equation*}
This implies that there exists $M>0$ such that $\| V_{0}(t,0) \| \leq Me^{\frac{\lambda}{\tau}t}$.
By a direct computation, there holds
\begin{equation}\label{2-4}
\frac{1}{t} \ln \| V_{0}(t,0) \| \leq \frac{1}{t}\ln M + \frac{\lambda}{\tau}
\enspace\text{for all}\enspace t>0.
\end{equation}
Thus, from \eqref{2-3} and \eqref{2-4}, we get  $\lambda \geq s(\mathcal{L})$.

Next, we prove that if $s(\mathcal{L})$ is the principal eigenvalue of $\mathcal{L}$ with an eigenfunction $\varphi \in \mathcal{X}^{++}_{1}$, then (i) and (ii) hold.
Notice that $\varphi$ satisfies
\begin{equation}\label{2-5}
-\tau \partial_{t} \varphi(x,t) + D(x,t)\mathcal{P}[\varphi](x,t) + A(x,t)\varphi(x,t) = s(\mathcal{L})\varphi(x,t)
\enspace \text{for all} \enspace (x,t)\in\bar{\Omega}\times[0,1].
\end{equation}
By considering the adjoint problem of \eqref{1-1}, there exists function $\phi\in\mathcal{X}_{1}^{+}\backslash
\{ 0 \}$ satisfying
\begin{equation}
\left\{
\begin{aligned}
& \tau\frac{d\phi^{T}(x,t)}{dt}+ \phi^{T}(x,t)A(x,t) =\lambda_{A}(x)\phi^{T}(x,t), &t\in\mathbb{R},\\
&\phi(x,t)=\phi(x,t+1), &t\in\mathbb{R}.
\end{aligned}
\right.
\end{equation}
Assume that $x_{0}\in\bar{\Omega}$ satisfies $\lambda_{A}(x_{0}) = \max_{x\in\bar{\Omega}} \lambda_{A}(x)=s(\mathcal{N})$.
Multiplying \eqref{2-5} by $\phi^{T}$, taking $x=x_{0}$,
and integrating in $t$ over $[0,1]$, we obtain
\begin{equation*}
\int_{0}^{1}\varphi^{T}(t)D(x_{0},t)\mathcal{P}[\varphi](x_{0},t)dt
=
(s(\mathcal{L})-s(\mathcal{N}))
\int_{0}^{1} \phi^{T}(t)\varphi(x_{0},t)dt.
\end{equation*}
Hence, (i) holds.
Moreover, it is easy to check that $\mathcal{M}[\varphi]$ satisfies $\mathcal{M}(s(\mathcal{L})I - \mathcal{N})^{-1} \mathcal{M}[\varphi] \geq \mathcal{M}[\varphi]$.
As a result, (ii) also holds.
\end{proof}

\begin{proof}[Proof of Theorem \ref{THM1.4}]
Thanks to Theorem \ref{THM1.3} (ii), it suffices to construct a function $\phi\in\mathcal{X}^{+}\backslash\{ 0\}$ such that
\begin{equation*}
\mathcal{M}(\alpha I - \mathcal{N})^{-1}\phi \geq \phi
\enspace \text{for some} \enspace
\alpha > \max_{x\in\bar{\Omega}} \lambda_{A}(x).
\end{equation*}
The construction is similar to the proof of \cite[Proposition 3.4]{BSS17}, we omit it here.
\end{proof}

\subsection{Approximation of the principal spectrum point}
\indent

\begin{proof}[Proof of Theorem \ref{THM1.5}]
The proof is divided into three steps. 

{\bf Step 1.} In this step, we construct a smooth approximation for each element $a_{ij}$.

Consider the following functions:
\begin{equation*}
\phi(x)=\left\{
\begin{aligned}
&C_{2}\exp \left\{ \frac{1}{|x|^{2}-1} \right\} &\text{ if } |x|<1,\\
&0  &\text{ if } |x|\geq 1,
\end{aligned}
	\right.
\quad
\eta(t)=\left\{
\begin{aligned}
&C_{1}\exp \left\{ \frac{1}{t^{2}-1} \right\} &\text{ if } |t|<1,\\
&0  &\text{ if } |t|\geq 1,
\end{aligned}
	\right.
\end{equation*}
where $C_{1},C_{2}>0$ are constants such that $\int_{\mathbb{R}^{N}}\phi(x)dx=1$ and $\int_{\mathbb{R}} \eta(t)dt=1 $. Define
\begin{equation*}
\phi_{\varepsilon}(x)= \frac{1}{\varepsilon^{N}}\phi\left( \frac{x}{\varepsilon} \right),
\enspace
\eta_{\varepsilon}(t)= \frac{1}{\varepsilon}\eta \left(\frac{t}{\varepsilon}\right),
\enspace\text{and}\enspace
\psi_{\varepsilon}(t)=\sum_{k \in \mathbb{Z}} \eta_{\varepsilon}(t-k),
\end{equation*}
for $\varepsilon<\frac{1}{3}$.
It is easy to check that $\psi_{\varepsilon}(t)$ is 1-periodic. Let $\Omega^{'}\subset\mathbb{R}^{N}$ satisfying $\Omega\Subset\Omega^{'}$ and continuously extend $a_{ij}(\cdot,t)$ to $\Omega^{\prime}$.
Define
\begin{equation*}
b^{k}_{ij}(x,t) = \int_{\Omega^{\prime}}\int_{0}^{1}\phi_{\varepsilon_{k}}(x-y)\psi_{\varepsilon_{k}}(t-s)a_{ij}(y,s)dsdy,
\end{equation*}
where $\varepsilon_{k} \to 0$ as $k\to\infty$.
Clearly, $b^{k}_{ij}(x,t)$ is $1-$periodic in $t$. Moreover,
by \cite[Appendix C, Theorem 7]{Evans}, $b^{k}_{ij}(x,t) \in C^{\infty}(\bar{\Omega}\times\mathbb{R})$ and
\begin{equation*}
\| b^{k}_{ij} - a_{ij}\|_{\mathcal{X}} = \delta^{k}_{ij} \to 0
\enspace \text{as} \enspace k\to\infty.
\end{equation*}

{\bf Step 2.} In this step, we construct both upper and lower smooth approximations for matrix-valued function $A$.

We begin by constructing a lower approximation to the off-diagonal elements. Let $1\leq i\neq j\leq l$ be fixed.
Define
\begin{equation*}
c_{ij}^{k}(x,t) = \max \{ b_{ij}^{k}(x,t) - 2\delta_{ij}^{k} , 0 \}.
\end{equation*}
It follows that
\begin{equation}\label{2-7}
0 \leq c_{ij}^{k} \leq a_{ij}
\enspace\text{and}\enspace
\| c_{ij}^{k} - a_{ij} \|_{\mathcal{X}} \leq 3\delta^{k}_{ij}.
\end{equation}
Similarly, we can extend $c_{ij}^{k}$ to $\Omega^{\prime}$ and define
\begin{equation*}
u^{k}_{ij}(x,t) = \int_{\Omega^{\prime}}\int_{0}^{1}
\phi_{\gamma_{k}}(x-y)\psi_{\gamma_{k}}(t-s) c^{k}_{ij}(y,s)dsdy,
\end{equation*}
where $\gamma_{k} \to 0$ as $k\to\infty$. Taking $\gamma_{k}$ small enough, such that
\begin{equation}\label{2-8}
\| u^{k}_{ij} - c_{ij}^{k} \|_{\mathcal{X}} \leq \delta_{ij}^{k}.
\end{equation}
Let
\begin{equation*}
\alpha^{k}_{ij} = \sup_{(x,t)\in\bar{\Omega}\times [0,1]} \{ u^{k}_{ij}(x,t) - c^{k}_{ij}(x,t)  \}_{+}.
\end{equation*}
Define $v^{k}_{ij}(x,t) =  u_{ij}^{k}(x,t) - \alpha^{k}_{ij} $. Clearly,
\begin{equation}\label{2-9}
0\leq \{v^{k}_{ij}\}_{+} \leq c_{ij}^{k}
\enspace \text{and} \enspace
\| \{v^{k}_{ij}\}_{+} - c_{ij}^{k} \|_{\mathcal{X}} \leq 2\delta_{ij}^{k}.
\end{equation}
Consider the following functions:
\begin{equation*}
\chi(s) =
\begin{cases}
0, & s \leq 0, \\
\exp \left( -\dfrac{1}{s} \right), & s > 0,
\end{cases}
\quad
\zeta(s) =
\begin{cases}
0, & s \leq 0, \\
\dfrac{\chi(s)}{\chi(s) + \chi(1-s)}, & 0 < s < 1, \\
1, & s \geq 1.
\end{cases}
\quad
\Phi(s) = s\zeta(s).
\end{equation*}
Note that $\Phi\in C^{\infty}(\mathbb{R})$.
Define $w_{ij}^{k}(x,t) = \eta_{k}\Phi\left( \frac{v_{ij}^{k}(x,t)}{\eta_{k}} \right)$. It is easy to check that
$w_{ij}^{k} =0 $ if $v_{ij}^{k} \leq 0$, $w_{ij}^{k} = v_{ij}^{k}$ if $v_{ij}^{k} \geq \eta_{k}$, and $w_{ij}^{k} \leq v_{ij}^{k}$ if $0\leq v_{ij}^{k} \leq \eta_{k}$. Consequently, $w_{ij}^{k} \leq \{v_{ij}^{k}\}_{+}$, $w_{ij}^{k} \in C^{\infty}(\bar{\Omega}\times \mathbb{R})$, $w_{ij}^{k}(x,t)$ is $1-$periodic in $t$, and $\| w_{ij}^{k} - \{v_{ij}^{k}\}_{+}\|_{\mathcal{X}} \leq \delta_{ij}^{k}$. This together with \eqref{2-7}, \eqref{2-8} and \eqref{2-9} implies that
\begin{equation*}
0 \leq w_{ij}^{k} \leq a_{ij}
\enspace \text{and} \enspace
\|  w_{ij}^{k} - a_{ij}     \|_{\mathcal{X}} \leq 6\delta_{ij}^{k}.
\end{equation*}
Moreover, for any $1 \leq i \leq l$, define $w_{ii}^{k} = b_{ii}^{k} -\delta_{ii}^{k}$.
Let
$A_{-}^{k}(x,t) = (w_{ij}^{k}(x,t))_{l\times l}$.

On the other hand, for any $1\leq i,j \leq l$, define $q_{ij}^{k} = b_{ij}^{k} + \delta_{ij}^{k}$. Let $A_{+}^{k}(x,t) = (q_{ij}^{k}(x,t))_{l\times l}$.
We can easily verify that $A_{-}^{k}$ and $A_{+}^{k}$ satisfy (H2), $A_{-}^{k} \leq A \leq A_{+}^{k}$, $\lim_{k\to\infty} \| A_{-}^{k} - A \|_{\infty} =0$, and
$\lim_{k\to\infty} \| A_{+}^{k} - A \|_{\infty} =0$.

{\bf Step 3.} In this step, we will construct lower and upper approximating sequences of matrices, $\tilde{A}_{-}^{k}$ and $\tilde{A}_{+}^{k}$.
The elements of these sequences are smooth, and the corresponding operators $\mathcal{L}_{-}^{k}$ and  $\mathcal{L}_{+}^{k}$ admit the principal eigenvalue.

Let $\lambda_{A_{-}^{k}}(x)$ be determined in Lemma \ref{LEM1.2} with
$A(x,t)$ replaced by $A_{-}^{k}(x,t)$. By the continuity of $\lambda_{A_{-}^{k}}(x)$, there exists $x_{-}^{k} \in \bar{\Omega}$ such that
$\lambda_{A_{-}^{k}}(x_{-}^{k}) = \max_{x\in\bar{\Omega}} \lambda_{A_{-}^{k}}(x)$, and for any $\varepsilon_{k}>0$, there exists $r_{k}>0$ such that
\begin{equation*}
\lambda_{A_{-}^{k}}(x) \geq \max_{x\in\bar{\Omega}} \lambda_{A_{-}^{k}}(x) - \varepsilon_{k}
\enspace \text{for all} \enspace
x\in B_{r_{k}}(x_{-}^{k}) \cap \Omega.
\end{equation*}
Next, we choose a smooth cut-off function $\rho_{k}$ satisfying $0 \leq \rho_{k} \leq 1$, $\rho_{k}(x) = 1$ for $| x - x_{-}^{k} | \leq \frac{r_{k}}{2}$, $\rho_{k}(x) = 0$ for $| x - x_{-}^{k} | \geq r_{k} $.
Define
\begin{equation*}
\lambda_{-}^{k}(x) = \left( \max_{x\in\bar{\Omega}} \lambda_{A_{-}^{k}}(x) - \varepsilon_{k} \right)\rho_{k}(x) + \left( \lambda_{A_{-}^{k}}(x) - 2\varepsilon_{k} \right)\left( 1 - \rho_{k}(x) \right).
\end{equation*}
It is easy to check that $\lambda_{-}^{k}(x) = \max_{x\in\bar{\Omega}} \lambda_{A_{-}^{k}}(x) - \varepsilon_{k}$ for $ |x-x^{k}_{-}|\leq \frac{r_{k}}{2}$,
$\lambda_{-}^{k}(x) \leq \max_{x\in\bar{\Omega}} \lambda_{A_{-}^{k}}(x) - \varepsilon_{k}$ for $ |x-x^{k}_{-}| > \frac{r_{k}}{2}$,
$\lambda_{-}^{k} \leq \lambda_{A_{-}^{k}}$, and $\lambda_{-}^{k}$ is a smooth function.
Let $\beta_{k} = \lambda_{-}^{k} - \lambda_{A_{-}^{k}}$ and $\tilde{A}_{-}^{k}(x,t) = A_{-}^{k}(x,t) + \beta_{k}(x) I$.
Clearly, $\tilde{A}_{-}^{k}(x,t)$ is smooth in $(x,t)\in \bar{\Omega}\times \mathbb{R}$ and
$\tilde{A}_{-}^{k} \leq A$.
Define operator $\mathcal{L}_{-}^{k}$ by
\begin{equation*}
\mathcal{L}_{-}^{k}[\varphi](x,t) = -\tau\varphi_{t} + D(x,t)\mathcal{P}[\varphi](x,t) +  \tilde{A}_{-}^{k}(x,t)\varphi(x,t).
\end{equation*}
Let $\lambda_{\tilde{A}_{-}^{k}}(x)$ be determined in Lemma \ref{LEM1.2} with
$A(x,t)$ replaced by $\tilde{A}_{-}^{k}(x,t)$.
Note that $\lambda_{\tilde{A}_{-}^{k}}(x) =  \lambda_{-}^{k}(x)$ and $\lambda_{-}^{k}(x)$ satisfies
\begin{equation*}
\frac{1}{ \max_{x\in\bar{\Omega}}\lambda_{-}^{k}(x)- \lambda_{-}^{k}(x)}\not\in L^{1}(\Omega_{0})
\enspace\text{for some}\enspace \Omega_{0}\subset\Omega.
\end{equation*}
Hence, by Theorem \ref{THM1.4}, $s(\mathcal{L}_{-}^{k})$ is the principal eigenvalue of $\mathcal{L}_{-}^{k}$.
Since $\tilde{A}_{-}^{k} \leq A$, we can obtain   $s(\mathcal{L}_{-}^{k}) \leq s(\mathcal{L})$.
Moreover, we can choose $\varepsilon_{k}$ and $r_{k}$ satisfying $\varepsilon_{k} \to 0, r_{k}\to 0$ as $k\to\infty$. Hence $\| \tilde{A}_{-}^{k}  - A\|_{\infty} \to 0$ as $k\to\infty$.
Following similar arguments as in \cite[Theorem A.1]{ZL2}, we can prove $s(\mathcal{L})$ is continuity with respect to $A$.
Finally, we get $\lim_{k\to\infty} s(\mathcal{L}_{-}^{k}) = s(\mathcal{L})$.
The construction of $\tilde{A}_{+}^{k}$ is similar, and we omit the details.
\end{proof}

\subsection{Properties of the principal spectrum point}
\indent

We first establish several properties of the generalized principal eigenvalue $\lambda_{p}(\mathcal{L})$. To highlight
the dependence on $\Omega, k_{i}, A$, we write $\mathcal{L}$ as $\mathcal{L}_{\Omega, k_{i}, A}$.
Following the same arguments as in \cite[Proposition 1.1]{Co10}, we can obtain the following result.

\begin{proposition}\label{PP2.11}
Let $\tilde\Omega$ be a subdomain of $\mathbb{R}^{N}$,  $\tilde{A}(x,t)= (\tilde{a}_{ij}(x,t))_{l\times l}$ and $\tilde{k}_{i}$ satisfy the same assumption of $A(x,t)$ and $k_{i}$ as in (H2) and (H3), respectively.
Then the following conclusions hold.
\begin{itemize}[leftmargin=0.9cm]

\item[(i)]
If $\tilde\Omega \subset \Omega$, then we have
\begin{equation*}
\lambda_{p}(\mathcal{L}_{\Omega, k_{i}, A}) \geq \lambda_{p}(\mathcal{L}_{\tilde\Omega, k_{i}, A}).
\end{equation*}

\item[(ii)]
If $\tilde{k}_{i}$ satisfies the same condition as $k_{i}$ and $ \tilde{k}_{i}\leq k_{i}$ for all $1\leq i \leq l$, then
\begin{equation*}
\lambda_{p}(\mathcal{L}_{\Omega, k_{i}, A})
\geq
\lambda_{p}(\mathcal{L}_{\Omega, \tilde{k}_{i}, A}).
\end{equation*}

\item[(iii)]
If $\tilde{A} \leq A $ $(i.e., \tilde{a}_{ij}\leq a_{ij})$, then we have
\begin{equation*}
\lambda_{p}(\mathcal{L}_{\Omega, k_{i}, A})
\geq
\lambda_{p}(\mathcal{L}_{\Omega, k_{i}, \tilde{A}}).
\end{equation*}

\end{itemize}
\end{proposition}
The conclusions (i)-(iii) of Proposition \ref{PP2.11} are still true for $\lambda_{p}^{\prime}(\mathcal{L}_{\Omega, {k}_{i}, A})$.

\begin{lemma}\label{SEMI}
$\lambda_{p}(\mathcal{L}_{\Omega, k_{i}, A})$ and $\lambda_{p}^{\prime}(\mathcal{L}_{\Omega, k_{i}, A})$ are lower semicontinuous and upper semicontinuous in $A$, respectively.
More precisely, for any $\varepsilon>0$, there exists $\delta>0$ such that for any matrix-valued
function $\tilde{A}$ satisfying $\| \tilde{A}-A\|_{L^{\infty}(\Omega\times[0,1])}\le \delta$, one has
\begin{equation*}
\lambda_{p}(\mathcal{L}_{\Omega,k_i,\tilde{A} })
\geq
\lambda_{p}(\mathcal{L}_{\Omega,k_i,A})-\varepsilon
\enspace \text{and} \enspace
\lambda_{p}^{\prime}(\mathcal{L}_{\Omega,k_i,\tilde{A} })
\leq
\lambda_{p}^{\prime}(\mathcal{L}_{\Omega,k_i,A})+\varepsilon.
\end{equation*}
\end{lemma}
\begin{proof}
Fix $\varepsilon>0$. Set $\lambda = \lambda_{p}(\mathcal{L}_{\Omega,k_i,A}) - \frac{\varepsilon}{2}$.
By the definition of $\lambda_{p}(\mathcal{L}_{\Omega, k_{i}, A})$, there exists $\varphi \in \mathcal{X}_{1}^{++}$ such that
\begin{equation}\label{2-10}
\tau \partial_{t} \varphi(x,t) - D(x,t)\mathcal{P}[\varphi](x,t) - A(x,t)\varphi(x,t) + \lambda\varphi(x,t) \leq 0
\enspace \text{for all} \enspace (x,t)\in\bar{\Omega}\times \mathbb{R}.
\end{equation}
Define
\begin{equation*}
m_{0} = \min_{1\leq i \leq l} \min_{(x,t)\in\bar{\Omega}\times [0,1]} \varphi_{i}(x,t)
\enspace \text{and} \enspace
M_{0} = \max_{(x,t)\in\bar{\Omega}\times[0,1]} \sum_{i=1}^{l} \varphi_{i}(x,t).
\end{equation*}
It follows that $\tilde{A}\varphi - A\varphi \geq -\frac{\delta M_{0}}{m_{0}}\varphi$ in $\bar{\Omega}\times \mathbb{R}$. This together with \eqref{2-10} implies that
\begin{equation*}
\tau \partial_{t} \varphi(x,t) - D(x,t)\mathcal{P}[\varphi](x,t) - \tilde{A}(x,t)\varphi(x,t) - \frac{\delta M_{0}}{m_{0}}\varphi + \lambda\varphi(x,t) \leq 0
\enspace \text{for all} \enspace (x,t)\in\bar{\Omega}\times \mathbb{R}.
\end{equation*}
By the definition of $\lambda_{p}(\mathcal{L}_{\Omega,k_i,\tilde{A} })$, we get $\lambda_{p}(\mathcal{L}_{\Omega,k_i,\tilde{A} }) \geq \lambda - \frac{\delta M_{0}}{m_{0}}$. Choosing $\delta = \frac{\varepsilon m_{0}}{2M_{0}}$, we get $\lambda_{p}(\mathcal{L}_{\Omega,k_{i},A})$ is lower semicontinuous in $A$. The upper semicontinuous of $\lambda_{p}^{\prime}(\mathcal{L}_{\Omega,k_{i},A})$ can be proved similarly.
\end{proof}
\begin{proof}[Proof of Theorem \ref{THM1.6}]
If $s(\mathcal{L}_{\Omega,k_{i},A})$ is the principal eigenvalue of $\mathcal{L}_{\Omega,k_{i},A}$, such result can be obtained by a similar argument as in \cite[Proposition 3.5]{Feng}.
If $s(\mathcal{L}_{\Omega,k_{i},A})$ is not the principal eigenvalue,
using Theorem \ref{THM1.5}, there exists a sequence of matrix-valued functions $A^{k}_{-}$ such that $A^{k}_{-} \leq A$, $A^{k}_{-} \to A$ as $k\to\infty$ and $s(\mathcal{L}_{\Omega,k_{i},A})$ is the principal eigenvalue. Hence,
\begin{equation}\label{2-11}
s(\mathcal{L}_{\Omega,k_{i},A^{k}_{-}})
=\lambda_{p}(\mathcal{L}_{\Omega,k_{i},A^{k}_{-}}).
\end{equation}
Moreover, using the monotonicity of $\lambda_{p}(\mathcal{L}_{\Omega,k_{i},A})$ in $A$ (see Proposition \ref{PP2.11} (iii)) and the lower semicontinuous in $A$ (see Lemma \ref{SEMI}), we have $\lim_{k\to\infty} \lambda_{p}(\mathcal{L}_{\Omega,k_{i},A^{k}_{-}}) = \lambda_{p}(\mathcal{L}_{\Omega,k_{i},A})$.
Furthermore, similar to the proof of \cite[Theorem A.1]{ZL2}, $s(\mathcal{L}_{\Omega,k_{i},A})$ is continuous in $A$. Passing $k\to\infty$ in \eqref{2-11}, we get
\begin{equation*}
s(\mathcal{L}_{\Omega,k_{i},A})
=\lambda_{p}(\mathcal{L}_{\Omega,k_{i},A}).
\end{equation*}

Similarly, choosing approximation sequence $A_{+}^{k}$ such that $A^{k}_{+}\geq A$ and $A^{k}_{+} \to A$ as $k\to\infty$, and using the monotonicity of $\lambda_{p}^{\prime}(\mathcal{L}_{\Omega,k_{i},A})$ in $A$ (see Proposition \ref{PP2.11} (iii)) and the upper semicontinuous in $A$ (see Lemma \ref{SEMI}), we can obtain
\begin{equation*}
s(\mathcal{L}_{\Omega,k_{i},A})
=\lambda_{p}^{\prime}(\mathcal{L}_{\Omega,k_{i},A}).
\end{equation*}
This completes the proof.
\end{proof}

By Lemma \ref{SEMI} and Theorem \ref{THM1.6}, we have the following conclusion.
\begin{corollary}
$\lambda_{p}(\mathcal{L}_{\Omega, k_{i}, A})$ and $\lambda_{p}^{\prime}(\mathcal{L}_{\Omega, k_{i}, A})$ are continuous in $A$.
More precisely, for any $\varepsilon>0$, there exists $\delta>0$, such that for any matrix-valued
function $\tilde{A}$ satisfying $\| \tilde{A}-A\|_{L^{\infty}(\Omega\times[0,1])}\le \delta$, one has
\begin{equation*}
| \lambda_{p}(\mathcal{L}_{\Omega,k_i,\tilde{A} }) - \lambda_{p}(\mathcal{L}_{\Omega,k_i,A})|
\leq
\varepsilon
\enspace \text{and} \enspace
|\lambda_{p}^{\prime}(\mathcal{L}_{\Omega,k_i,\tilde{A} })
-
\lambda_{p}^{\prime}(\mathcal{L}_{\Omega,k_i,A})|
\leq
\varepsilon.
\end{equation*}
\end{corollary}

\begin{theorem}
Assume that $\Omega_{1}\subset\Omega$ and $s(\mathcal{L}_{\Omega,k_{i},A})$ is the principal eigenvalue, then
\begin{equation*}
| s(\mathcal{L}_{\Omega_{1}, k_{i}, A}) -  s(\mathcal{L}_{\Omega, k_{i}, A})| \leq C|\Omega \backslash \Omega_{1}|,
\end{equation*}
where $|\Omega \backslash \Omega_{1}|$  is the Lebesgue measure of $\Omega \backslash \Omega_{1}$ and $C$ is a positive constant depending on $d_{ij}, k_{i}$ and $\Omega$.
\end{theorem}
\begin{proof}
Let $\varphi$ be the positive principal eigenfunction corresponding to $s(\mathcal{L}_{\Omega,k_{i},A})$ with the normalization  $\max_{(x,t)\in\bar{\Omega}\times\mathbb{R}} \sum_{i=1}^{l} \varphi_{i}(x,t) =1$.
Let
$$\bar{d}=\max_{1\leq i,j\leq l} \|d_{ij} \|_{L^{\infty}(\Omega\times (0,1))} \text{ and } \bar{k}=\max_{1\leq i\leq l} \|k_{i} \|_{L^{\infty}(\Omega\times\Omega\times (0,1))}.$$
A direct computation gives that
\begin{align*}
&\tau\partial_{t}\varphi_{i}(x,t)-\sum_{j=1}^{l}d_{ij}(x,t)\int_{\Omega_{1}}k_{j}(x,y,t)\varphi_{j}(y,t)dy
-\sum_{j=1}^{l}a_{ij}(x,t)\varphi_{j}(x,t) + s(\mathcal{L}_{\Omega,k_{i},A})\varphi_{i}(x,t)
\\
=& \sum_{j=1}^{l}d_{ij}(x,t)\int_{\Omega \backslash \Omega_{1}}k_{j}(x,y,t)\varphi_{j}(y,t)dy
\\
\leq&  \frac{  \bar{d} \bar{k} |\Omega\backslash\Omega_{1}|}{\min_{ i } \min_{\bar{\Omega}\times[0,1]} \varphi_{i} }\varphi_{i}(x,t)
\end{align*}
for all $(x,t)\in\bar{\Omega}_{1} \times \mathbb{R}$ and $1\leq i \leq l$.
It then follows from the definition of $\lambda_{p}(\mathcal{L}_{\Omega_{1}, k_{i}, A})$
that
\begin{equation*}
\lambda_{p}(\mathcal{L}_{\Omega_{1}, k_{i}, A})
\geq
s(\mathcal{L}_{\Omega,k_{i},A}) - C|\Omega\backslash\Omega_{1}|,
\end{equation*}
where $C= \frac{ \bar{d} \bar{k} }{\min_{ i } \min_{\bar{\Omega}\times[0,1]} \varphi_{i}}$.
Using Theorem \ref{THM1.6} and Proposition \ref{PP2.11}, we get the desired conclusion.
\end{proof}

\section{Asymptotic behavior of the principal spectrum point}

\subsection{Asymptotic behavior with respect to dispersal rate}
\indent

For simplicity, we write $L_{\tau,D,\sigma,m} = L_{D}$.
\begin{proof}[Proof of Theorem \ref{THM1.7}]
We give the main proof by the following two steps.

\textbf{Step 1.} We first prove that
\begin{equation*}
\lim_{ \bar{D} \to 0^{+}} s(L_{D}) = \max_{x \in \bar{\Omega}} \lambda_{A}(x).
\end{equation*}
By Lemma \ref{LEM2.6} and \cite[Lemma 2.2]{LZZ1}, we have
\begin{equation*}
s(L_{D}) \geq \max_{x \in \bar{\Omega}} \lambda_{A-D_{0}}(x).
\end{equation*}
It then follows that
\begin{equation*}
\liminf_{\bar{D} \to 0^{+}} s(L_{D}) \geq \max_{x \in \bar{\Omega}} \lambda_{A}(x).
\end{equation*}
It remains to prove that
\begin{equation*}
\limsup_{\bar{D} \to 0^{+}} s(L_{D}) \leq \max_{x \in \bar{\Omega}} \lambda_{A}(x).
\end{equation*}
To highlight the dependence on $A$, we write $L_{D}$ as $L_{D}(A)$.
Consider the following perturbed
matrix of $A(x,t)$.
\begin{equation*}
A_{k}(x,t) = \left( a_{ij}(x,t) + \frac{1}{k} \right)_{l\times l}.
\end{equation*}
By \cite[Theorem 1.4]{baihe}, there exists $\phi(x,t)\in\mathcal{X}_{1}^{++}$ such that
\begin{equation*}
-\tau\partial_{t}\phi(x,t)+A_{k}(x,t)\phi(x,t)=\lambda_{A_{k}}(x)\phi(x,t)
\enspace \text{for all} \enspace (x,t)\in\bar{\Omega}\times \mathbb{R}.
\end{equation*}
Let $0<\varepsilon \ll 1$. It then follows that there exists $d_{0}>0$ such that for $0< \bar{D} <d_{0}$,
\begin{equation*}
-L_{D}(A_{k})[\phi](x,t) + \left( \sup_{x\in\bar{\Omega}}\lambda_{A_{k}}(x) + \varepsilon \right) \phi(x,t)
\geq \varepsilon\phi(x,t)-D(x,t)\mathcal{P}[\phi](x,t) \geq 0
\enspace \text{for all} \enspace (x,t)\in\bar{\Omega}\times \mathbb{R}.
\end{equation*}
Hence, by the definition of $\lambda_{p}^{\prime}(L_{D}(A_{k}))$, we get
\begin{equation*}
\lambda_{p}^{\prime}(L_{D}(A_{k}))
\leq
\max_{x\in\bar{\Omega}}\lambda_{A_{k}}(x) + \varepsilon.
\end{equation*}
Since $\varepsilon>0$ is arbitrary,
\begin{equation*}
\limsup_{\bar{D} \to 0^{+}} \lambda_{p}^{\prime}(L_{D}(A_{k})) \leq \max_{x \in \bar{\Omega}} \lambda_{A_{k}}(x).
\end{equation*}
Using Proposition \ref{PP2.11} (iii), there holds
\begin{equation*}
\limsup_{\bar{D} \to 0^{+}} \lambda_{p}^{\prime}(L_{D}(A))
\leq
\limsup_{\bar{D} \to 0^{+}} \lambda_{p}^{\prime}(L_{D}(A_{k}))
\leq \max_{x \in \bar{\Omega}} \lambda_{A_{k}}(x).
\end{equation*}
Finally, letting $k \to \infty$ and using Theorem \ref{THM1.6}, we get the desired conclusion.

\textbf{Step 2.} In this step, under assumption {\rm (W)} holds and $k_{i} (1\leq i \leq l)$ are independent of $t$, we prove the following result.
\begin{equation}\label{3-1}
\lim_{ \underline{D} \to +\infty} s(L_{D}) = -\infty.
\end{equation}
By \cite[Theorem 2.1 and Proposition 3.4]{SX1}, there exist $\lambda_{i}<0$ and $\varphi_{i}\in X^{++}$ such that
\begin{equation*}
\int_{\Omega} k_{i,\sigma}(x-y)\varphi_{i}(y)dy-\varphi_{i}(x)
=
\lambda_{i}\varphi_{i}(x)
\enspace\text{in}\enspace \bar\Omega.
\end{equation*}
for all $1\leq i \leq l$.
Normalize  $\varphi_{i}$ by $\max_{\bar{\Omega}}\varphi_{i}=1$.
Recall that $D=CD_{0}$.
Define
\begin{equation*}
c_{0} = \min_{1 \leq i\leq l} c_{ii},
\enspace
\bar{a}=\max_{1 \leq i,j\leq l} \| a_{ij} \|_{L^{\infty}(\bar{\Omega}\times \mathbb{R})},
\enspace
\bar\lambda=\max_{ 1\leq i\leq l} \lambda_{i},
\enspace
\theta = \frac{1}{\min_{ 1 \leq i\leq l,x\in \bar{\Omega}} \varphi_{i}(x)}
\end{equation*}
and
$\varphi=(\varphi_{1},\varphi_{2},\cdots,\varphi_{l})$.
Let $\lambda = \bar\lambda c_{0} \underline{D} +  \bar{a} l \theta $.
Then, by direct computation, we have
\begin{align*}
&-c_{ii}d_{i}(x,t) \int_{\Omega} k_{i,\sigma}(x-y)\varphi_{i}(y)dy + c_{ii}d_{i}(x,t)\varphi_{i}(x)-\sum_{j=1}^{l}a_{ij}(x,t)\varphi_{j}(x) + \lambda \varphi_{i}(x)
\\
\geq& -\lambda_{i} c_{ii}d_{i}(x,t) \varphi_{i}(x)  - \bar{a}l + \lambda \varphi_{i}(x)
\geq \left(  -\bar\lambda c_{0}\underline{D} -  \bar{a} l \theta \right)\varphi_{i}(x) + \lambda \varphi_{i}(x)
\geq 0
\end{align*}
for all $(x,t)\in \bar{\Omega}\times \mathbb{R}$ and $1\leq i \leq l$.
This implies that
\begin{equation*}
-L_{D}[\varphi](x,t)+\lambda\varphi(x) \geq 0
\enspace \text{in} \enspace \bar{\Omega}\times \mathbb{R}.
\end{equation*}
It then follows from the definition of $\lambda_{p}^{\prime}(L_{D})$ that
\begin{equation*}
\lambda_{p}^{\prime}(L_{D}) \leq \bar\lambda c_{0} \underline{D} +  \bar{a} l \theta.
\end{equation*}
Letting $\underline{D} \to +\infty$ and using Theorem \ref{THM1.6} and  Remark \ref{REMARK1}, we arrive at \eqref{3-1}.
\end{proof}

\subsection{Asymptotic behavior with respect to dispersal range}
\indent

For simplicity, we write $L_{\tau,D,\sigma,m} = L_{\sigma,m}$.
\begin{proof}[Proof of Theorem \ref{THM1.8} (i)]
Let $\varphi\in \mathcal{X}^{++}$. Note that for any $m>0$,
\begin{equation*}
\lim_{\sigma \to +\infty}
\left\| \frac{1}{\sigma^{m}} \left( \sum_{j=1}^{l}  c_{ij}d_{j}(x,t) \int_{\Omega} k_{j,\sigma}(x-y,t)\varphi_{j}(y,t)dy - \varphi_{i}(x,t) \right) \right\|_{L^{\infty}(\bar{\Omega}\times \mathbb{R})} = 0
\enspace \text{for} \enspace
1 \leq i\leq l.
\end{equation*}
Moreover,
\begin{equation*}
\lim_{\sigma \to +\infty}
\left\|
\sum_{j=1}^{l}  c_{ij}d_{j}(x,t) \int_{\Omega} k_{i,\sigma}(x-y,t)\varphi_{i}(y,t)dy
\right\|_{L^{\infty}(\bar{\Omega}\times \mathbb{R})}  =0
\enspace \text{for} \enspace
1 \leq i\leq l.
\end{equation*}
Hence, the desired conclusion can be obtained by the same arguments as in the proof of Theorem \ref{THM1.7}, and we omit the details.
\end{proof}

\begin{proof}[Proof of Theorem \ref{THM1.8} (ii) $\mathit{m\in [0,2)}$]
By Theorem \ref{THM1.5}, there exist two sequences of smooth matrix-valued functions $\{A^{k}_{-}\}$ and $\{A^{k}_{+}\}$ such that
\begin{equation}\label{3-2}
s(L_{\sigma,m}(A^{k}_{-}))
\leq
s(L_{\sigma,m}(A))
\leq
s(L_{\sigma,m}(A^{k}_{+})),
\end{equation}
and $s(L_{\sigma,m}(A^{k}_{-}))$ and $s(L_{\sigma,m}(A^{k}_{+}))$ are the principal eigenvalues of $L_{\sigma,m}(A^{k}_{-})$ and $L_{\sigma,m}(A^{k}_{+})$, respectively.
Note that all entries of the matrix-valued functions  $A^{k}_{-}$ and $A^{k}_{+}$ are smooth.
Consequently, $\varphi^{k}_{-}$ and $\varphi^{k}_{+}$ are smooth as well, where $\varphi^{k}_{-}$ and $\varphi^{k}_{+}$ are defined as the principal eigenfunctions obtained by substituting $A^{k}_{-}$ and $A^{k}_{+}$ for $A$ in Lemma \ref{LEM1.2}.
Then the proof proceeds similarly to that of Shen and Vo \cite[Theorem D]{ShenVo}.
In fact, by using the test functions $\varphi^{k}_{-}$ and $\varphi^{k}_{+}$
together with a Taylor expansion and the monotonicity of the eigenvalue with respect to the domain, we obtain that
\begin{equation*}
\lim_{ \sigma \to 0^{+} } s(L_{\sigma,m}(A^{k}_{-})) =
\max_{x\in\bar{\Omega}} \lambda_{A^{k}_{-}}(x)
\enspace \text{and} \enspace
\lim_{ \sigma \to 0^{+} } s(L_{\sigma,m}(A^{k}_{+})) =
\max_{x\in\bar{\Omega}} \lambda_{A^{k}_{+}}(x).
\end{equation*}
This together with \eqref{3-2} implies that
\begin{align*}
\max_{x\in\bar{\Omega}} \lambda_{A^{k}_{-}}(x) & =
\lim_{ \sigma \to 0^{+} } s\left(L_{\sigma,m}(A^{k}_{-})\right)
\leq
\liminf_{ \sigma \to 0^{+} } s\left(L_{\sigma,m}(A)\right) \\
& \leq
\limsup_{ \sigma \to 0^{+} } s\left(L_{\sigma,m}(A)\right)
\leq
\lim_{ \sigma \to 0^{+} } s\left(L_{\sigma,m}(A^{k}_{+})\right) =
\max_{x\in\bar{\Omega}} \lambda_{A^{k}_{+}}(x).
\end{align*}
Finally, letting $k \to \infty$ in the above equation, we get the desired result.

\end{proof}

Next, we prove Theorem \ref{THM1.8} (ii) for the case $m=2$.
Under the given assumptions, we can represent $L_{\sigma,2}$ as
\begin{equation*}
L_{\sigma,2}[\varphi](x,t) = -\tau\varphi_{t}(x,t) + D_{1}(x,t)\mathcal{K}_{\Omega,\sigma,2}[\varphi](x,t) + A(x,t)\varphi(x,t),
\end{equation*}
where $ D_{1}(x,t) = \text{diag}\left( d_{1}(x,t),d_{2}(x,t),\cdots d_{l}(x,t)  \right)$, $\mathcal{K}_{\Omega,\sigma,2}[\varphi]=
\left( \mathcal{K}^{1}_{\Omega,\sigma,2}[\varphi_{1}], \cdots, \mathcal{K}^{l}_{\Omega,\sigma,2}[\varphi_{l}] \right)^{T}$ and
\begin{equation*}
\mathcal{K}^{i}_{\Omega,\sigma,2}[\varphi_{i}](x,t)=
\frac{1}{\sigma^{2}} \left( \int_{\Omega} \frac{1}{\sigma^{N}}k_{i}\left(\frac{x-y}{\sigma},t \right)\varphi_{i}(y,t)dy - \varphi_{i}(x,t) \right).
\end{equation*}
Moreover, notice that ${ \rm (\tilde{H}2)}$ and $(D)$ imply that $A$ is irreducible.

We begin with some preliminary results.
The following proposition follows from \cite{baihe,Anton0}.
\begin{proposition}
The problem \eqref{1-2} admits a principal eigenvalue $\lambda^{local}$. Moreover, $\lambda^{local}$ is algebraically simple associated with a unique (up to a multiplicative constant) eigenfunction $\varphi^{r}\in \mathbf{C}^{2+\alpha,1+\frac{\alpha}{2}}(\bar\Omega\times\mathbb{R})$ whose components are all positive.
\end{proposition}

Define
\begin{equation*}
L_{r}[\varphi](x,t)=-\tau\varphi_{t}(x,t) + D_{r}(x,t)R[\varphi](x,t)+A(x,t)\varphi(x,t), \quad (x,t)\in\bar\Omega\times\mathbb{R}.
\end{equation*}
Consider the following initial-boundary value problem:
\begin{equation}\label{3-3}
\left\{
\begin{aligned}
&L_{r}[u](x,t)=0, && x\in\Omega, t> 0,\\
&u(x,t)=0,  && x\in\partial\Omega, t>0,\\
&u(x,0)=u_{0}(x),  && x\in\Omega.\\
\end{aligned}
\right.
\end{equation}
Set
\begin{equation*}
Y = \{u\in W^{2}_{p}(\Omega): u=0 \text{ on } \partial{\Omega}, p>N \}.
\end{equation*}
Let $A_{i}$ be the realization of operator $A_{i}(t)$ = $-d_{r,i}(t) \Delta  $ with Dirichlet boundary condition in $L^{p}(\Omega)$.
Note that the domain of $A_{i}$ is $Y$, which is independent of $t$.
For $\beta$ satisfying $\frac{1}{2} + \frac{N}{2p} < \beta \leq 1$, let $Y^{i}_{\beta}$ be the fractional power space with respect to $A_{i}(0)$.
Set $Y_{\beta} = \prod_{i=1}^{l} Y^{i}_{\beta}$.
Define operator family $V^{r}(t,0)$ on $Y_{\beta}$ by
\begin{equation*}
V^{r}(t,0)\varphi(x)=u(x,t;\varphi)
\enspace \text{for all} \enspace x\in\Omega,t\geq 0, \varphi\in Y_{\beta},
\end{equation*}
where $u(x,t;\varphi)$ is the unique solution at time $t$ of \eqref{3-3} with initial function $\varphi$.
Note that $Y_{\beta}$ is a strongly ordered Banach space with the positive
cone $Y^{+}_{\beta} = \{ u\in Y_{\beta} : u\ge 0 \}$.
Moreover, according to \cite[Remark 12.1]{Peter} and \cite[Lemma 5.5]{baihe}, the operator
$V^{r}(1,0) : Y_{\beta} \to Y_{\beta}$ is strongly positive and compact.
Then by the Krein-Rutman Theorem, $r(V^{r}(1,0))$ is an isolated algebraically simple eigenvalue of $V^{r}(1,0)$.
In addition, a similar argument as in \cite[Proposition 14.4]{Peter} yields
\begin{equation*}
r(V^{r}(1,0)) = e^{ \frac{\lambda^{local}}{\tau}}
\end{equation*}
and $\varphi^{r}_{0}=\varphi^{r}(\cdot,0)$ is the principal eigenfunction of $V^{r}(1,0)$.
Consequently, the following standard decomposition holds.
\begin{lemma}\label{LEM3.2}
There exists a codimension one subspace $Y^{1}_{\beta}$ of $Y_{\beta}$ such that
$$
Y_{\beta} = Y^{1}_{\beta} \oplus Y^{2}_{\beta},
$$
where $Y^{2}_{\beta} = \text{span}\{\varphi^{r}_{0}\}$, and there are $M>0$ and $\gamma>0$ such that for any $u\in Y^{1}_{\beta}$, there holds
\begin{equation*}
\frac{\|V^{r}(t,0)u\|_{Y_{\beta}}}
{\|V^{r}(t,0)\varphi^{r}_{0}\|_{Y_{\beta}}}
\leq Me^{-\gamma t}.
\end{equation*}
\end{lemma}

To highlight the dependence on the domain, we write $V^{r}$ as $V^{r}_{\Omega}$.
Define $$\Omega_{n} = \{ x\in\mathbb{R}^{N} \,|\, \text{dist}(x,\Omega) < \frac{1}{n} \}. $$ By \cite{Foote}, there exists $n_{0}>0$ such that for $n\geq n_{0}$, $\Omega_{n}$ has the same boundary regularity as $\Omega$.
Set $\Omega_{*} = \Omega_{n_{0}}$.
Throughout this section, we assume that the coefficients
$d_i(\cdot,\cdot)$ and $a_{ij}(\cdot,\cdot)$ are defined on $\bar{\Omega}_{*}\times\mathbb{R}$,
$1$-periodic in $t$, and satisfy the same regularity and positivity assumptions as on $\bar{\Omega}\times\mathbb{R}$.

For each $n\geq n_{0}$, let $V^{r}_{\Omega_{n}}$ denote the evolution operator associated of \eqref{3-3} with $\Omega$ replaced by $\Omega_n$, identified as a bounded operator on $Y_{\beta}(\Omega_{*})$ via restriction
to $\Omega_n$ and zero extension to $\Omega_{*}$.
Moreover, let $\lambda^{r}_{\Omega_{n}}$ be the principal eigenvalue of \eqref{1-2} with $\Omega$ replaced by $\Omega_{n}$.

\begin{lemma}\label{LEM3.3}
$\lim_{n\to\infty} \lambda^{r}_{\Omega_{n}} = \lambda^{local}.$
\end{lemma}
\begin{proof}
For each $n\geq n_{0}$, let $U^{r}_{\Omega_{n}}(t,s)$ be the evolution operator on $Y_{\beta}(\Omega_{*})$ generated by the dispersal part
with Dirichlet boundary condition on $\Omega_n$, namely the problem
\begin{equation*}
\tau u_t = D_r(t,\cdot)\,R[u]
\enspace \text{in} \enspace \Omega_n,
\quad u=0 \enspace\text{on}\enspace \partial\Omega_n.
\end{equation*}
It follows that
\begin{equation*}
V_{\Omega_n}^r(t,0) = U_{\Omega_n}^{r}(t,0) +  \frac{1}{\tau}\int_0^t U_{\Omega_n}^{r}(t,s) A(s,\cdot) V_{\Omega_n}^r(s,0) ds.
\end{equation*}
Set $W_{n}(t) = V_{\Omega_n}^r(t,0) - V_{\Omega}^r(t,0)$. A direct computation yields that
\begin{align*}
W_{n}(t) = &\left( U^{r}_{\Omega_n}(t,0)-U^{r}_{\Omega}(t,0) \right) + \frac{1}{\tau}\int_{0}^{t} \left( U^{r}_{\Omega_n}(t,s)-U^{r}_{\Omega}(t,s) \right)A(s,\cdot) V_{\Omega}^{r}(s,0)ds \\
&+ \frac{1}{\tau}\int_{0}^{t} U^{r}_{\Omega_n}(t,s) A(s,\cdot) W_{n}(s) ds.
\end{align*}
Since $A\in L^{\infty}(\Omega_{*} \times [0,1])$ and $\{U^{r}_{\Omega_n}(t,s)\}$ is uniformly bounded on $Y_{\beta}(\Omega_{*})$ for $0\le s\le t\le 1$,
there exist constants $C_{1},C_{2}>0$, independent of $n$, such that
\begin{equation*}
\sup_{0\leq s \leq t \leq 1} \| U_{\Omega_{n}}^{r}(t,s)A(s,\cdot) \|\leq C_{1}
\enspace \text{and} \enspace
\sup_{0\leq s \leq 1} \| A(s,\cdot)V_{\Omega}^{r}(s,0) \|\leq C_{2}.
\end{equation*}
Hence,
\begin{equation}\label{3-4}
\| W_{n}(1) \| \leq  \alpha_{n} + C_{1} \int_{0}^{1}  \| W_{n}(s)\| ds,
\end{equation}
where $\alpha_{n} = \| U^{r}_{\Omega_n}(1,0)-U^{r}_{\Omega}(1,0) \| +  C_{2}\int_{0}^{1}  \| U^{r}_{\Omega_n}(1,s)-U^{r}_{\Omega}(1,s) \|ds$. By \cite[Remark 4.5]{Dan} and the interpolation inequality \cite[Proposition 1.2.6]{lun}, we get $\alpha_{n}\to 0$ as $n\to\infty$. It then follows from \eqref{3-4} and Gronwall's inequality that
$\| W_{n}(1) \|\to 0$ as $n\to\infty$. This together with the perturbation result in \cite[Section IV. 3.5]{Kato} implies that
\begin{equation*}
r(V_{\Omega_{n}}^{r}(1,0)) \to r(V_{\Omega}(1,0))
\enspace \text{as} \enspace
n\to\infty.
\end{equation*}
The desired conclusion then follows from
$\tau \ln r(V^{r}_{\Omega_{n}}(1,0)) = \lambda^{r}_{\Omega_{n}}$.
\end{proof}

In what follows, we prove that the solution of the initial value problem of nonlocal dispersal systems converges to the solution of the corresponding local dispersal systems.
Consider the following initial Cboundary value problems:
\begin{equation}\label{3-5}
\left\{
\begin{aligned}
&\tau\partial_{t}u_{i}(x,t)= d_{r,i}(x,t)\Delta u_{i}(x,t) + f_{i}(x,t,u), && x\in\Omega, t> 0,\\
&u_{i}(x,t)=0,  && x\in\partial\Omega, t>0,\\
&u_{i}(x,0)=u_{i,0}(x),  && x\in \bar\Omega,\\
&i=1,2,\cdots,l.
\end{aligned}
\right.
\end{equation}
and
\begin{equation}\label{3-6}
\left\{
\begin{aligned}
&\tau\partial_{t}u_{i}(x,t)= d_{i}(x,t)\mathcal{K}^{i}_{\mathbb{R}^{N},\sigma,2}[u_{i}](x,t) + f_{i}(x,t,u), && x\in \bar\Omega, t> 0,\\
&u_{i}(x,t)=0,  &&  x\in\mathbb{R}^{N}\setminus \bar\Omega, t>0,\\
&u_{i}(x,0)=u_{i,0}(x),  && x\in \bar\Omega,\\
&i=1,2,\cdots,l.
\end{aligned}
\right.
\end{equation}
Here, $f_{i}(x,t,s)$ $(1 \leq i \leq l)$ satisfy
\begin{align*}
&f_{i}\in C(\bar{\Omega}\times[0,1]\times \mathbb{R}^{l}),
f_{i}(\cdot,\cdot,s) \in C^{\alpha,\frac{\alpha}{2}}(\bar{\Omega}\times [0,1]) \text{ uniformly for } s \text{ on bounded subsets of } \mathbb{R}^{l},\\
&\partial_{s}f_{i} \text{ exists and is continuous on }
\bar{\Omega}\times [0,1] \times \mathbb{R}^{l}, \text{ and }\partial_{s_{j}}f_{i}(x,t,s_{1},\cdots,s_{l}) \geq 0 \text{ if } i\neq j.
\end{align*}
\begin{theorem}\label{THM3.4}
Let $u \in C^{2+\alpha, 1+\frac{\alpha}{2}}(\bar{\Omega}\times [0,T])$ be the solution to \eqref{3-5} and let $u^{\sigma}$ be the solution to \eqref{3-6}.
Then there exists $C=C(T)>0$ such that
\begin{equation*}
\sup_{t\in[0,T]} \left|\left| u^{\sigma}(\cdot,t) - u(\cdot,t)\right|\right|_{\mathbf{L}^{\infty}(\Omega)} \leq C\sigma^{\alpha} \to 0,
\enspace \text{as} \enspace \sigma \to 0.
\end{equation*}
\end{theorem}
\begin{proof}
Without loss of generality, we assume that $k_{i}(\cdot,t)(1 \leq l \leq l)$ are compactly supported in $B_{1}(0)$.
Let $\tilde{u}$ be a $C^{2+\alpha,1+\frac{\alpha}{2}}$ extension of $u$ to $\mathbb{R}^{N}\times[0,T]$.
Define $\omega^{\sigma}= \tilde{u}- u^{\sigma}$.
A direct computation yields that
\begin{equation}\label{3-7}
\left\{
\begin{aligned}
&\tau\partial_t\omega^\sigma(x,t)
= D(x,t)\mathcal{K}_{\mathbb{R}^N,\sigma,2}[\omega^\sigma](x,t)
+ A^\sigma(x,t)\omega^\sigma(x,t) + b^\sigma(x,t),
&& x\in\bar\Omega,\ t\in[0,T],\\
&\omega^\sigma(x,t)=g(x,t), && x\in\mathbb{R}^N\setminus\bar\Omega,\ t\in(0,T],\\
&\omega^\sigma(x,0)=0, && x\in\bar\Omega,
\end{aligned}
\right.
\end{equation}
where $A^\sigma(x,t)=(a_{ij}^\sigma(x,t))_{l\times l}$ with
\begin{equation*}
a_{ij}^\sigma(x,t)=\int_0^1 \partial_{s_j}f_i\bigl(x,t,u^\sigma(x,t)+\theta(\tilde u(x,t)-u^\sigma(x,t))\bigr)\,d\theta,
\end{equation*}
$b^{\sigma}(x,t) = \left(b_{1}^{\sigma}(x,t), b_{2}^{\sigma}(x,t) , \cdots, b_{l}^{\sigma}(x,t) \right)$ with
\begin{equation*}
b_{i}^{\sigma}(x,t)=d_{r,i}(x,t)\Delta \tilde{u}_{i}(x,t) -d_{i}(x,t)\mathcal{K}^{i}_{\mathbb{R}^{N},\sigma,2}[\tilde{u}_{i}](x,t),
\end{equation*}
and
$g(x,t)=\tilde{u}(x,t)$ for $x\in \mathbb{R}^{N}\setminus\Omega $ and $t\in [0,T]$.
Since $k_{i}(\cdot,t) (1\leq i \leq l) $ are symmetric,
\begin{align*}
b_{i}^{\sigma}(x,t)
=&d_{r,i}(x,t)\Delta \tilde{u}_{i}(x,t) - \frac{d_{i}(x,t)}{\sigma^{2}}
\left( \int_{\mathbb{R}^{N}} \frac{1}{\sigma^{N}}k_{i}\left(\frac{x-y}{\sigma},t\right)\tilde{u}_{i}(y,t)dy-\tilde{u}_{i}(x,t) \right)
\\
=& d_{r,i}(x,t)\Delta \tilde{u}_{i}(x,t) - \frac{d_{i}(x,t)}{\sigma^{2}}\left(\int_{\mathbb{R}^{N}}k_{i}(z,t)\left(\tilde{u}_{i}(x+\sigma z,t)-\tilde{u}_{i}(x,t)\right)dz \right)
\\
=& d_{r,i}(x,t)\Delta \tilde{u}_{i}(x,t) -   \left(\frac{d_{i}(x,t)}{2N}\int_{\mathbb{R}^{N}}k_{i}(z,t)|z|^{2}dz\right) \Delta \tilde{u}_{i}(x,t)+O(\sigma^{\alpha})
\\
=&O(\sigma^{\alpha})
\end{align*}
for all $(x,t)\in\bar{\Omega}\times[0,T]$ and $1\leq i \leq l$.
Moreover, since $u(\cdot,t)=0$ on $\partial\Omega$ and $\tilde u$ is $C^{1}$ in space, we have
\begin{equation*}
\|g_{i}(\cdot,t)\|_{L^\infty(\{x\notin\Omega:\ \mathrm{dist}(x,\partial\Omega)\le\sigma\})} = O(\sigma),
\quad t\in[0,T].
\end{equation*}
Let $\sigma_0\in(0,1]$ and
$$A_{0}= \sup_{0<\sigma \leq \sigma_{0}} \max_{1\leq i \leq l} \max_{ (x,t)\in \bar{\Omega}\times [0,T] } \sum_{j=1}^{l} a^{\sigma}_{ij}(x,t).$$
Define
\begin{equation*}
\bar{\omega}= \left(K_{1}\sigma^{\alpha}te^{\frac{A_{0}t}{\tau}}+K_{2}\sigma \right)e \enspace\text{and}\enspace\underline{\omega}=
-\bar{\omega},
\end{equation*}
where $e= (1,\cdots,1)^{T}$.
Choosing $K_{1},K_{2}>0$ large enough and $\sigma>0$ small, one checks that $(\bar{\omega},\underline{\omega})$ is a pair of upper and lower solutions for \eqref{3-7}.
As a result, we get the desired conclusion.
\end{proof}

\begin{proof}[Proof of Theorem \ref{THM1.8} (ii) $\mathit{m = 2}$]
For simplicity, we let $\tau=1$.
We first prove that for any $\varepsilon>0$, there exists $\sigma_{1}>0$ such that for $0<\sigma<\sigma_{1}$,
\begin{equation*}
s(L_{\sigma,2}) \leq \lambda^{local} + \varepsilon.
\end{equation*}
Let us fix $\varepsilon>0$. By Lemma \ref{LEM3.3}, there exists a smooth domain $\Omega^{\prime}$ such that $\Omega\subset\Omega^{\prime}$ and
\begin{equation}\label{3-8}
\lambda^{r}_{\Omega^{\prime}} < \lambda^{local}+\frac{\varepsilon}{2}.
\end{equation}
Choose $n_{0}\in\mathbb{N}^{+}$ such that
\begin{equation}\label{3-9}
e^{-\frac{\varepsilon}{2}n_{0}}\leq \frac{1}{2}.
\end{equation}
Let $\phi^{r}_{\Omega^{\prime}}$ be the principal eigenfunction of $V_{\Omega^{\prime}}^{r}(1,0)$. Since $\min_{\bar{\Omega}}  \phi^{r}_{\Omega^{\prime}} >0$, there exists $\varepsilon_{0}>0$ such that
\begin{equation}\label{3-10}
\varepsilon_{0} \leq \frac{1}{2}e^{(\lambda^{local} + \varepsilon)n_{0}}\phi^{r}_{\Omega^{\prime}}(x)\enspace\text{for all}\enspace x\in\bar{\Omega}.
\end{equation}
Let $V^{\sigma}_{\Omega}(t,0)$ be the evolution operator and $ \left(V^{\sigma}_{\Omega}(t,0)u_{0}\right) (x)$ be the solution associated with the following problem.
\begin{equation*}
\left\{
\begin{aligned}
& \partial_{t} u(x,t)= D_{1}(x,t)\mathcal{K}_{\Omega,\sigma,2}[u](x,t) + A(x,t)u(x,t), && x\in \bar\Omega, t> 0,\\
&u(x,0)=u_{0}(x),  && x\in \bar\Omega.
\end{aligned}
\right.
\end{equation*}
By the comparison principle, we get
\begin{equation}\label{3-11}
(V^{\sigma}_{\Omega}(n_{0},0)\phi^{r}_{\Omega^{\prime}}|_{\bar{\Omega}})(x)
\leq
(V^{\sigma}_{\Omega^{\prime}}(n_{0},0)\phi^{r}_{\Omega^{\prime}})(x) \enspace\text{for all}\enspace x\in\bar\Omega.
\end{equation}
Moreover, we derive from Theorem \ref{THM3.4} that there exists $\sigma_{1}>0$ such that for $0<\sigma<\sigma_{1}$,
\begin{equation*}
(V^{\sigma}_{\Omega^{\prime}}(n_{0},0)\phi^{r}_{\Omega^{\prime}})(x) \leq
(V^{r}_{\Omega^{\prime}}(n_{0},0)\phi^{r}_{\Omega^{\prime}})(x) +
\varepsilon_{0} \enspace\text{for all}\enspace x\in\bar\Omega.
\end{equation*}
This together with \eqref{3-8}-\eqref{3-11} implies that
\begin{align*}
(V^{\sigma}_{\Omega}(n_{0},0)\phi^{r}_{\Omega^{\prime}}|_{\bar{\Omega}})(x)
\leq&
(V^{r}_{\Omega^{\prime}}(n_{0},0)\phi^{r}_{\Omega^{\prime}})(x) +
\varepsilon_{0}
\\
\leq& e^{\lambda^{r}_{\Omega^{\prime}}n_{0}}\phi^{r}_{\Omega^{\prime}}(x) + \frac{1}{2} e^{(\lambda^{local} + \varepsilon)n_{0}}\phi^{r}_{\Omega^{\prime}}(x)
\\
\leq& e^{(\lambda^{local}+\frac{\varepsilon}{2})n_{0}}\phi^{r}_{\Omega^{\prime}}(x) + \frac{1}{2}e^{(\lambda^{local} + \varepsilon)n_{0}}\phi^{r}_{\Omega^{\prime}}(x)
\\
=& e^{-\frac{\varepsilon}{2}n_{0}}e^{(\lambda^{local} + \varepsilon)n_{0}}\phi^{r}_{\Omega^{\prime}}(x) + \frac{1}{2}e^{(\lambda^{local} + \varepsilon)n_{0}}\phi^{r}_{\Omega^{\prime}}(x)
\\
\leq& e^{(\lambda^{local} + \varepsilon)n_{0}}\phi^{r}_{\Omega^{\prime}}(x)
\end{align*}
for all $x\in\bar{\Omega}$. Using Proposition \ref{PP2.10}, we get
\begin{equation*}
s(L_{\sigma,2}) \leq \lambda^{local} + \varepsilon.
\end{equation*}

On the other hand, following similar arguments as in the proof of \cite[Theorem B]{SX}, and Lemma \ref{LEM3.2}
we can prove that for any $\varepsilon>0$ there exists $\sigma_{2}>0$ such that for $0<\sigma<\sigma_{2}$
\begin{equation*}
s(L_{\sigma,2}) \geq \lambda^{local} - \varepsilon.
\end{equation*}
The details of this are omitted here.
\end{proof}

\subsection{Asymptotic behavior with respect to frequency}
\indent

In this section, we write $L_{\tau,D,\sigma,m}$ as $L_{\tau}$.
Denote the adjoint operator of $L_{\tau}$ by $L^{*}_{\tau}$, that is,
\begin{equation*}
L^{*}_{\tau}[u] = \tau u_{t}(x,t) + D\mathcal{P}[u](x,t) - D_{0}u(x,t) + A(x,t)u(x,t).
\end{equation*}
For $f,g\in  L^{2}\left( \Omega\times[0,1], \mathbb{R}^{l}\right)$, we set
$(f,g)_{0} = \sum_{i=1}^{l} \int_{0}^{1}\int_{\Omega} f(x,t)g(x,t)dxdt.$
The following result follows from an argument similar to that of \cite[Lemma 3.1]{Feng}.

\begin{lemma}
Suppose ${\rm (H2),(\tilde{H}3), (F1), (F2)}$ hold.
Assume that $\varphi$ and $\psi$ are the principal eigenfunctions of $s(L_{\tau})$ and $s(L^{*}_{\tau})$, respectively. Then the following statements hold:
\begin{itemize}[leftmargin=0.9cm]
\item[$(i)$]
Given $f\in\mathcal{X}$, there exists $u\in\mathcal{X}_{1}$ satisfying $L_{\tau}[u] - s(L_{\tau})u = f$ if and only if
$(f,\psi)_{0} = 0$.
Also there exists $v\in\mathcal{X}_{1}$ satisfying $L^{*}_{\tau}[v] - s(L_{\tau})v = f$ if and only if
$(f,\varphi)_{0} = 0$.

\item[$(ii)$]
$s(L_{\tau})$, $\varphi$ and $\psi$ are continuously differentiable with respect to $\tau$.
\end{itemize}
\end{lemma}

We first establish the monotonicity result (Theorem \ref{THM1.9}). Note that under assumptions ${\rm (\tilde{H}3) }$ and ${\rm (F2)}$, all kernel functions coincide and are of the form $k_{i}(x,y,t)=k_{\sigma}(x-y,t)$.

\begin{proof}[Proof of Theorem \ref{THM1.9}]
For clarity, the proof proceeds as follows in two steps.

\textbf{Step 1.} In this step, we prove that  $s(L_{\tau})$ is  nonincreasing in $\tau$.

By Theorem \ref{THM1.5}, it suffices to consider the case where $s(L_{\tau})$ is the principal eigenvalue; the general case then follows directly from the approximation argument.

Let $\varphi$ be the eigenfunction corresponding to $s(L_{\tau})$. That is,
\begin{equation}\label{3-12}
\tau \partial_{t}\varphi(x,t) =  D\mathcal{P}[\varphi](x,t) - D_{0}\varphi(x,t) + A(x,t)\varphi(x,t) - s(L_{\tau})\varphi(x,t) \enspace \text{in} \enspace\bar{\Omega}\times[0,1].
\end{equation}
Let $\psi$ be the eigenfunction of the adjoint problem  to \eqref{3-12}, namely,
\begin{equation*}
-\tau \partial_{t}\psi(x,t) =  D\mathcal{P}[\psi](x,t)-D_{0}\psi(x,t) + A(x,t)\psi(x,t) - s(L_{\tau})\psi(x,t) \enspace \text{in} \enspace\bar{\Omega}\times[0,1].
\end{equation*}
Define
\begin{equation*}
\alpha_{i}(x,t)=\sqrt{\varphi_{i}(x,t)\psi_{i}(x,t)},\quad \beta_{i}(x,t)=\frac{1}{2}\ln\left( \frac{\varphi_{i}(x,t)}{\psi_{i}(x,t)} \right)
\end{equation*}
and
\begin{equation*}
L_{i}[\phi](x,t)=\sum_{j=1}^{l}d_{ij}\int_{\Omega}k_{\sigma}(x-y,t)\phi_{j}(y,t)dy - d_{ii}\phi_{i}(x,t)
+\sum_{j=1}^{l}a_{ij}(x,t)\phi_{j}(x,t)-s(L_{\tau})\phi_{i}(x,t)
\end{equation*}
for $ i \in \mathbb{S}$ and $(x,t)\in \bar{\Omega}\times \mathbb{R}$.
It follows that
\begin{equation}\label{3-13}
\frac{L_{i}[\varphi](x,t)}{\varphi_{i}(x,t)}+\frac{L_{i}[\psi](x,t)}{\psi_{i}(x,t)}=2\tau \partial_{t}\beta_{i}(x,t).
\end{equation}
On the other hand, by direct calculation, we have
\begin{align}\label{3-14}
\nonumber
\frac{L_{i}[\varphi]}{\varphi_{i}}+\frac{L_{i}[\psi]}{\psi_{i}}
=&
\sum_{j=1}^{l} d_{ij}\int_{\Omega}k_{\sigma}(x-y,t) \left( \frac{\varphi_{j}(y,t)}{\varphi_{i}(x,t)} + \frac{\psi_{j}(y,t)}{\psi_{i}(x,t)} \right)dy
- 2d_{ii}
\\ \nonumber
&+\sum_{j=1}^{l} a_{ij}(x,t)\left( \frac{\varphi_{j}(x,t)}{\varphi_{i}(x,t)} + \frac{\psi_{j}(x,t)}{\psi_{i}(x,t)} \right) - 2s(L_{\tau})
\\
=& 2 \sum_{j=1}^{l} d_{ij}\int_{\Omega}k_{\sigma}(x-y,t) \frac{\alpha_{j}(y,t)}{\alpha_{i}(x,t)} dy
-
2d_{ii}
+
2\sum_{j=1}^{l} a_{ij}(x,t) \frac{\alpha_{j}(x,t)}{\alpha_{i}(x,t)} - 2s(L_{\tau})
\\ \nonumber
& + 4\sum_{j=1}^{l} d_{ij} \int_{\Omega} k_{\sigma}(x-y,t)\frac{\alpha_{j}(y,t)}{\alpha_{i}(x,t)}\sinh^{2}\left(\frac{\beta_{j}(y,t)-\beta_{i}(x,t)}{2}\right)dy
\\ \nonumber
&+ 4\sum_{j=1}^{l} a_{ij}(x,t) \frac{\alpha_{j}(x,t)}{\alpha_{i}(x,t)}\sinh^{2}\left(\frac{\beta_{j}(x,t)-\beta_{i}(x,t)}{2}\right).
\end{align}
Combining \eqref{3-13} and \eqref{3-14}, we obtain
\begin{equation}\label{3-15}
\begin{aligned}
&-s(L_{\tau})\alpha_{i}(x,t)\\
=&
-\sum_{j=1}^{l} d_{ij}\int_{\Omega}k_{\sigma}(x-y,t)\alpha_{j}(y,t)dy + d_{ii}\alpha_{i}(x,t) - \sum^{l}_{j=1}a_{ij}(x,t)\alpha_{j}(x,t) -c_{i}(x,t)\alpha_{i}(x,t),
\end{aligned}
\end{equation}
where
\begin{align*}
c_{i}(x,t)=
&-\tau\partial_{t}\beta_{i}(x,t) + \frac{1}{2}\sum_{j=1}^{l}d_{ij}\int_{\Omega}k_{\sigma}(x-y,t) \left( \sqrt{\frac{\varphi_{j}(y,t)}{\varphi_{i}(x,t)}} - \sqrt{\frac{\psi_{j}(y,t)}{\psi_{i}(x,t)}} \right)^{2}dy
\\
&+\frac{1}{2}\sum_{j=1}^{l}a_{ij}(x,t)\left( \sqrt{\frac{\varphi_{j}(x,t)}{\varphi_{i}(x,t)}}
- \sqrt{\frac{\psi_{j}(x,t)}{\psi_{i}(x,t)}} \right)^{2}.
\end{align*}
Differentiating \eqref{3-15} with respect to $\tau$, and then multiplying $\alpha_{i}(x,t)$, we get
\begin{align}\label{3-16}
& - s^{\prime}(L_{\tau})\alpha_{i}^{2}(x,t) - s(L_{\tau})\alpha_{i}^{\prime}(x,t)\alpha_{i}(x,t)
\\ \nonumber
=& \left(- d^{\prime}_{ii}\int_{\Omega}k_{\sigma}(x-y,t)\alpha_{i}(y,t)dy -\sum_{j=1}^{l} d_{ij}\int_{\Omega}k_{\sigma}(x-y,t)\alpha_{j}^{\prime}(y,t)dy + d_{ii}^{\prime}\alpha_{i}(x,t) + d_{ii}\alpha_{i}^{\prime}(x,t)  \right.
\\ 
& \left. -\sum^{l}_{j=1}a_{ij}(x,t)\alpha_{j}^{\prime}(x,t) - c^{\prime}_{i}(x,t)\alpha_{i}(x,t) - c_{i}(x,t)\alpha^{\prime}_{i}(x,t) \right)\alpha_{i}(x,t). \nonumber
\end{align}
Multiplying \eqref{3-15} by $\alpha_{i}^{\prime}(x,t)$, subtracting equation \eqref{3-16}, integrating over $\Omega\times (0,1)$, and finally sum over $i$ from $1$ to $l$  yields
\begin{align}\label{3-17}
& s^{\prime}(L_{\tau}) \sum_{i=1}^{l} \int_{0}^{1} \int_{\Omega} \alpha^{2}_{i}(x,t)dxdt
\\ \nonumber
=&
\sum_{i=1}^{l} d_{ii}^{\prime} \int_{0}^{1}\int_{\Omega}\int_{\Omega}k_{\sigma}(x-y,t)\alpha_{i}(y,t)\alpha_{i}(x,t)dydxdt - \sum_{i=1}^{l} d_{ii}^{\prime} \int_{0}^{1}\int_{\Omega} \alpha_{i}^{2}(x,t)dxdt
\\  \nonumber
&+ \sum_{i=1}^{l} \int_{0}^{1}\int_{\Omega}c_{i}^{\prime}(x,t)\alpha_{i}^{2}(x,t)dxdt
\\ \nonumber
=&
- \sum_{i=1}^{l}\frac{d_{ii}^{\prime}}{2} \int_{0}^{1} \int_{\Omega}\int_{\Omega} k_{\sigma}(x-y,t)(\alpha_{i}(x,t)-\alpha_{i}(y,t))^{2}dxdydt
+ \sum_{i=1}^{l} \int_{0}^{1}\int_{\Omega}c_{i}^{\prime}(x,t)\alpha_{i}^{2}(x,t)dxdt
\\ \nonumber
& +  \sum_{i=1}^{l}d_{ii}^{\prime}\int_{0}^{1} \int_{\Omega} \left( \int_{\Omega}  k_{\sigma}(x-y,t) dy - 1 \right)\alpha_{i}^{2}(x,t)dxdt
\\ \nonumber
\leq & \sum_{i=1}^{l} \int_{0}^{1}\int_{\Omega}c_{i}^{\prime}(x,t)\alpha_{i}^{2}(x,t)dxdt.
\end{align}
In particular,
\begin{align*}
c_{i}^{\prime}(x,t)=& \frac{1}{2} d_{ii}^{\prime}(\tau) \int_{\Omega}k_{\sigma}(x-y,t)\left( \sqrt{ \frac{\varphi_{i}(y,t)}{\varphi_{i}(x,t)}} - \sqrt{\frac{\psi_{i}(y,t)}{\psi_{i}(x,t)}} \right)^{2}dy \\
&+ \frac{1}{2}\sum_{j=1}^{l}d_{ij} \int_{\Omega}k_{\sigma}(x-y,t) \left[ \left( \sqrt{ \frac{\varphi_{j}(y,t)}{\varphi_{i}(x,t)}} - \sqrt{\frac{\psi_{j}(y,t)}{\psi_{i}(x,t)}} \right)^{2} \right]^{\prime}dy\\
&+ \frac{1}{2}\sum_{j=1}^{l}a_{ij}(x,t)\left[ \left( \sqrt{ \frac{\varphi_{j}(x,t)}{\varphi_{i}(x,t)}} - \sqrt{\frac{\psi_{j}(x,t)}{\psi_{i}(x,t)}} \right)^{2} \right]^{\prime} - \partial_{t}\beta_{i}(x,t) - \tau \partial_{t}\beta_{i}^{\prime}(x,t).
\end{align*}
Then, simple calculations give
\begin{equation}\label{3-18}
\begin{aligned}
&\frac{1}{2}\sum_{i=1}^{l} d_{ii}^{\prime}(\tau) \int_{0}^{1}\int_{\Omega}  \int_{\Omega}  k_{\sigma}(x-y,t) \left[ \left( \sqrt{ \frac{\varphi_{i}(y,t)}{\varphi_{i}(x,t)}} - \sqrt{\frac{\psi_{i}(y,t)}{\psi_{i}(x,t)}} \right)^{2} \right]\alpha_{i}^{2}(x,t)dydxdt
\\
=&\sum_{i=1}^{l} d_{ii}^{\prime}(\tau)\int_{0}^{1}\int_{\Omega}  \int_{\Omega} \left( k_{\sigma}(x-y,t) \left( \frac{1}{2} - \sqrt{ \frac{\varphi_{i}(x,t)\psi_{i}(y,t)}{\varphi_{i}(y,t)\psi_{i}(x,t) } } +\frac{\varphi_{i}(x,t)\psi_{i}(y,t)}{2\varphi_{i}(y,t)\psi_{i}(x,t)} \right) \right.\\
&\hspace{3.5cm}\varphi_{i}(y,t)\psi_{i}(x,t) \Bigg)dydxdt.
\end{aligned}
\end{equation}
Using the symmetry of $d_{ij}$ and of $k_{\sigma}(\cdot,t)$, we have
%
%
It then follows that
\begin{align}\label{3-19}
&\frac{1}{2}\sum_{i,j=1}^{l}d_{ij} \int_{0}^{1} \int_{\Omega} \int_{\Omega}k_{\sigma}(x-y,t) \left[ \left( \sqrt{ \frac{\varphi_{j}(y,t)}{\varphi_{i}(x,t)}} - \sqrt{\frac{\psi_{j}(y,t)}{\psi_{i}(x,t)}} \right)^{2} \right]^{\prime}\alpha_{i}^{2}(x,t) dydxdt\\ \nonumber
=& \frac{1}{2}\sum_{i,j=1}^{l}d_{ij} \int_{0}^{1} \int_{\Omega} \int_{\Omega} k_{\sigma}(x-y,t)  \left( \frac{\psi_{i}(x,t)}{\varphi_{i}(x,t)} \right) \left( \frac{\varphi_{i}(x,t)}{\psi_{i}(x,t)} \right)^{\prime} \left( \varphi_{i}(x,t)\psi_{j}(y,t)\right. \\\nonumber
& \hspace{8.7cm} - \left. \varphi_{j}(y,t)\psi_{i}(x,t) \right)dydxdt.
\end{align}
Moreover, we have
\begin{equation}\label{3-20}
\begin{aligned}
&\frac{1}{2}\sum_{i,j=1}^{l} \int_{0}^{1}  \int_{\Omega}a_{ij}(x,t)\left[ \left( \sqrt{ \frac{\varphi_{j}(x,t)}{\varphi_{i}(x,t)}} - \sqrt{\frac{\psi_{j}(x,t)}{\psi_{i}(x,t)}} \right)^{2} \right]^{\prime}\alpha_{i}^{2}(x,t) dx dt\\
=& \frac{1}{2}\sum_{i,j=1}^{l} \int_{0}^{1}  \int_{\Omega} a_{ij}(x,t) \left( \frac{\psi_{i}(x,t)}{\varphi_{i}(x,t)} \right) \left( \frac{\varphi_{i}(x,t)}{\psi_{i}(x,t)} \right)^{\prime} \left( \varphi_{i}(x,t)\psi_{j}(x,t) - \varphi_{j}(x,t)\psi_{i}(x,t) \right)dx dt.
\end{aligned}
\end{equation}
Define $h(z_{1},z_{2})=(z_{1}-z_{2})\left( \ln z_{1} - \ln z_{2} \right)$, it follows that
\begin{align}\label{3-21}
&- \sum_{i=1}^{l} \int_{0}^{1}  \int_{\Omega} \partial_{t}\beta_{i}(x,t)\alpha_{i}^{2}(x,t) dx dt
\\ \nonumber
=& \sum_{i,j=1}^{l} \frac{d_{ij}}{2\tau} \int_{0}^{1}  \int_{\Omega} \int_{\Omega}  \ln\left( \frac{\varphi_{i}(x,t)}{\psi_{i}(x,t)} \right) k_{\sigma}(x-y,t) \left( \varphi_{j}(y,t)\psi_{i}(x,t)-\psi_{j}(y,t)\varphi_{i}(x,t) \right)dydxdt
\\ \nonumber
& +\sum_{i,j=1}^{l} \int_{0}^{1}  \int_{\Omega} \frac{1}{2\tau} \ln\left( \frac{\varphi_{i}(x,t)}{\psi_{i}(x,t)} \right) a_{ij}(x,t)\left( \varphi_{j}(x,t)\psi_{i}(x,t)-\psi_{j}(x,t)\varphi_{i}(x,t) \right)dxdt
\\ \nonumber
=& \frac{\tau}{4} \sum_{i=1}^{l}\left( \frac{d_{ii}(\tau)}{\tau} \right)^{\prime}\int_{0}^{1}  \int_{\Omega} \int_{\Omega} k_{\sigma}(x-y,t) h(\varphi_{i}(y,t)\psi_{i}(x,t),\psi_{i}(y,t)\varphi_{i}(x,t))  dydxdt
\\ \nonumber
&- \sum_{\substack{ 1\leq i,j \leq l \\i\neq j }}\frac{d_{ij}}{4 \tau}  \int_{0}^{1}  \int_{\Omega} \int_{\Omega} k_{\sigma}(x-y,t) h(\varphi_{j}(y,t)\psi_{i}(x,t),\psi_{j}(y,t)\varphi_{i}(x,t)) dydxdt
\\ \nonumber
&+ \sum_{i=1}^{l} \frac{d_{ii}^{\prime}(\tau)}{2} \int_{0}^{1}\int_{\Omega}\int_{\Omega} \ln \left( \frac{\varphi_{i}(x,t)\psi_{i}(y,t)}{\psi_{i}(x,t)\varphi_{i}(y,t)} \right)k_{\sigma}(x-y,t)\varphi_{i}(y,t)\psi_{i}(x,t)dydxdt
\\ \nonumber
&-  \frac{1}{4\tau} \sum_{\substack{ 1\leq i,j \leq l \\i\neq j }} \int_{0}^{1}  \int_{\Omega} a_{ij}(x,t) h(\varphi_{j}(x,t)\psi_{i}(x,t),\psi_{j}(x,t)\varphi_{i}(x,t))  dxdt,
\end{align}
and
\begin{equation}\label{3-22}
\begin{aligned}
&- \sum_{i=1}^{l} \int_{0}^{1}  \int_{\Omega} \tau \partial_{t}\beta_{i}^{\prime}(x,t)\alpha_{i}^{2}(x,t)dxdt\\
=& \sum_{i,j=1}^{l} \frac{d_{ij}}{2} \int_{0}^{1}  \int_{\Omega} \int_{\Omega} \left( \frac{\psi_{i}(x,t)}{\varphi_{i}(x,t)} \right) \left( \frac{\varphi_{i}(x,t)}{\psi_{i}(x,t)} \right)^{\prime} k_{\sigma}(x-y,t)\left(  \varphi_{j}(y,t)\psi_{i}(x,t) - \psi_{j}(y,t)\varphi_{i}(x,t)  \right)dydxdt\\
&+ \sum_{i,j=1}^{l}\frac{1}{2}\int_{0}^{1}  \int_{\Omega} \left( \frac{\psi_{i}(x,t)}{\varphi_{i}(x,t)} \right) \left( \frac{\varphi_{i}(x,t)}{\psi_{i}(x,t)} \right)^{\prime} a_{ij}(x,t) \left( \varphi_{j}(x,t)\psi_{i}(x,t) - \psi_{j}(x,t)\varphi_{i}(x,t) \right)dx dt.
\end{aligned}
\end{equation}
Define $g(z)=z - 1 - \ln z$. Let
$r= \frac{\varphi_{i}(x,t)\psi_{i}(y,t)}{\varphi_{i}(y,t)\psi_{i}(x,t) }  $.
Combining \eqref{3-18}-\eqref{3-22}, we get
\begin{align} \label{3-23}
&\sum_{i=1}^{l} \int_{0}^{1}\int_{\Omega}c_{i}^{\prime}(x,t)\alpha_{i}^{2}(x,t)dxdt
\\ \nonumber
=&\sum_{i=1}^{l} d_{ii}^{\prime}(\tau)\int_{0}^{1}\int_{\Omega}  \int_{\Omega} k_{\sigma}(x-y,t)\varphi_{i}(y,t)\psi_{i}(x,t)   \left( \frac{1}{2} - \sqrt{r} + \frac{r}{2} \right) dydxdt
\\ \nonumber
&+ \sum_{i = 1}^{l} \frac{d_{ii}^{\prime}(\tau)}{2} \int_{0}^{1}\int_{\Omega}\int_{\Omega} \ln \left( r \right)k_{\sigma}(x-y,t)\varphi_{i}(y,t)\psi_{i}(x,t)dydxdt
\\ \nonumber
&+ \frac{\tau}{4} \sum_{i=1}^{l}\left( \frac{d_{ii}(\tau)}{\tau} \right)^{\prime}\int_{0}^{1}  \int_{\Omega} \int_{\Omega} k_{\sigma}(x-y,t) h(\varphi_{i}(y,t)\psi_{i}(x,t),\psi_{i}(y,t)\varphi_{i}(x,t) )dydxdt
\\ \nonumber
&-\sum_{i\neq j}^{l}\frac{d_{ij}}{4 \tau}  \int_{0}^{1}  \int_{\Omega} \int_{\Omega}
k_{\sigma}(x-y,t) h(\varphi_{j}(y,t)\psi_{i}(x,t),\psi_{j}(y,t)\varphi_{i}(x,t))dydxdt
\\ \nonumber
&- \frac{1}{4\tau} \sum_{i\neq j}^{l} \int_{0}^{1}  \int_{\Omega} a_{ij}(x,t) h(\varphi_{j}(x,t)\psi_{i}(x,t),\psi_{j}(x,t)\varphi_{i}(x,t)) dxdt
\\ \nonumber
\leq& -\sum_{i=1}^{l} d_{ii}^{\prime}(\tau) \int_{0}^{1}\int_{\Omega}\int_{\Omega}  k_{\sigma}(x-y,t) \varphi_{i}(y,t)\psi_{i}(x,t) g\left( \sqrt{ r } \right)dydxdt
\\ \nonumber
&+ \sum_{i=1}^{l} d_{ii}^{\prime}(\tau) \int_{0}^{1}\int_{\Omega}\int_{\Omega}  k_{\sigma}(x-y,t) \varphi_{i}(y,t)\psi_{i}(x,t) \left( -\frac{1}{2} + \frac{r}{2} \right)dydxdt \\\nonumber
\leq& 0.
\end{align}
Consequently, it follows from \eqref{3-17} that $s^{\prime}(L_{\tau}) \leq 0$.

\textbf{Step 2.}
In this step, we prove the remaining results.

Note that $K(x,t) = \left( \int_{\Omega}k_{\sigma}(x-y,t)dy - 1 \right)\alpha_{i}^{2}(x,t) \leq 0$ and there exists $(x_{0},t_{0}) \in \Omega \times (0,1)$ such that $K(x_{0},t_{0})<0$. By the continuity of $K$, we get
\begin{equation*}
\int_{0}^{1} \int_{\Omega} \left( \int_{\Omega}  k_{\sigma}(x-y,t) dy - 1 \right)\alpha_{i}^{2}(x,t)dxdt < 0.
\end{equation*}
Hence, if there exists $1 \leq k\leq l$ such that $d_{kk}^{\prime}(\tau)> 0$, if follows from \eqref{3-17} that  $ s^{\prime}(L_{\tau})<0$.

If there exists $(\phi,\lambda)$ such that
\begin{equation*}
D \hat{\mathcal{P}}[\phi](x) - D_{0}\phi(x) + \hat{A}(x)\phi(x) = \lambda \phi(x).
\end{equation*}
Let $u(x,t) = e^{\frac{1}{\tau}\int_{0}^{t}r(s)ds} \phi(x)$. It follows that $u(x,t)$ is 1-periodic in $t$ and
\begin{equation*}
L_{\tau}[u](x,t) - \lambda u(x,t) = -r(t)u(x,t) + \left( D\mathcal{P}[u](x,t) - D\hat{\mathcal{P}}[u](x,t) \right) +   \left(A(x,t) - \hat{A}(x) \right)u(x,t) = 0
\end{equation*}
for all $(x,t) \in \bar{\Omega}\times [0,1]$. This implies that $s(L_{\tau})=\lambda$ and $s^{\prime}(L_{\tau})=0$.

On the other hand, if $s^{\prime}(L_{\tau}) = 0$, it follows from \eqref{3-23} that
\begin{equation*}
\varphi_{i}(y,t)\psi_{i}(x,t) = \psi_{i}(y,t)\varphi_{i}(x,t)
\enspace \text{and} \enspace
\varphi_{j}(x,t)\psi_{i}(x,t) = \psi_{j}(x,t)\varphi_{i}(x,t).
\end{equation*}
for all $x,y\in\Omega,t\in[0,1],i,j\in \mathbb{S}$. Hence, $\varphi_{j}(y,t)\psi_{i}(x,t) = \psi_{j}(y,t)\varphi_{i}(x,t)$ for all $x,y\in\Omega,t\in[0,1],i,j\in \mathbb{S}$.
It then follows that
\begin{equation*}
\begin{aligned}
\partial_{t}(\alpha_{i}^{2}(x,t)) =& \frac{1}{\tau}
\left(
\sum_{j=1}^{l} d_{ij}\int_{\Omega} k_{\sigma}(x-y,t) \left(\varphi_{j}(y,t)\psi_{i}(x,t) - \psi_{j}(y,t)\varphi_{i}(x,t) \right) dy
\right.\\
& \left. + \sum_{j=1}^{l} a_{ij}(x,t) \left(\varphi_{j}(x,t)\psi_{i}(x,t) -\psi_{j}(x,t)\varphi_{i}(x,t)\right) \right) 
=0
\end{aligned}
\end{equation*}
for all $(x,t)\in\Omega\times[0,1],i\in\mathbb{S}$
and
$\varphi_{i} = \rho(t) \psi_{i}$ for some $1-$periodic function $\rho(t)$.
This implies that $\alpha_{i} = \alpha_{i}(x)$ independent of $t$ and $\beta_{i}= \frac{1}{2} \ln \rho(t)$.
By \eqref{3-15}, we have
\begin{equation}\label{3-24}
-s(L_{\tau})\alpha_{i}(x)
=
-\sum_{j=1}^{l} d_{ij}\int_{\Omega}k_{\sigma}(x-y,t)\alpha_{j}(y)dy + d_{ii}\alpha_{i}(x) - \sum^{l}_{j=1}a_{ij}(x,t)\alpha_{j}(x) + \tau\partial_{t}\beta_{i}(t)\alpha_{i}(x).
\end{equation}
Integrating \eqref{3-24} over $(0,1)$, we get
\begin{equation}\label{3-25}
\sum_{j=1}^{l}d_{ij}\int_{\Omega}\int_{0}^{1} k_{\sigma}(x-y,t)dt \alpha_{j}(y) dy - d_{ii}\alpha_{i}(x) + \sum_{j=1}^{l} \int_{0}^{1} a_{ij}(x,t)dt\alpha_{j}(x)
=
s(L_{\tau})\alpha_{i}(x).
\end{equation}
Let $r(t) = \frac{\tau}{2} \frac{d \ln \rho(t)}{ dt }$. Combining \eqref{3-24} and \eqref{3-25}, we obtain that
\begin{equation*}
r(t)\alpha(x) = \left( DP[\alpha](x,t) - D\hat{P}[\alpha](x,t) \right) +   \left(A(x,t) - \hat{A}(x) \right)\alpha(x)
\enspace  \text{in} \enspace \bar{\Omega}\times[0,1].
\end{equation*}
The proof is completed.
\end{proof}

\begin{proof}[Proof of Theorem \ref{THM1.10}]
\indent

(i) For a fixed $t\in [0,1]$, since $s(N_{1}(t))$ is the principal eigenvalue of $N_{1}(t)$, there exists $v(\cdot,t)\in X^{++}$ such that
\begin{equation*}
B\mathcal{P}[v](x,t) - B_{0}v(x,t) + A(x,t)v(x,t) = s(N_{1}(t)) v(x,t) \enspace\text{in}\enspace \bar{\Omega}.
\end{equation*}
By the perturbation theory (see \cite{Kato}), we can conclude that $v\in C^{1}([0,1],C(\bar{\Omega},\mathbb{R}^{l}))$ and $v(x,t+1)=v(x,t)$.
Define
\begin{equation*}
c(t)=\exp \left\{ -\frac{1}{\tau} \left( t\int_{0}^{1}s(N_{1}(z))dz - \int_{0}^{t} s(N_{1}(z))dz \right) \right\}
\enspace \text{and} \enspace
\varphi(x,t)=c(t)v(x,t).
\end{equation*}
Clearly, $\varphi(x,t)$ is a 1-periodic function.
Given arbitrary $\varepsilon>0$, there exists a sufficiently small $\tau_{0}>0$ such that $\tau|\partial_{t}v_{i}|\leq \frac{\varepsilon}{2}v_{i}$ for $\tau<\tau_{0}$. Moreover,
a direct computation yields that
\begin{align*}
&-L_{\tau}[\varphi](x,t)+ \left( \int_{0}^{1} s(N_{1}(z)) dz - \varepsilon \right) \varphi(x,t) \\
\leq&
\left( \tau \partial_{t}v(x,t) + \left( B- D \right)[v](x,t) +\left( D_{0} - B_{0} \right)[v](x,t)- \varepsilon v(x,t) \right) c(t)
\end{align*}
for all $(x,t)\in \bar{\Omega}\times [0,1]$.
Since   $\mathcal{P}[v](x,t)$ and $v(x,t)$ are bounded, there exists $0<\tau_{1}<\tau_{0}$, such that for $0<\tau<\tau_{1}$,
\begin{equation*}
-L_{\tau}[\varphi](x,t)+ \left( \int_{0}^{1} s(N_{1}(z)) dz - \varepsilon \right) \varphi(x,t)\le 0 \enspace\text{for all}\enspace (x,t) \in \bar{\Omega}\times [0,1].
\end{equation*}
This implies that $\lambda_{p}(L_{\tau}) \geq \int_{0}^{1} s(N_{1}(z)) dz - \varepsilon$. Similarly, we can obtain that there exists $\tau_{2}$ such that for $0<\tau<\tau_{2}$,
\begin{equation*}
\lambda_{p}^{\prime}(L_{\tau}) \leq \int_{0}^{1} s(N_{1}(z)) dz + \varepsilon.
\end{equation*}
Since $\varepsilon$ can be chosen arbitrarily, using Theorem \ref{THM1.6} and Remark \ref{REMARK1}, we get the desired result.

(ii) By Theorem \ref{THM1.5}, we may assume, without loss of generality, that $s(L_{\tau})$ is the principal eigenvalue for all $\tau>0$.
Let $\{\tau_{n}\}_{n\geq 1}$ be any sequence satisfying $\tau_{n} \to +\infty$ as $n\to\infty$ and let $( s_{n}(L_{\tau}), \varphi_{n})$ be the corresponding eigenpair with normalization $\| \varphi_{n} \|_{\mathcal{X}} = 1$. That is, $\varphi_{n}$ satisfies
\begin{equation}\label{3-26}
\begin{aligned}
\tau_{n} \partial_{t} \varphi_{n,i}(x,t) =& \sum_{j=1}^{l} d_{ij}(\tau_{n}) \int_{\Omega} k(x-y,t) \varphi_{n,j}(y,t)dy - d_{ii}(\tau_{n})\varphi_{n,i}(x,t) + \sum_{j=1}^{l} a_{ij}(x,t) \varphi_{n,j}(x,t) \\
&- s_{n}(L_{\tau}) \varphi_{n,i}(x,t)
\end{aligned}
\end{equation}
for all $(x,t)\in\bar{\Omega} \times \mathbb{R}$ and $1\leq i \leq l$.
Dividing the above equation by $\varphi_{n,i}$ and integrating it over $\Omega \times (0,1)$, we get
\begin{align*} s_{n}(L_{\tau}) \geq& \frac{d_{ii}(\tau_{n})}{2|\Omega|} \int_{0}^{1}\int_{\Omega} \int_{\Omega}  k(x-y,t) \left( \sqrt{\frac{ \varphi_{n,i}(y,t) }{ \varphi_{n,i}(x,t)}} - \sqrt{\frac{ \varphi_{n,i}(x,t) }{ \varphi_{n,i}(y,t) } } \right)^{2}dydxdt
\\
&+ d_{ii}(\tau_{n}) \left( \frac{1}{|\Omega|} \int_{0}^{1} \int_{\Omega} \int_{\Omega} k(x-y,t)dydxdt - 1 \right) + \frac{1}{|\Omega|} \int_{0}^{1} \int_{\Omega} a_{ii}(x,t) dx dt
\\
\geq& d_{ii}(\tau_{n}) \left( \frac{1}{|\Omega|} \int_{0}^{1} \int_{\Omega} \int_{\Omega} k(x-y,t)dydxdt - 1 \right) + \frac{1}{|\Omega|} \int_{0}^{1} \int_{\Omega} a_{ii}(x,t) dx dt.
\end{align*}
Since $ \{d_{ii}(\tau_{n})\}_{n\geq 1}$ is a bounded sequence,
$s_{n}(L_{\tau})$ is bounded from below. This together with Theorem \ref{THM1.9} implies that there exists $s^{*}$ such that
$\lim_{n \to \infty} s_{n}(L_{\tau}) = s^{*}.$

(iii)
The proof is similar to that of Theorem \ref{THM1.7}, so we omit the details here.
\end{proof}

\section{Applications }
\subsection{Zika virus model}
In this section, we study the global dynamics of the following Zika virus model \cite{LZ,MWW,WFLS}:
\begin{equation}\label{4-1}
\left\{
\begin{aligned}
&\tau \partial_{t} H_{i} =d_{1}(x,t)\mathcal{K}^{1}_{\Omega,\sigma,m}[H_{i}](x,t)-\rho(x,t)H_{i}+\sigma_{1}(x,t)H_{u}(x)V_{i}, &x&\in \bar\Omega, t>0,\\
&\tau \partial_{t} V_{u}
=d_{2}(x,t)\mathcal{K}^{2}_{\Omega,\sigma,m}[V_{u}](x,t)-\sigma_{2}(x,t)V_{u}H_{i}+\beta(x,t)(V_{u}+V_{i})\\
&\hspace{1.5cm}-\mu(x,t)(V_{u}+V_{i})V_{u}, &x&\in \bar\Omega, t>0, \\
&\tau \partial_{t} V_{i}
=d_{2}(x,t)\mathcal{K}^{2}_{\Omega,\sigma,m}[V_{i}](x,t)+\sigma_{2}(x,t)V_{u}H_{i}-\mu(x,t)(V_{u}+V_{i})V_{i}, &x&\in \bar\Omega, t>0,\\
&(H_{i}(x,0), V_{u}(x,0), V_{i}(x,0))=(H_{i0}(x), V_{u0}(x), V_{i0}(x)), &x&\in\bar\Omega.
\end{aligned}
\right.
\end{equation}
Here, the dispersal operator is defined as
\begin{equation*}
\mathcal{K}^{j}_{\Omega,\sigma,m}[u](x,t) =
\frac{1}{\sigma^{m}}\left( \int_{\Omega}\frac{1}{\sigma^{N}} k_{j} \left( \frac{x-y}{\sigma},t \right)u(y,t)dy - u(x,t) \right),
\end{equation*}
with $k_{j}$ satisfying $\rm{(\tilde{H}3) }$, and $k_{j}(\cdot,t)$ is symmetric and compactly supported for $j=1, 2$;
$H_{i}(x,t),$ $V_{u}(x,t)$ and $V_{i}(x,t)$ denote the densities of infected hosts, uninfected vectors, and infected vectors at position $x$ and time $t$, respectively;
$d_{1}$ and $d_{2}$ denote the dispersal rates of the hosts and vectors, respectively;
$\tau$ is the positive constant representing the frequency;
$\rho$ is the removal rate of infected hosts; $\sigma_{1}$ and $\sigma_{2}$ denote the infection rates of uninfected hosts and uninfected vectors, respectively; $\beta$ is the birth rate of the vectors, and $\mu$ is the death rate of the vectors;
$H_{u}$ is the density of uninfected hosts.

We always assume that $H_{u}$ is a continuous positive function. Parameters $d_{1},d_{2},\rho, \sigma_{1}, \sigma_{2}, \beta, \mu$ are assumed to be continuous, positive functions of $(x,t)$ and $1-$periodic in time $t$. The initial function $(H_{i0}, V_{u0}, V_{i0})$ is continuous, nonnegative, and nontrivial.

By general semigroup theory
\cite{Pazy} and the comparison principle, \eqref{4-1} admits a unique positive global solution $(H_{i}, V_{u}, V_{i})$.

Next, we study the following time periodic problem associated to \eqref{4-1}:
\begin{equation}\label{4-2}
\left\{
\begin{aligned}
&\tau \partial_{t} H_{i} =d_{1}(x,t)\mathcal{K}^{1}_{\Omega,\sigma,m}[H_{i}](x,t)-\rho(x,t)H_{i}+\sigma_{1}(x,t)H_{u}(x)V_{i}, &&(x,t)\in \bar\Omega\times \mathbb{R},\\
&\tau \partial_{t} V_{u}
=d_{2}(x,t)\mathcal{K}^{2}_{\Omega,\sigma,m}[V_{u}](x,t)-\sigma_{2}(x,t)V_{u}H_{i}+\beta(x,t)(V_{u}+V_{i})\\
&\hspace{1.5cm}-\mu(x,t)(V_{u}+V_{i})V_{u},&&(x,t)\in \bar\Omega\times \mathbb{R}, \\
&\tau \partial_{t} V_{i}
=d_{2}(x,t)\mathcal{K}^{2}_{\Omega,\sigma,m}[V_{i}](x,t)+\sigma_{2}(x,t)V_{u}H_{i}-\mu(x,t)(V_{u}+V_{i})V_{i},&&(x,t)\in \bar\Omega\times\mathbb{R},\\
&(H_{i}(x,t), V_{u}(x,t), V_{i}(x,t))=(H_{i}(x,t+1), V_{u}(x,t+1), V_{i}(x,t+1)),&&(x,t)\in \bar\Omega\times\mathbb{R}.
\end{aligned}
\right.
\end{equation}
Let $V=V_{u}+V_{i}$. It then follows from \eqref{4-2} that
\begin{equation}\label{4-3}
\left\{
\begin{aligned}
&\tau\partial_{t} V = d_{2}(x,t)\mathcal{K}^{2}_{\Omega,\sigma,m}[V](x,t) + \beta(x,t)V - \mu(x,t) V^{2}, &&(x,t)\in \bar\Omega\times\mathbb{R},\\
&V(x,t)= V(x,t+1), &&(x,t)\in \bar\Omega\times\mathbb{R}.
\end{aligned}
\right.
\end{equation}
Define operator $L_{0}$ by
\begin{equation*}
L_{0}[\varphi](x,t) = -\tau\partial_{t}\varphi(x,t)  + d_{2}(x,t)\mathcal{K}^{2}_{\Omega,\sigma,m}[\varphi](x,t) + \beta(x,t)\varphi(x,t).
\end{equation*}
It then follows from \cite{RS,ShenVo} that if $s(L_{0})>0$, then \eqref{4-3} admits a unique positive solution $V_{*}$,
if $s(L_{0}) \leq 0$, then
\eqref{4-3} admits no nonnegative nontrivial solution.
In what follows, we assume that $s(L_{0})>0$.
As a result, investigating the existence of positive solutions of \eqref{4-2} is equivalent to studying the following problem:
\begin{equation}\label{4-4}
\left\{
\begin{aligned}
&\tau \partial_{t} H_{i} =d_{1}(x,t)\mathcal{K}^{1}_{\Omega,\sigma,m}[H_{i}](x,t)-\rho(x,t)H_{i}+\sigma_{1}(x,t)H_{u}(x)V_{i},&&(x,t)\in \bar\Omega\times \mathbb{R}, \\
&\tau \partial_{t} V_{i}
=d_{2}(x,t)\mathcal{K}^{2}_{\Omega,\sigma,m}[V_{i}](x,t)+\sigma_{2}(x,t)(V_{*}-V_{i})H_{i}-\mu(x,t)V_{*}V_{i},&&(x,t)\in \bar\Omega\times \mathbb{R},\\
&(H_{i}(x,t), V_{i}(x,t))=(H_{i}(x,t+1), V_{i}(x,t+1)),&&(x,t)\in \bar\Omega\times\mathbb{R}.
\end{aligned}
\right.
\end{equation}
Define
\begin{equation*}
N[\varphi](x,t) = \left(
\begin{array}{c}
-\tau\partial_{t} \varphi_{1}(x,t)  +\frac{d_{1}(x,t)}{\sigma^{m}}\int_{\Omega}k_{1,\sigma}(x-y,t)\varphi_{1}(y,t)dy  \\
-\tau\partial_{t} \varphi_{2}(x,t)  + \frac{d_{2}(x,t)}{\sigma^{m}}\int_{\Omega}k_{2,\sigma}(x-y,t)\varphi_{2}(y,t)dy   \\
\end{array}
\right),
\end{equation*}
\begin{equation*}
A(x,t) \!=\! \left(
\begin{array}{cc}
-\rho(x,t) &  \sigma_{1}(x,t)H_{u}(x)\\
\sigma_{2}(x,t)V_{*}(x,t) &  -\mu(x,t)V_{*}(x,t)\\
\end{array} \right)\!,
A_{1}(x,t) = \text{diag}\left\{-\frac{d_{1}(x,t)}{\sigma^{m}},  -\frac{d_{2}(x,t)}{\sigma^{m}} \right\} + A(x,t)
\end{equation*}
and $L_{1}[\varphi](x,t) = N[\varphi](x,t) + A_{1}(x,t)\varphi(x,t)$.
\begin{theorem}\label{T Zika 1}
Suppose $s(L_{0}) > 0$.
If $s(L_{1})>0$, then the equation \eqref{4-4} admits a unique positive solution $(H_{i,*},V_{i,*})$.
If $s(L_{1})<0$ or $s(L_{1})=0$ is the principal eigenvalue, then the equation \eqref{4-4} admits no nonnegative nontrivial solution.
\end{theorem}
\begin{proof}
{\bf Step 1.} In this step, we assume that $s(L_{1})>0$.
Consider the following equation
\begin{equation}\label{4-5}
\left\{
\begin{aligned}
&\tau \partial_{t} H \!=\!d_{1}(x,t)\mathcal{K}^{1}_{\Omega,\sigma,m}[H](x,t)\!-\!\rho(x,t)H(x,t)+\sigma_{1}(x,t)H_{u}(x)V_{*}(x,t),&& (x,t)\in \bar\Omega\times\mathbb{R},\\
&H(x,t)=H(x,t+1),&&(x,t)\in \bar\Omega\times\mathbb{R}.
\end{aligned}
\right.
\end{equation}
Define
\begin{equation*}
L_{2}[u](x,t) = -\tau \partial_{t} u (x,t) + d_{1}(x,t)\mathcal{K}^{1}_{\Omega,\sigma,m}[u](x,t) -\rho(x,t)u(x,t).
\end{equation*}
Using Theorem \ref{THM1.6}, by choosing test function $1$, it is easy to see that $s(L_{2})<0$. By Theorem \ref{THM1.5}, for any $\varepsilon>0$, there exists an approximation operator $L_{2}^{\varepsilon}$ such that
$||L_{2}^{\varepsilon}-L_{2}  || \leq \varepsilon, s(L_{2}^{\varepsilon})<0$, and $L_{2}^{\varepsilon}$ admits the principal eigenvalue.
Let $\phi_{\varepsilon}$ be the associated principal eigenfunction.
It is easy to verify that $C\phi_{\varepsilon}$ and $0$ are the ordered upper-lower solutions of \eqref{4-5}, for some positive constant $C$.
It then follows from \cite[Theorem]{SLW} that \eqref{4-5} admits a unique positive solution $\bar{H}$.
A direct computation yields that $(\bar{H},V_{*})$ is a strict upper solution of \eqref{4-4}.

Next, we construct the lower solution of \eqref{4-4}.  By Theorem \ref{THM1.5}, there exists $A^{k}_{1}(x,t)$ satisfying
\begin{equation*}
\lim_{k\to\infty} \left| s(L_{1}(A^{k}_{1})) - s(L_{1}(A_{1})) \right|=0,
\end{equation*}
and
$s(L(A^{k}_{1}))$ is the principal eigenvalue of the following eigenvalue problem
\begin{equation*}
\left\{
\begin{aligned}
& L(A^{k}_{1})[\varphi](x,t) =  \lambda \varphi(x,t), & (x,t) \in \bar{\Omega} \times \mathbb{R}, \\
&\varphi(x,t+1) = \varphi(x,t), & (x,t) \in \bar{\Omega} \times \mathbb{R}.
\end{aligned}
\right.
\end{equation*}
Since $s(L_{1}(A_{1}))>0$, for large enough $k$, say $k \geq k_{0}$, we have
\begin{equation*}
s(L(A^{k}_{1})) - \| A_{1}- A^{k}_{1} \| _{\infty} > 0.
\end{equation*}
Let $\varphi_{k}$ be the eigenfunction corresponding to $s(L(A^{k}_{1}))$. A direct computation yields that $\varepsilon \varphi_{k}$ is a lower solution of \eqref{4-4} for $\varepsilon>0$ small enough and $k \geq k_{0}$.
By a standard iterative method, we get the existence of positive solutions to \eqref{4-4}. The proof of uniqueness is similar to \cite{WFLS}, so we omit it.

{\bf Step 2.} In this step, we prove the remaining results. We only consider the case where $s(L_{1})=0$ is the principal eigenvalue. The case $s(L_{1})<0$ follows directly from Theorem \ref{th6.11} (ii).
If \eqref{4-4} admits a nonnegative nontrivial solution $u=(u_{1},u_{2})$, we can conclude that $u_{i}>0, i=1,2$ by Corollary \ref{COR2.9}.
Let $\varphi=(\varphi_{1},\varphi_{2})$ be the principal eigenfunction associated to $s(L_{1})$ with normalization $\varphi_{i} < u_{i}, i=1,2$, and define
\begin{equation*}
\tilde{A}_{1}(x,t)= A_{1}(x,t) -
\left(
\begin{array}{cc}
0 &  0\\
\sigma_{2}(x,t)  \varphi_{2} &  0\\
\end{array} \right).
\end{equation*}
Then we have $0=- L_{1}[\varphi](x,t)\leq -N[\varphi](x,t) - \tilde{A}_{1}(x,t)\varphi(x,t)$ for all $x\in\bar{\Omega},t\in[0,1]$. This together with the comparison principle implies that $\varphi \geq u$, which contradicts the normalization.

\end{proof}

\begin{theorem}\label{th6.11}
Let $(H_{i},V_{u},V_{i})$ be the unique positive solution of \eqref{4-1}.
\begin{itemize}[leftmargin=0.9cm]
\item[(i)] If $s(L_{0})>0$ and $s(L_{1})>0$, then
\begin{equation*}
\lim\limits_{t\to\infty} \| (H_{i},V_{u},V_{i})(\cdot,t) -(H_{i,*},V_{*}-V_{i,*},V_{i,*})(\cdot,t)\|_{{C(\bar{\Omega})}^{3}} = 0.
\end{equation*}

\item[(ii)] If $s(L_{0})>0$ and $s(L_{1})<0$, then
\begin{equation*}
\lim\limits_{t\to\infty} \| (H_{i},V_{u},V_{i})(\cdot,t) - (0,V_{*},0)(\cdot,t)\|_{{C(\bar{\Omega})}^{3}} = 0.
\end{equation*}

\item[(iii)] If $s(L_{0})<0$, then
\begin{equation*}
\lim\limits_{t\to\infty} \| (H_{i},V_{u},V_{i})(\cdot,t) - (0,0,0)\|_{{C(\bar{\Omega})}^{3}} = 0.
\end{equation*}
\end{itemize}
\end{theorem}

\begin{proof}
(i)
Note that $V=V_{u}+V_{i}$ satisfies
\begin{equation*}  	
\left\{
\begin{aligned}
&\tau \partial_{t} V(x,t)
=d_{2}(x,t)\mathcal{K}^{2}_{\Omega,\sigma,m}[V](x,t) + \beta(x,t)V(x,t) - \mu(x,t)V^{2}(x,t), &x& \in\bar{\Omega}, t>0,\\
&V(x,0)=V_{0}(x)=V_{u}(x,0)+V_{i}(x,0), &x&\in\bar{\Omega}.
\end{aligned}
\right.
\end{equation*}
Since $ s(L_{0})>0 $,
it follows from \cite[Theorem E]{RS}  and  \cite[Theorem B]{ShenVo}  that
\begin{equation}\label{4-6}
\lim\limits_{t\to\infty}\|V(\cdot, t)-V_{*}(\cdot, t)\|_{C(\bar{\Omega})}=0.
\end{equation}
Hence, there exists $T_{\varepsilon}\gg 1$ such that
\begin{equation*}
0<V_{*}(x,t)-\varepsilon\le V(x,t) \le V_{*}(x,t)+\varepsilon \enspace\text{for all}\enspace  x\in\bar{\Omega}, t\geq T_{\varepsilon}.
\end{equation*}
Then we find that $(H_{i},V_{i})$ satisfies
\begin{equation}\label{4-7}
\left\{
\begin{aligned}
& \tau \partial_{t} H_{i}
=d_{1}(x,t)\mathcal{K}^{1}_{\Omega,\sigma,m}[H_{i}](x,t)-\rho(x,t)H_{i}+\sigma_{1}(x,t)H_{u}(x)V_{i}, &x\in\bar\Omega, t>T_{\varepsilon},\\
& \tau \partial_{t} V_{i}
\le
d_{2}(x,t)\mathcal{K}^{2}_{\Omega,\sigma,m}[V_{i}](x,t)+\sigma_{2}(x,t)(V_{*}(x,t)+\varepsilon-V_{i})^{+}H_{i}\\
&\hspace{1.5cm}-\mu(x,t)(V_{*}(x,t)-\varepsilon)V_{i}, &x\in\bar\Omega, t>T_{\varepsilon}.
\end{aligned}
\right.
\end{equation}
Define 		
\begin{equation*}
A_{\varepsilon}(x,t)=\left(
\begin{array}{cccc}
-\frac{d_{1}(x,t)}{\sigma^{m}}-\rho(x,t) &  \sigma_{1}(x,t)H_{u}(x)\\
\sigma_{2}(x,t)(V_{*}(x,t)+\varepsilon) &  -\frac{d_{2}(x,t)}{\sigma^{m}}-\mu(x,t)(V_{*}(x,t)-\varepsilon)\\
\end{array} \right).
\end{equation*}
Since $s(L_{1}(A))>0$, for small enough $\varepsilon>0$, say $0< \varepsilon <\varepsilon_{0}$, there holds $s(L_{1}(A_{\varepsilon}))>0$.
It follows from a similar argument as in Theorem \ref{T Zika 1} that there exist a lower solution $\delta(\varphi_{\varepsilon,1}, \varphi_{\varepsilon,2} )$ and a unique positive function $(H_{i,*}^{\varepsilon},V_{i,*}^{\varepsilon})$ to the following problem:
\begin{equation*}
\left\{
\begin{aligned}
& \tau \partial_{t} H_{i}
=d_{1}\mathcal{K}^{1}_{\Omega,\sigma,m}[H_{i}]-\rho(x,t)H_{i}+\sigma_{1}(x,t)H_{u}(x)V_{i}, &x&\in\bar\Omega, t>T_{\varepsilon},\\
& \tau \partial_{t} V_{i}
=
d_{2}\mathcal{K}^{2}_{\Omega,\sigma,m}[V_{i}]+\sigma_{2}(x,t)(V_{*}(x,t)+\varepsilon-V_{i})^{+}H_{i}-\mu(x,t)(V_{*}(x,t)-\varepsilon)V_{i}, &x&\in\bar\Omega, t>T_{\varepsilon},\\
&(H_{i}(x,0), V_{i}(x,0))=(H_{i}(x,1), V_{i}(x,1)), &x&\in \bar\Omega.
\end{aligned}
\right.
\end{equation*}
Choose constants $0 < \delta\ll 1$ and $k\gg 1$ such that
\begin{equation*}
\delta (\varphi_{\varepsilon,1}(x,0), \varphi_{\varepsilon,2}(x,0) )
\leq
\left( H_{i}(x,T_{\varepsilon}),V_{i}(x,T_{\varepsilon}) \right)
\leq
k(H_{i,*}^{\varepsilon}(x,0),V_{i,*}^{\varepsilon}(x,0))
\enspace\text{for all}\enspace
x\in\bar{\Omega}.
\end{equation*}
Let $(H_{i}^{\varepsilon},V_{i}^{\varepsilon})$ be the unique positive solution of
\begin{equation*}
\left\{
\begin{aligned}
& \tau \partial_{t} H_{i}
=d_{1}\mathcal{K}^{1}_{\Omega,\sigma,m}[H_{i}](x,t)-\rho(x,t)H_{i}+\sigma_{1}(x,t)H_{u}(x)V_{i},&\quad &x\in\bar\Omega, t>T_{\varepsilon},\\
& \tau \partial_{t} V_{i}
=
d_{2}\mathcal{K}^{2}_{\Omega,\sigma,m}[V_{i}](x,t)+\sigma_{2}(x,t)(V_{*}(x,t)+\varepsilon-V_{i})^{+}H_{i}\\
&\hspace{1.5cm}-\mu(x,t)(V_{*}(x,t)-\varepsilon)V_{i},&\quad &x\in\bar\Omega, t>T_{\varepsilon},\\
&(H_{i}(x,0), V_{i}(x,0))=k(H_{i,*}^{\varepsilon}(x,0),V_{i,*}^{\varepsilon}(x,0)),&\quad&x\in \bar\Omega.
\end{aligned}
\right.
\end{equation*}
By the standard iterative method (see e.g., \cite[Theorem 1.1]{Bao}), we have
\begin{equation*}
\limsup\limits_{t\to\infty} \| (H_{i}^{\varepsilon},V_{i}^{\varepsilon})(\cdot,t)
- (H_{i,*}^{\varepsilon},V_{i,*}^{\varepsilon})(\cdot,t)  \|_{{C(\bar{\Omega})}^{2}} = 0.
\end{equation*}
It then follows from the comparison principle that for $n\in\mathbb{N}^{+}$
\begin{equation*}
\limsup_{n\to\infty} (H_{i}(x,t+n),V_{i}(x,t+n) ) \leq
(H_{i,*}^{\varepsilon}(x,t),V_{i,*}^{\varepsilon}(x,t))
\enspace \text{uniformly in} \enspace \bar{\Omega}\times[0,1].
\end{equation*}
Letting $\varepsilon\to 0$,
\begin{equation*}\label{6.15}
\limsup_{n\to\infty} (H_{i}(x,t+n),V_{i}(x,t+n)) \leq
(H_{i,*}(x,t),V_{i,*}(x,t))
\enspace \text{uniformly in} \enspace \bar{\Omega}\times[0,1].
\end{equation*}

On the other hand, by Theorem \ref{T Zika 1}, $V_{i,*}(x,t)<V_{*}(x,t)$ in $\bar\Omega \times \mathbb{R}$.
It follows that there exists $0<\delta_{0}\le\varepsilon_{0}$ such that
$V_{i,*}(x,t)<V_{*}(x,t)-2\delta$ in $\bar\Omega \times \mathbb{R}$ for all $0<\delta<\delta_{0}$.
Let us fix $0<\delta< \delta_{0}$.
By \eqref{4-6}, there exists $T_{\delta} > 0$, such that
\begin{equation*}
V_{i}(x,t) \leq V_{i,*}(x,t) + \delta < V_{*}(x,t) - \delta \enspace\text{for all}\enspace  x\in\bar{\Omega}, t\geq T_{\delta}.
\end{equation*}
and
\begin{equation*}
0<V_{*}(x,t)-\delta\le V(x,t) \le V_{*}(x,t)+\delta \enspace\text{for all}\enspace  x\in\bar{\Omega}, t\geq T_{\delta}.
\end{equation*}
Hence, $(H_{i},V_{i})$ satisfies
\begin{equation*}
\left\{
\begin{aligned}
 \tau \partial_{t} H_{i}
&=d_{1}(x,t) \mathcal{K}^{1}_{\Omega,\sigma,m}[H_{i}](x,t)-\rho(x,t)H_{i}+\sigma_{1}(x,t)H_{u}(x)V_{i}, &x&\in \bar\Omega, t>T_{\delta},\\
 \tau \partial_{t} V_{i}
&\geq d_{2}(x,t)\mathcal{K}^{2}_{\Omega,\sigma,m}[V_{i}](x,t)+\sigma_{2}(x,t)(V_{*}(x,t)-\delta-V_{i})H_{i}\\
&\hspace{0.5cm} -\mu(x,t)(V_{*}(x,t)+\delta)V_{i}\\
&=d_{2}(x,t)\mathcal{K}^{2}_{\Omega,\sigma,m}[V_{i}](x,t)+\sigma_{2}(x,t)(V_{*}(x,t)-\delta-V_{i})^{+}H_{i}\\
&\hspace{0.5cm} -\mu(x,t)(V_{*}(x,t)+\delta)V_{i}, &x&\in \bar\Omega, t>T_{\delta}.
\end{aligned}
\right.
\end{equation*}
Similarly, we can prove that for $n\in\mathbb{N}^{+}$
\begin{equation*}
\liminf_{n\to\infty} (H_{i}(x,t+n),V_{i}(x,t+n)) \geq
(H_{i,*}(x,t),V_{i,*}(x,t))
\enspace \text{uniformly in} \enspace \bar{\Omega}\times[0,1].
\end{equation*}
Case (i) is proved.

(ii)
By arguments similar to those in the first case, we can choose
$0<\varepsilon_{1}\ll1$ such that $s(L_{1}(A_{\varepsilon_{1}}))<0$ and $(H_{i},V_{i})$ satisfies \begin{equation*}
\left\{
\begin{aligned}
& \tau \partial_{t} H_{i}
= d_{1}(x,t)\mathcal{K}^{1}_{\Omega,\sigma,m}[H_{i}](x,t)
  -\rho(x,t)H_{i}+\sigma_{1}(x,t)H_{u}(x)V_{i},
&& x\in\bar\Omega, t>T_{\varepsilon_{1}},\\
& \tau \partial_{t} V_{i}
\le d_{2}(x,t)\mathcal{K}^{2}_{\Omega,\sigma,m}[V_{i}](x,t)
   +\sigma_{2}(x,t)(V_{*}(x,t)+\varepsilon_{1})H_{i}\\
   &\hspace{1.3cm}-\mu(x,t)(V_{*}(x,t)-\varepsilon_{1})V_{i},
&& x\in\bar\Omega, t>T_{\varepsilon_{1}},
\end{aligned}
\right.
\end{equation*}
holds for some $T_{\varepsilon_{1}}>0$.
Let $\{\Phi(t,s)\}_{t\geq s}$ be the evolution operator associated with
$L_{1}(A_{\varepsilon_{1}})$.
It suffices to prove that
\begin{equation*}
\left\|\Phi(t,0)\left(H_{i}(\cdot,T_{\varepsilon_{1}}),
V_{i}(\cdot,T_{\varepsilon_{1}})\right)
\right\|_{{C(\bar{\Omega})}^{2}} \to 0
\enspace\text{as}\enspace t\to\infty.
\end{equation*}
Write $t=[t]+c_{t}$, where $[t]$ denotes the integer part of $t$ and
$c_{t}\in[0,1)$. By the time-periodicity of the coefficients,  $\Phi(t,0)
= \Phi(c_{t},0)\,\Phi(1,0)^{[t]}$.
By Proposition \ref{PP2.10}, we have
$s\left(L(A_{\varepsilon_{1}})\right)
= \tau \ln r\bigl(\Phi(1,0)\bigr)$,
and Gelfand's formula
$r\bigl(\Phi(1,0)\bigr)
 = \lim_{n\to\infty}\|\Phi(1,0)^{n}\|^{1/n}$
then yields
\begin{equation*}
\|\Phi(1,0)^{n}\|
\le C_{1}\exp\!\left(\frac{s(L_{1}(A_{\varepsilon_{1}}))}{2\tau}\,n\right)
\enspace\text{for all}\enspace n\geq n_{0},
\end{equation*}
for some constants $C_{1}>0$ and $n_{0}\in\mathbb{N}$.
Hence,
\begin{equation*}
\left\|\Phi(t,0)\left(H_{i}(\cdot,T_{\varepsilon_{1}}),
V_{i}(\cdot,T_{\varepsilon_{1}})\right)
\right\|_{{C(\bar{\Omega})}^{2}}
\le C\exp\!\left(\frac{s(L(A_{\varepsilon_{1}}))}{2\tau}\,[t]\right)
\to 0 \enspace\text{as}\enspace t\to\infty,
\end{equation*}
which proves the desired conclusion.

(iii) If $s(L_{0})<0$, by \cite[Theorem B]{ShenVo}, there holds
\begin{equation*}
\lim_{ t \to\infty} V( \cdot,t) = 0
\enspace \text{uniformly in} \enspace \bar{\Omega}.
\end{equation*}
Hence, the desired conclusion follows by the first equation of \eqref{4-1}.
\end{proof}

Given $t\in[0,1]$, define $L_{1}^{t}$ on $C(\bar{\Omega};\mathbb{R}^{2})$ by $L_{1}^{t}[u](x) = N[u](x,t)+A_{1}(x,t)u(x)$. Now, combining Theorems \ref{th6.11}, \ref{THM1.7}, \ref{THM1.8}, and \ref{THM1.10}, we can obtain the following result.
\begin{corollary} The following results hold.
\begin{itemize}[leftmargin=0.8cm]
\item[(i)]
Suppose $ \sup_{\bar{\Omega}}\lambda_{1}(A(x)) > 0$, then there exists $d^{1}>0$ such that \eqref{4-2} admits a unique positive solution that is globally asymptotically stable if $0< \max_{\bar{\Omega}\times [0,1]}d_{i}(x,t) < d^{1}$ for $i=1,2$.

\item[(ii)]
There exists $d^{2}>0$, such that \eqref{4-2} admits no positive solution and the zero solution is globally asymptotically stable if $\min_{\bar{\Omega}\times[0,1]}d_{i}(x,t)>d^{2}$ for $i=1,2$.

\item[(iii)]
Suppose $m \in [0,2)$, $d_{i}(i=1,2)$ are constants and
$\sup_{\bar{\Omega}}\lambda_{1}(A(x)) > 0$, then there exists $\sigma_{1}>0$ such that \eqref{4-2} admits a unique positive solution that is globally asymptotically stable for all $0< \sigma < \sigma_{1}$.

\item[(iv)]
Suppose $m > 0$ and $ \sup_{\bar{\Omega}}\lambda_{1}(A(x)) > 0$ or $m=0$ and $ \sup_{\bar{\Omega}}\lambda_{1}(A_{1}(x)) > 0$ , then there exists $\sigma_{2}>0$ such that \eqref{4-2} admits a unique positive solution that is globally asymptotically stable for all $\sigma > \sigma_{2}$.

\item[(v)]
Suppose $s(L_{0})>0$, $s(L_{1}^{t})$ is the principal eigenvalue of $L_{1}^{t}$ and $\int_{0}^{1} s(L_{1}^{t}) dt > 0$, then there exists $\tau_{1}>0$ such that  equation \eqref{4-2} admits a unique positive solution that is globally asymptotically stable for all $0< \tau < \tau_{1}$.
\end{itemize}
\end{corollary}

In what follows, we prove that, when $m=2$ and $\sigma\to 0^{+}$, the solution of the initial value problem \eqref{4-1} and the positive periodic solution
$(H_{i,*},V_{*}-V_{i,*},V_{i,*})$ converge to the solutions of the corresponding local dispersal system.

To ensure the well-posedness of the local dispersal system, we assume further in this subsection that $H_{u}\in C^{\alpha}(\mathbb{R}^{N})$  and  $ \rho, \sigma_{1}, \sigma_{2}, \beta, \mu$ belong to $C^{\alpha,\frac{\alpha}{2}}(\mathbb{R}^{N} \times \mathbb{R})$ for some $0<\alpha<1$, and the initial function $(H_{i0}(x), V_{u0}(x), V_{i0}(x)) \in \{C^{3}(\bar\Omega)\}^{3}$ with $(H_{i0}(x), V_{u0}(x), V_{i0}(x))=(0,0,0)$ for $x\in\partial\Omega$. In addition, we always assume $m=2$ in this subsection.

Consider the following problem:
\begin{equation}\label{4-8}
\left\{
\begin{aligned}
&\tau \partial_{t} H_{i}
= d_{r,1}(x,t)\Delta H_{i} -\rho(x,t)H_{i}+\sigma_{1}(x,t)H_{u}(x)V_{i},& &x\in\Omega, t>0,\\
&\tau \partial_{t} V_{u}
= d_{r,2}(x,t)\Delta V_{u}  -\sigma_{2}(x,t)V_{u}H_{i}+\beta(x,t)(V_{u}+V_{i})-\mu(x,t)(V_{u}+V_{i})V_{u},&&x\in \Omega, t>0, \\
&\tau \partial_{t} V_{i}
= d_{r,2}(x,t)\Delta V_{i} + \sigma_{2}(x,t)V_{u}H_{i}-\mu(x,t)(V_{u}+V_{i})V_{i},&&x\in \Omega, t>0,\\
& (H_{i}(x,t), V_{u}(x,t), V_{i}(x,t))=(0, 0, 0),&&x\in \partial\Omega, t>0, \\
&(H_{i}(x,0), V_{u}(x,0), V_{i}(x,0))=(H_{i0}(x), V_{u0}(x), V_{i0}(x)),
&&x\in \bar\Omega,
\end{aligned}
\right.
\end{equation}
where $d_{r,j}(x,t)= \frac{d_{j}(x,t)}{2N} \int_{\mathbb{R}^{N}} k_{j}(z,t)|z|^{2}dz$ for $j=1,2$.
By the same arguments as in \cite{MWW,LZ}, \eqref{4-8} admits a unique global positive solution, denoted by
$(H_{i}^{r}(x,t;H_{i0}), V_{u}^{r}(x,t;V_{u0}), V_{i}^{r}(x,t;V_{i0}))$.
Moreover, we denote the solution of \eqref{4-1} by
$(H_{i}^{\sigma}(x,t;H_{i0}), V_{u}^{\sigma}(x,t;V_{u0}), V_{i}^{\sigma}(x,t;V_{i0}))$.
After extending this solution by $0$, it satisfies the following system.
\begin{equation*}
\left\{
\begin{aligned}
&\tau \partial_{t} H_{i} =d_{1}(x,t)\mathcal{K}^{1}_{\mathbb{R}^{N},\sigma,m}[H_{i}](x,t)-\rho(x,t)H_{i}+\sigma_{1}(x,t)H_{u}(x)V_{i}, &x&\in \bar\Omega, t>0,\\
&\tau \partial_{t} V_{u}
=d_{2}(x,t)\mathcal{K}^{2}_{\mathbb{R}^{N},\sigma,m}[V_{u}](x,t)-\sigma_{2}(x,t)V_{u}H_{i}+\beta(x,t)(V_{u}+V_{i})\\
&\hspace{1.5cm}-\mu(x,t)(V_{u}+V_{i})V_{u},&x&\in \bar\Omega, t>0, \\
&\tau \partial_{t} V_{i}
=d_{2}(x,t)\mathcal{K}^{2}_{\mathbb{R}^{N},\sigma,m}[V_{i}](x,t)+\sigma_{2}(x,t)V_{u}H_{i}-\mu(x,t)(V_{u}+V_{i})V_{i},&x&\in \bar\Omega, t>0,\\
&(H_{i}(x,t), V_{u}(x,t), V_{i}(x,t)) = (0,0,0), &x&\in \mathbb{R}^{N}\backslash\bar\Omega, t>0, \\
&(H_{i}(x,0), V_{u}(x,0), V_{i}(x,0))=(H_{i0}(x), V_{u0}(x), V_{i0}(x)),&x&\in\bar\Omega.
\end{aligned}
\right.
\end{equation*}

\begin{theorem}\label{THM8.4}
Let $T>0$ be fixed, then we have
\begin{equation*}
\begin{aligned}
&\lim_{\sigma \to 0^{+}}
(H_{i}^{\sigma}(x,t;H_{i0}), V_{u}^{\sigma}(x,t;V_{u0}), V_{i}^{\sigma}(x,t;V_{i0}))\\
=&
(H_{i}^{r}(x,t;H_{i0}), V_{u}^{r}(x,t;V_{u0}), V_{i}^{r}(x,t;V_{i0}))\enspace \text{uniformly in}\enspace  \bar{\Omega}\times[0,T].
\end{aligned}
\end{equation*}
\end{theorem}
\begin{proof}
Let $V^{r} = V_{u}^{r} + V^{r}_{i}$ and $V^{\sigma} = V_{u}^{\sigma} + V^{\sigma}_{i}$. Then $V^{r}$ satisfies
\begin{equation*}
\left\{
\begin{aligned}
&\tau\partial_{t} V^{r} = d_{r,2}(x,t)\Delta V^{r} + \beta(x,t)V^{r} - \mu(x,t) (V^{r})^{2}, &x& \in\Omega,t>0,\\
&V^{r}(x,t) = 0, &x& \in\partial\Omega,t>0,\\
&V^{r}(x,0) = V_{u0}(x) + V_{i0}(x), &x&\in\bar\Omega.
\end{aligned}
\right.
\end{equation*}
Moreover, $V^{\sigma}$ satisfies
\begin{equation*}
\left\{
\begin{aligned}
&\tau\partial_{t} V^{\sigma} = d_{2}(x,t)\mathcal{K}^{2}_{\mathbb{R}^{N},\sigma,2} [V^{\sigma}](x,t) + \beta(x,t)V^{\sigma} - \mu(x,t) (V^{\sigma})^{2}, &x&\in\bar\Omega,t>0,\\
&V^{\sigma}(x,t) = 0, &x&\in \mathbb{R}^{N}\backslash\bar\Omega, t>0,\\
&V^{\sigma}(x,0) = V_{u0}(x) + V_{i0}(x), &x&\in\bar{\Omega}.
\end{aligned}
\right.
\end{equation*}
It then follows from \cite[Theorem A]{SX} that
\begin{equation}\label{4-9}
\lim_{\sigma \to 0^{+}} V^{\sigma}(x,t) = V^{r}(x,t)
\enspace \text{uniformly in}\enspace \bar{\Omega}\times[0,T].
\end{equation}
Hence, it suffices to prove that
\begin{equation*}
\lim_{\sigma \to 0^{+}}
(H_{i}^{\sigma}(x,t;H_{i0}),  V_{i}^{\sigma}(x,t;V_{i0}))
=
(H_{i}^{r}(x,t;H_{i0}),  V_{i}^{r}(x,t;V_{i0}))
\enspace \text{uniformly in}\enspace  \bar{\Omega}\times[0,T].
\end{equation*}
Define $W_{1} = H_{i}^{r} - H_{i}^{\sigma}$ and $W_{2} = V_{i}^{r} - V_{i}^{\sigma}$, it follows that for  $t\in (0,T]$, $(W_{1},W_{2})$ satisfies
\begin{equation}\label{4-10}
\left\{
\begin{aligned}
&\tau \partial_{t} W_{1} = d_{1}(x,t)\mathcal{K}^{1}_{\mathbb{R}^{N},\sigma,2}[W_{1}] + a_{1}(x,t) - \rho(x,t)W_{1} + \sigma_{1}(x,t)H_{u}(x)W_{2},
&x&\in\bar\Omega \\
&\tau \partial_{t} W_{2}
=
d_{2}(x,t)\mathcal{K}^{2}_{\mathbb{R}^{N},\sigma,2}[W_{2}] + a_{2}(x,t) + \sigma_{2}(x,t)(V^{\sigma}(x,t)-V_{i}^{r}(x,t))W_{1} \\
&\hspace{1.5cm}+ \sigma_{2}(x,t)( V^{r}(x,t) - V^{\sigma}(x,t) )H_{i}^{r}(x,t) - \sigma_{2}(x,t)H^{\sigma}_{i}(x,t)W_{2} \\
&\hspace{1.5cm} - \mu(x,t) V^{r}(x,t) W_{2} + \mu(x,t)V_{i}^{\sigma}(x,t)(V^{\sigma}(x,t) - V^{r}(x,t)),
&x&\in\bar\Omega,\\
& (W_{1}(x,t),W_{2}(x,t) ) = (0,0), &x&\in \mathbb{R}^{N}\backslash\bar\Omega,\\
& (W_{1}(x,0), W_{2}(x,0) ) = (0,0),
&x&\in \bar{\Omega},
\end{aligned}
\right.
\end{equation}
where
\begin{equation*}
a_{1}(x,t) = d_{r,1}(x,t) \Delta H_{i}^{r}(x,t) - d_{1}(x,t)\mathcal{K}^{1}_{\mathbb{R}^{N},\sigma,2}[H_{i}^{r}](x,t),
\end{equation*}
and
\begin{equation*}
a_{2}(x,t) = d_{r,2}(x,t) \Delta V_{i}^{r}(x,t) - d_{2}(x,t)\mathcal{K}^{2}_{\mathbb{R}^{N},\sigma,2}[V_{i}^{r}](x,t).
\end{equation*}
Since $V_{i}^{r}<V^{r}$ and \eqref{4-9} holds, there exists $\sigma_{0}>0$ such that for $0<\sigma<\sigma_{0}$, $V^{\sigma}-V_{i}^{r}>0$ in $\bar{\Omega} \times [0,T]$. Hence, \eqref{4-10} is a cooperative system for $0< \sigma < \sigma_{0}$.
Then, by the similar arguments as in Theorem \ref{THM3.4}, we obtain the desired result.

\end{proof}

Next, we establish the convergence result of the positive periodic solution.
Consider the following problem:
\begin{equation}\label{4-11}
\left\{
\begin{aligned}
&\tau \partial_{t} H_{i}
= d_{r,1}(x,t)\Delta H_{i}(x,t) -\rho(x,t)H_{i}+\sigma_{1}(x,t)H_{u}(x)V_{i}, &&\text{in } \Omega\times\mathbb{R},\\
&\tau \partial_{t} V_{u}
= d_{r,2}(x,t)\Delta V_{u}(x,t)  -\sigma_{2}(x,t)V_{u}H_{i}+\beta(x,t)(V_{u}+V_{i}) \\
&\hspace{1.5cm}-\mu(x,t)(V_u+V_{i})V_{u},&&\text{in } \Omega\times\mathbb{R}, \\
&\tau \partial_{t} V_{i}
= d_{r,2}(x,t)\Delta V_{i}(x,t) + \sigma_{2}(x,t)V_{u}H_{i}-\mu(x,t)(V_{u}+V_{i})V_{i},&& \text{in } \Omega\times\mathbb{R},\\
& (H_{i}(x,t), V_{u}(x,t), V_{i}(x,t))=(0, 0, 0),&&\text{on } \partial\Omega\times\mathbb{R}, \\
&(H_{i}(x,0), V_{u}(x,0), V_{i}(x,0))=(H_{i}(x,1), V_{u}(x,1), V_{i}(x,1)),&&\text{in } \bar\Omega\times\mathbb{R}.
\end{aligned}
\right.
\end{equation}
Let $V=V_{u}+V_{i}$, it follows that
\begin{equation}\label{4-12}
\left\{
\begin{aligned}
&\tau\partial_{t} V = d_{r,2}(x,t)\Delta V(x,t) + \beta(x,t)V(x,t) - \mu(x,t) V^{2}(x,t), && \text{in } \Omega\times\mathbb{R},\\
&V(x,t) = 0, && \text{on } \partial\Omega\times\mathbb{R},\\
&V(x,t) = V(x,t+1), && \text{in } \Omega\times\mathbb{R}.
\end{aligned}
\right.
\end{equation}
Let $\lambda_{0}^{r}$ be the principal eigenvalue of the following problem
\begin{equation*}
\left\{
\begin{aligned}
& -\tau\partial_{t} \varphi(x,t) + d_{r,2}(x,t)\Delta \varphi(x,t) + \beta(x,t)\varphi(x,t) = \lambda \varphi(x,t) && \text{in } \Omega\times\mathbb{R},\\
&\varphi(x,t) = 0 && \text{on } \partial\Omega\times\mathbb{R},\\
&\varphi(x,t) = \varphi(x,t+1) && \text{in } \Omega\times\mathbb{R},.
\end{aligned}
\right.
\end{equation*}
If $\lambda_{0}^{r}>0$,
by a standard result on the periodic parabolic logistic equations \cite[Theorem 3.1.5]{Zhao}, \eqref{4-12} admits a globally stable unique solution, denoted by $V^{r}_{*}$.
Moreover, by \cite[Theorems B and C]{SX}, we have
\begin{equation}\label{lambdar}
\lim_{\sigma\to 0^{+}} s(L_{0}) = \lambda_{0}^{r}
\end{equation}
and
\begin{equation}\label{4-14}
\lim_{\sigma\to 0^{+}} V_{*}^{\sigma}(x,t) = V^{r}_{*}(x,t) \enspace \text{uniformly in}\enspace \bar{\Omega}\times\mathbb{R}.
\end{equation}

Let $\lambda_{1}^{r}$ be the principal eigenvalue of the following problem:
\begin{equation}\label{4-15}
\left\{
\begin{aligned}
& -\tau \partial_{t} \varphi_{1} + d_{r,1}(x,t)\Delta \varphi_{1}(x,t)-\rho(x,t)\varphi_{1}+\sigma_{1}(x,t)H_{u}(x)\varphi_{2} = \lambda \varphi_{1},&& \text{in } \Omega\times\mathbb{R},
\\
&- \tau \partial_{t} \varphi_{2} +d_{r,2}(x,t)\Delta \varphi_{2}(x,t)+\sigma_{2}(x,t)V^{r}_{*}(x,t)\varphi_{1}-\mu(x,t)V^{r}_{*}(x,t)\varphi_{2} = \lambda \varphi_{2},&& \text{in } \Omega\times\mathbb{R},
\\
&( \varphi_{1}(x,t), \varphi_{2}(x,t)) = ( 0, 0),&& \text{on } \partial\Omega\times\mathbb{R},
\\
&( \varphi_{1}(x,t), \varphi_{2}(x,t)) = ( \varphi_{1}(x,t+1), \varphi_{2}(x,t+1)),&& \text{in } \Omega\times\mathbb{R},
\end{aligned}
\right.
\end{equation}
Following same arguments as in the proof of Theorem \ref{T Zika 1}, we can obtain that if $\lambda_{0}^{r}>0$ and $\lambda_{1}^{r}>0$, \eqref{4-11} admits a unique positive solution $({H}_{i,*}^{r},V^{r}_{*}-{V}^{r}_{i,*},{V}^{r}_{i,*})$, which is globally asymptotically stable. Moreover,
by Theorem \ref{THM1.8} (ii) and \eqref{4-14}, we have
\begin{equation*}
\lim_{\sigma \to 0^{+}} s(L_{1}) = \lambda_{1}^{r}.
\end{equation*}
This together with Theorem \ref{th6.11} implies that if $\lambda_{0}^{r}>0$ and $\lambda_{1}^{r}>0$, there exists $\sigma_{0}>0$ such that \eqref{4-2} admits a globally asymptotically stable solution $(H_{i,*}^{\sigma},V_{*}^{\sigma}-V_{i,*}^{\sigma},V_{i,*}^{\sigma})$ for $0<\sigma<\sigma_{0}$.

Recall that  $\Omega_{n} = \{ x\in\mathbb{R}^{N} \,|\, dist(x,\Omega) < \frac{1}{n} \} $ and $\Omega_{n}$ has the same boundary regularity as $\Omega$ for $n\geq n_{0}$.
If $\lambda_{0}^{r}>0$, by Theorem \ref{LEM3.3}, we can choose $n_{0}$ larger if necessary, such that for $n \geq n_{0}$ the following problem admits a unique positive solution, denoted by $V^{r,n}_{*}$.
\begin{equation*}
\left\{
\begin{aligned}
&\tau\partial_{t} V = d_{r,2}(x,t)\Delta V(x,t) + \beta(x,t)V(x,t) - \mu(x,t) V^{2}(x,t) && \text{in } \Omega_{n}\times\mathbb{R},\\
&V(x,t) = 0  &&\text{on } \partial\Omega_{n}\times\mathbb{R},\\
&V(x,t) = V(x,t+1) &&\text{in } \Omega_{n}\times\mathbb{R}.
\end{aligned}
\right.
\end{equation*}
By the same arguments as in the proof of \cite[Theorem C]{SX}, we have
\begin{equation*}
\lim_{n\to\infty} V^{r,n}_{*}(x,t) = V^{r}_{*}(x,t) \enspace \text{uniformly in}\enspace \bar{\Omega}\times\mathbb{R}.
\end{equation*}
Furthermore, if $\lambda_{1}^{r}>0$, by Theorem \ref{LEM3.3} and the perturbation theory of isolated eigenvalue (see \cite{Kato}), we can choose $n_{0}$ larger if necessary such that $\lambda_{1}^{r,n}>0$ for $n\geq n_{0}$, where $\lambda_{1}^{r,n}$ is the principal eigenvalue of the following problem:
\begin{equation*}
\left\{
\begin{aligned}
& -\tau \partial_{t} \varphi_{1} + d_{r,1}(x,t)\Delta \varphi_{1}(x,t)-\rho(x,t)\varphi_{1}+\sigma_{1}(x,t)H_{u}(x)\varphi_{2} = \lambda \varphi_{1} && \text{in } \Omega_{n}\times\mathbb{R},
\\
&- \tau \partial_{t} \varphi_{2} + d_{r,2}(x,t)\Delta \varphi_{2}(x,t)+\sigma_{2}(x,t)V_{*}^{r,n}(x,t)\varphi_{1}-\mu(x,t)V_{*}^{r,n}(x,t)\varphi_{2} = \lambda \varphi_{2} && \text{in } \Omega_{n}\times\mathbb{R},
\\
&( \varphi_{1}(x,t), \varphi_{2}(x,t)) = ( 0, 0) && \text{on } \partial\Omega_{n}\times\mathbb{R},
\\
&( \varphi_{1}(x,t), \varphi_{2}(x,t)) = ( \varphi_{1}(x,t+1), \varphi_{2}(x,t+1)) && \text{in } \Omega_{n}\times\mathbb{R},
\end{aligned}
\right.
\end{equation*}
It follows that for $n \geq n_{0}$, there exists $( H_{i,*}^{r,n}, V_{i,*}^{r,n})$ satisfies
\begin{equation}\label{4-16}
\left\{
\begin{aligned}
&\tau \partial_{t} H_{i,*}^{r,n}
= d_{r,1}(x,t)\Delta H_{i,*}^{r,n} (x,t)-\rho(x,t)H_{i,*}^{r,n} +\sigma_{1}(x,t)H_{u}(x)V_{i,*}^{r,n} && \text{in } \Omega_{n}\times\mathbb{R},
\\
&\tau \partial_{t} V_{i,*}^{r,n}
= d_{r,2}(x,t)\Delta V_{i,*}^{r,n} (x,t) + \sigma_{2}(x,t)(V^{r,n}_{*}(x,t) - V_{i,*}^{r,n})H_{i,*}^{r,n}\\ &\hspace{1.7cm}-\mu(x,t)V^{r,n}_{*}(x,t)V_{i,*}^{r,n} && \text{in } \Omega_{n}\times\mathbb{R},
\\
&( H_{i,*}^{r,n}(x,t), V_{i,*}^{r,n}(x,t)) = ( 0, 0) && \text{on } \partial\Omega_{n}\times\mathbb{R},
\\
&(H_{i,*}^{r,n}(x,t), V_{i,*}^{r,n}(x,t))=( H_{i,*}^{r,n}(x,t+1), V_{i,*}^{r,n}(x,t+1) ) && \text{in } \Omega_{n}\times\mathbb{R}.
\end{aligned}
\right.
\end{equation}
By the comparison principle, it is easy to check that
\begin{equation*}
( H_{i,*}^{r,n}(x,t), V_{i,*}^{r,n}(x,t) )
\leq
( H_{*}^{r,n_{0}}(x,t), V_{*}^{r,n_{0}}(x,t) )
\enspace\text{for all}\enspace (x,t)\in\bar{\Omega}\times\mathbb{R}
\enspace\text{and}\enspace n\geq n_{0},
\end{equation*}
where $H_{*}^{r,n_{0}}(x,t)$ is the unique positive solution of the following problem:
\begin{equation*}
\left\{
\begin{aligned}
&\tau\partial_{t} H(x,t) = d_{r,1}(x,t)\Delta H(x,t) - \rho(x,t)H(x,t) + \sigma_{1}(x,t)H_{u}(x)V^{r,n_{0}}_{*}(x,t)  && \text{in } \Omega_{n_{0}}\times\mathbb{R},\\
&H(x,t) = 0&& \text{on } \partial\Omega_{n_{0}}\times\mathbb{R},\\
&H(x,t) = H(x,t+1) && \text{in } \Omega_{n_{0}}\times\mathbb{R}.
\end{aligned}
\right.
\end{equation*}
As a result, sequences $\{ H_{i,*}^{r,n}(x,t) \}_{n\geq n_{0}} $ and $ \{V_{i,*}^{r,n}(x,t) \}_{n\geq n_{0}}$ are uniformly bounded.
\begin{lemma}\label{LEM8.5}
Suppose $\lambda_{0}^{r}>0$, $\lambda_{1}^{r}>0$ and $n\geq n_{0}$,
there holds
\begin{equation*}
\lim_{n\to\infty} ( H_{i,*}^{r,n}(x,t), V_{i,*}^{r,n}(x,t) ) = ( H_{i,*}^{r}(x,t), V_{i,*}^{r}(x,t)) \enspace\text{uniformly in}\enspace \bar{\Omega}\times\mathbb{R}.
\end{equation*}
\end{lemma}
\begin{proof}
Since $\|H_{i,*}^{r,n}\|_{\infty}$ and $\|V_{i,*}^{r,n}\|_{\infty}$ are uniformly bounded, by the standard interior estimates,
\begin{equation*}
\|H_{i,*}^{r,n}\|_{C_{loc}^{2+\alpha,1+\frac{\alpha}{2}}(\Omega\times [0,1]) }
\enspace \text{and} \enspace
\|V_{i,*}^{r,n}\|_{C_{loc}^{2+\alpha,1+\frac{\alpha}{2}}(\Omega\times [0,1]) }
\end{equation*}
are uniformly bounded.
Hence, by a diagonal argument, up to a subsequence, there exists a function $(\bar{H}_{i}^{r},\bar{V}_{i}^{r})$ such that
\begin{equation}\label{4-17}
(H_{i,*}^{r,n}(x,t),V_{i,*}^{r,n}(x,t)) \to (\bar{H}_{i}^{r}(x,t),\bar{V}_{i}^{r}(x,t))
\enspace\text{in}\enspace C_{loc}^{2,1}(\Omega\times[0,1]).
\end{equation}
Moreover, since $\Omega_{n}$ has the same boundary regularity as $\Omega$ for $n\geq n_{0}$.
We can conclude that
\begin{equation*}
\|H_{i,*}^{r,n}\|_{C^{\alpha,\frac{\alpha}{2}}(\bar{\Omega}_{n}\times [0,1]) }
\enspace\text{and}\enspace
\|V_{i,*}^{r,n}\|_{C^{\alpha,\frac{\alpha}{2}}(\bar{\Omega}_{n}\times [0,1]) }
\end{equation*}
are uniformly bounded
(see \cite[Remark 4.6 (b)]{Dan}, \cite[Theorem D]{Aronson}, \cite[Corollary 7.41]{Lie}).
It follows that the family $\{(H_{i,*}^{r,n},V_{i,*}^{r,n})\}_{n\geq n_{0}}$ is relatively compact in
$C(\bar{\Omega}\times[0,1])$.
Hence, up to a subsequence,
\begin{equation}\label{4-18}
(H_{i,*}^{r,n}(x,t),V_{i,*}^{r,n}(x,t)) \to (\bar{H}_{i}^{r}(x,t),\bar{V}_{i}^{r}(x,t))
\enspace\text{uniformly in}\enspace \bar{\Omega}\times[0,1].
\end{equation}
It follows that
$(\bar{H}_{i}^{r},\bar{V}_{i}^{r})$ is a $1$-periodic function.
We next show that $(\bar H_i^r,\bar V_i^r)=(0,0)$ on $\partial\Omega$.
Fix $x\in\partial\Omega$. Choose $x_n\in\partial\Omega_n$ such that $|x-x_n|\le 1/n$.
Since $H_{i,*}^{r,n}=V_{i,*}^{r,n}=0$ on $\partial\Omega_n$ and the above H\"older norms are uniformly
bounded, there exists $C>0$ independent of $n$ such that for $t\in [0,1]$,
\begin{equation*}
|H_{i,*}^{r,n}(x,t)|
=|H_{i,*}^{r,n}(x,t)-H_{i,*}^{r,n}(x_n,t)|
\le C|x-x_n|^{\alpha}\to0,
\end{equation*}
and similarly $|V_{i,*}^{r,n}(x,t)|\to0$. Passing to the limit in \eqref{4-18} yields
$\bar H_i^r(x,t)=\bar V_i^r(x,t)=0$ for all $x\in\partial\Omega$ and $t\in[0,1]$.
Finally, let $\Omega'\Subset\Omega$ be arbitrary. Passing to the limit in \eqref{4-16} on
$\Omega'\times(0,1)$ is justified by \eqref{4-17} and by
$V_*^{r,n}\to V_*^{r}$ uniformly on $\bar\Omega\times\mathbb{R}$.
Since $\Omega'\Subset\Omega$ is arbitrary, we conclude that
$(\bar{H}_{i}^{r},\bar{V}_{i}^{r})$ satisfies \eqref{4-11}.

Next, we prove that $\bar{H}_{i}^{r}$ and $\bar{V}_{i}^{r}$ are positive. By the strong maximum principle, it suffices to show that
$(\bar{H}_{i}^{r},\bar{V}_{i}^{r})\not\equiv(0,0)$.
Assume by contradiction that $(\bar{H}_{i}^{r},\bar{V}_{i}^{r})\equiv(0,0)$.
Define
\begin{equation*}
M_{n}=\max_{(x,t)\in\bar{\Omega}\times[0,1]}
\big(H_{i,*}^{r,n}(x,t)+V_{i,*}^{r,n}(x,t)\big),
\enspace
w_{1,n}=\frac{H_{i,*}^{r,n}}{M_{n}},\enspace
w_{2,n}=\frac{V_{i,*}^{r,n}}{M_{n}},
\end{equation*}
and set $w_{n}=(w_{1,n},w_{2,n})$.
Then $M_{n}\to 0$ as $n\to\infty$, $w_{1,n}>0$, $w_{2,n}>0$, and $w_{n}$ is $1$-periodic in $t$.
Arguing as above, up to a subsequence, $w_{n}\to w=(w_{1},w_{2})$ uniformly in
$\bar{\Omega}\times[0,1]$, where $w\not\equiv(0,0)$ and $w$ is $1$-periodic in $t$.
Moreover, $w$ satisfies
\begin{equation*}
\left\{
\begin{aligned}
&\tau \partial_{t} w_{1} = d_{r,1}(x,t)\Delta w_{1}-\rho(x,t) w_{1}(x,t)+\sigma_{1}(x,t)H_{u}(x)w_{2}(x,t) && \text{in } \Omega\times\mathbb{R},
\\
&\tau \partial_{t} w_{2} = d_{r,2}(x,t)\Delta w_{2}+\sigma_{2}(x,t)V^{r}_{*}(x,t)w_{1}(x,t)-\mu(x,t)V^{r}_{*}(x,t)w_{2}(x,t) && \text{in } \Omega\times\mathbb{R},
\\
&(w_{1}(x,t),w_{2}(x,t)) = (0,0) && \text{on } \partial\Omega\times\mathbb{R},
\\
&(w_{1}(x,t),w_{2}(x,t)) = (w_{1}(x,t+1),w_{2}(x,t+1)) && \text{in } \Omega\times\mathbb{R}.
\end{aligned}
\right.
\end{equation*}
This yields $\lambda_{1}^{r} = 0$, a contradiction. The desired result then follows from the uniqueness of \eqref{4-11}.
\end{proof}

Consider the following problem:
\begin{equation}\label{4-19}
\left\{
\begin{aligned}
&\tau\partial_{t} V = d_{2} \mathcal{K}^{n}_{2,\sigma,2} [V] (x,t) + \beta(x,t)V(x,t) - \mu(x,t) V^{2}(x,t) && \text{in } \Omega_{n}\times\mathbb{R},\\
&V(x,t) = V(x,t+1) && \text{in } \Omega_{n}\times\mathbb{R}.
\end{aligned}
\right.
\end{equation}
Suppose $\lambda_{0}^{r}>0$.
By Theorem \ref{LEM3.3} and equation \eqref{lambdar},  there exists  $0<\sigma_{0}\ll 1$, such that for $0<\sigma\leq\sigma_{0}$, \eqref{4-19} admits a unique positive solution, denoted by $V_{*}^{\sigma,n}$.

\begin{lemma}\label{LEM8.8}
Suppose $\lambda_{0}^{r}>0$.
There exists $0< \sigma_{*} \leq \sigma_{0}$ such that for $0<\sigma < \sigma_{0}$,
\begin{equation*}
\lim_{n\to\infty} V_{*}^{\sigma,n}(x,t) = V^{\sigma}_{*}(x,t) \enspace \text{uniformly in}\enspace \bar{\Omega}\times [0,1].
\end{equation*}
\end{lemma}
\begin{proof}
Let $c_{0}>0$ be a constant and choose $0< \sigma_{*} \leq \sigma_{0}$ such that $\frac{d_{2}}{\sigma^{2}_{*}} - \max_{\bar{\Omega}\times [0,1]}\beta(x,t) \geq c_{0} > 0$.
Let us fix $0< \sigma < \sigma_{*}$.
By the comparison principle, we obtain that $\| V_{*}^{\sigma,n} \|_{\infty}$ is uniformly bounded. Hence, $\| \partial_{t}V_{*}^{\sigma,n} \|_{\infty}$ is also uniformly bounded. Next, we prove that $V_{*}^{\sigma,n}(x,t)$ is equicontinuous in $x$.
Given $h\in\mathbb{R}^{N}$ small enough,
define
\begin{equation*}
\Delta_{h} V_{n}(x,t) = V_{*}^{\sigma,n}(x+h,t) - V_{*}^{\sigma,n}(x,t)
\enspace \text{on} \enspace \Omega_{h},
\end{equation*}
where $\Omega_{h} = \{ x\in \Omega \,|\, x+h\in\Omega \}$.
A direct computation yields that
\begin{equation}\label{4-20}
\tau \partial_{t} \Delta_{h} V_{n}(x,t) = \alpha_{n}(x,t)\Delta_{h} V_{n}(x,t) + R_{n}(x,t),
\end{equation}
where $\alpha_{n} = \left( -\frac{d_{2}}{\sigma^{2}} + \beta(x+h,t) - \mu(x+h,t)(V_{*}^{\sigma,n}(x+h,t) + V_{*}^{\sigma,n}(x,t) )  \right)$ and
\begin{align*}
R_{n}(x,t) =&(\beta(x+h,t)-\beta(x,t))V_{*}^{\sigma,n}(x,t) - (\mu(x+h,t)-\mu(x,t))(V_{*}^{\sigma,n})^{2}(x,t)\\
&+\frac{d_{2}}{\sigma^{2}} \int_{\Omega_{n}} \left( k_{2,\sigma}\left (x-y+h,t \right) - k_{2,\sigma}(x-y,t)  \right)V_{*}^{\sigma,n}(y,t)dy.
\end{align*}
Notice that $\alpha_{n} \leq -c_{0} < 0$, and there exists $C>0$ independent of $n$ such that
$$\max_{ \bar{\Omega}_{h} \times [0,1]}  |R_{n}(x,t)| \leq C\omega(|h|)$$ by the boundedness of $\{V_{*}^{\sigma,n} \}$ and the compact support of $k_{2}$, where $\omega(r) \to 0$ as $r\to0$ is modulus of continuity depending on $k_{2},\beta,\mu$.
Hence, by taking the maximum and minimum points of $\Delta_{h} V_{n}$ in \ref{4-20}, we get
$\max_{ \bar{\Omega}_{h} \times [0,1]} |\Delta_{h} V_{n}(x,t) | \leq \frac{C}{c_{0}}\omega(|h|)$. Thus $\{ V_{*}^{\sigma,n} \}$ is equicontinuous on $\bar{\Omega} \times [0,1]$. It then follows from the Arzel\`a--Ascoli theorem that there exists a subsequence $V_{*}^{\sigma,n_{k}}$ such that $V_{*}^{\sigma,n_{k}} \to \bar{V}$ in $C(\bar{\Omega}\times[0,1])$ as $k\to\infty$ and $\bar{V}$ satisfies \eqref{4-3}. Moreover, following the same arguments as in the proof of Lemma \ref{LEM8.5},  $\bar{V}$ is positive. On the other hand, the comparison principle gives $V_{*}^{\sigma,n+1} \leq V_{*}^{\sigma,n}$ in $\bar{\Omega}\times[0,1]$. Hence, the desired conclusion follows from Dini's theorem and the uniqueness of \eqref{4-3}.
\end{proof}

\begin{theorem}
Suppose $\lambda_{0}^{r}>0$, $\lambda_{1}^{r}>0$ and  $m=2$. There holds
\begin{equation*}
\begin{aligned}
&\lim_{\sigma \to 0^{+}} \left(H^{\sigma}_{i,*}(x,t),V^{\sigma}_{*}(x,t)-V^{\sigma}_{i,*}(x,t),V^{\sigma}_{i,*}(x,t)\right) \\
=&
(H_{i,*}^{r}(x,t), V^{r}_{*}(x,t)- V_{i,*}^{r}(x,t), V_{i,*}^{r}(x,t))\enspace \text{uniformly in}\enspace  (x,t)\in\bar{\Omega}\times\mathbb{R}.
\end{aligned}
\end{equation*}
\end{theorem}

\begin{proof}
Since
\begin{equation*}
\lim_{\sigma \to 0^{+}} V^{\sigma}_{*}(x,t) = V^{r}_{*}(x,t) \enspace \text{uniformly in}\enspace \bar{\Omega}\times\mathbb{R}.
\end{equation*}
It is sufficient to show that
\begin{equation*}
\lim_{\sigma \to 0^{+}} \left(H^{\sigma}_{i,*}(x,t),V^{\sigma}_{i,*}(x,t)\right)
=
( H_{i,*}^{r}(x,t), V_{i,*}^{r}(x,t))\enspace \text{uniformly in}\enspace  \bar{\Omega}\times\mathbb{R}.
\end{equation*}
Choose a smooth nonnegative nontrivial function $(H_{-},V_{-})$ defined on $\bar{\Omega}$ with $(H_{-},V_{-})=(0,0)$ on $\partial\Omega$ satisfying
\begin{equation*}
(H_{-}(x),V_{-}(x)) < \left( H_{i,*}^{r}(x,0), V_{i,*}^{r}(x,0) \right)
\enspace \text{for all} \enspace x\in\bar{\Omega}.
\end{equation*}
Then, using a similar argument as in the proof of  \cite[Theorem C]{SX}, we can obtain that for any $\varepsilon>0$, there exists $\tilde\sigma_{1}>0$  such that for $0< \sigma < \tilde\sigma_{1}$,
\begin{equation*}
(H_{i,*}^{r}(x,t), V_{i,*}^{r}(x,t))
\leq
(H_{i,*}^{\sigma}(x,t) + \varepsilon, V_{i,*}^{\sigma}(x,t) + \varepsilon)
\enspace \text{for all} \enspace (x,t)\in\bar{\Omega}\times\mathbb{R}.
\end{equation*}

Next, we prove that for any $\varepsilon>0$, there exists $\tilde\sigma_{2}>0$ such that for $0< \sigma < \tilde\sigma_{2}$,
\begin{equation}\label{4-21}
(H_{i,*}^{r}(x,t), V_{i,*}^{r}(x,t))
\geq
(H_{i,*}^{\sigma}(x,t) - \varepsilon, V_{i,*}^{\sigma}(x,t) - \varepsilon)
\enspace \text{for all} \enspace (x,t)\in\bar{\Omega}\times\mathbb{R}.
\end{equation}
By Lemma \ref{LEM8.5}, there exists $n_{0}>0$ such that
\begin{equation}\label{4-22}
(H_{i,*}^{r}(x,t), V_{i,*}^{r}(x,t) )
\geq
\left(H_{i,*}^{r,n_{0}}(x,t) - \frac{\varepsilon}{4}, V_{i,*}^{r,n_{0}}(x,t)- \frac{\varepsilon}{4} \right)
\enspace \text{for all} \enspace (x,t)\in\bar{\Omega}\times\mathbb{R}.
\end{equation}
Next, we choose $M\gg 1$ such that for $0< \sigma \leq 1$,
\begin{equation}\label{4-23}
M\left(H_{i,*}^{r,n_{0}}(x,t), V_{i,*}^{r,n_{0}}(x,t) \right)
\geq
\left(H_{i,*}^{\sigma}(x,t) , V_{i,*}^{\sigma}(x,t) \right) + 1
\enspace \text{for all} \enspace (x,t)\in\bar{\Omega}\times\mathbb{R}.
\end{equation}
Define $$\left( H_{+}^{n_{0}}, V_{+}^{n_{0}} \right) = \left(MH_{i,*}^{r,n_{0}}(x,0), MV_{i,*}^{r,n_{0}}(x,0) \right),\, ( H_{+}, V_{+} ) = \left.\left(MH_{i,*}^{r,n_{0}}(x,0), MV_{i,*}^{r,n_{0}}(x,0) \right)\right|_{\Omega}.$$
Since $\left( H_{i,*}^{r,n_{0}}, V_{i,*}^{r,n_{0}} \right)$ is globally asymptotically stable, there exists $T_{1}\gg 1$, such that
\begin{equation}\label{4-24}
\left( H_{i,*}^{r,n_{0}}(x,t), V_{i,*}^{r,n_{0}}(x,t) \right)
\geq
\left( H_{i}^{r,n_{0}}(x,T_{1}+t;H_{+}^{n_{0}}) - \frac{\varepsilon}{4} ,
V_{i}^{r,n_{0}}(x,T_{1}+t;V_{+}^{n_{0}}) - \frac{\varepsilon}{4} \right)
\end{equation}
for all $(x,t)\in\bar{\Omega}\times [0,1]$.
It then follows from Theorem \ref{THM8.4} that there exists $\sigma_{1}>0$ such that for $0<\sigma<\sigma_{1}$
\begin{equation}\label{4-25}
\begin{aligned}
&\left( H_{i}^{r,n_{0}}(x,T_{1}+t;H_{+}^{n_{0}})  ,
V_{i}^{r,n_{0}}(x,T_{1}+t;V_{+}^{n_{0}}) \right)\\
\geq&
\left( H_{i}^{\sigma,n_{0}}(x,T_{1}+t;H_{+}^{n_{0}}) - \frac{\varepsilon}{4} ,
V_{i}^{\sigma,n_{0}}(x,T_{1}+t;V_{+}^{n_{0}}) - \frac{\varepsilon}{4} \right).
\end{aligned}
\end{equation}
for all $(x,t)\in\bar{\Omega}_{n_{0}} \times [0,1]$.

Notice that $\lim_{N \to \infty} V^{\sigma,n_{0}}(x,N+t) = V_{*}^{\sigma,n_{0}}(x,t)$ uniformly in $\bar{\Omega} \times [0,1]$. By same arguments as in the proof of Theorem \ref{th6.11}, for any $\gamma>0$, we can choose $T_{1}$ larger if necessary, such that $H_{i}^{\sigma,n_{0}}(x,t;H_{+})$ and $V_{i}^{\sigma,n_{0}}(x,t;V_{+})$ satisfies
\begin{equation*}
\left\{
\begin{aligned}
 \tau \partial_{t} H_{i}^{\sigma,n_{0}}
=&d_{1}(x,t) \mathcal{K}^{1}_{\Omega_{n_{0}},\sigma,2}[H_{i}^{\sigma,n_{0}}]-\rho(x,t)H_{i}^{\sigma,n_{0}}+\sigma_{1}(x,t)H_{u}(x)V_{i}^{\sigma,n_{0}}, &x&\in\bar\Omega_{n_{0}}, t>T_{1},
\\
 \tau \partial_{t} V_{i}^{\sigma,n_{0}}
\geq &d_{2}(x,t) \mathcal{K}^{2}_{\Omega_{n_{0}},\sigma,2}[V_{i}^{\sigma,n_{0}}]+\sigma_{2}(x,t)(V_{*}^{\sigma,n_{0}}(x,t)-\gamma-V_{i}^{\sigma,n_{0}})^{+}H_{i}^{\sigma,n_{0}}
\\
\quad &-\mu(x,t)(V_{*}^{\sigma,n_{0}}(x,t)+\gamma)V_{i}^{\sigma,n_{0}}, &x&\in\bar\Omega_{n_{0}}, t>T_{1}.
\end{aligned}
\right.
\end{equation*}
Therefore,
\begin{equation}\label{4-26}
\begin{aligned}
&\left(H_{i}^{\sigma,n_{0}}(x,T_{1}+t;H_{+}^{n_{0}}), V_{i}^{\sigma,n_{0}}(x,T_{1}+t;V_{+}^{n_{0}})\right)\\
\geq&
\left(H_{i,\gamma}^{\sigma,n_{0}}(x,T_{1}+t;H_{+}^{n_{0}}), V_{i,\gamma}^{\sigma, n_{0}}(x,T_{1}+t;V_{+}^{n_{0}})\right)
\end{aligned}
\end{equation}
for all $(x,t)\in\bar{\Omega}_{n_{0}}\times[0,1]$, where $\left(H_{i,\gamma}^{\sigma,n_{0}}(x,t;H_{+}^{n_{0}}), V_{i,\gamma}^{\sigma,n_{0}}(x,t;V_{+}^{n_{0}})\right)$ satisfies
\begin{equation*}
\left\{
\begin{aligned}
 &\tau \partial_{t} H_{i,\gamma}^{\sigma,n_{0}}
=d_{1} \mathcal{K}^{1}_{\Omega_{n_{0}},\sigma,2}[H_{i,\gamma}^{\sigma,n_{0}}]-\rho(x,t)H_{i,\gamma}^{\sigma,n_{0}} + \sigma_{1}(x,t)H_{u}(x)V_{i,\gamma}^{\sigma,n_{0}},\quad &&x\in\bar\Omega_{n_{0}}, t>0,
\\
 &\tau \partial_{t} V_{i,\gamma}^{\sigma, n_{0}}
= d_{2}\mathcal{K}^{2}_{\Omega_{n_{0}},\sigma,2}[V_{i,\gamma}^{\sigma,n_{0}}]+\sigma_{2}(x,t)(V_{*}^{\sigma, n_{0}}(x,t)-\gamma-V_{i,\gamma}^{\sigma,n_{0}})^{+}H_{i,\gamma}^{\sigma,n_{0}}
\\
&\hspace{1.7cm} -\mu(x,t)(V_{*}^{\sigma, n_{0}}(x,t)+\gamma)V_{i,\gamma}^{\sigma, n_{0}}, &&x\in\bar\Omega_{n_{0}}, t>0.\\
&(H_{i,\gamma}^{\sigma,n_{0}}(x,0), V_{i,\gamma}^{\sigma,n_{0}}(x,0)) = (H_{+}^{n_{0}},V_{+}^{n_{0}}), && x\in \bar{\Omega}_{n_{0}}.
\end{aligned}
\right.
\end{equation*}
Moreover, by the comparison principle, we get
\begin{equation}\label{4-27}
(H_{i,\gamma}^{\sigma,n_{0}}(x,t+T_{1};H_{+}^{n_{0}}), V_{i,\gamma}^{\sigma,n_{0}}(x,t+T_{1};V_{+}^{n_{0}}))
\geq
(H_{i,\gamma}(x,t+T_{1};H_{+}), V_{i,\gamma}(x,t+T_{1};V_{+}))
\end{equation}
for all $(x,t)\in\bar{\Omega} \times [0,1]$,
where $(H_{i,\gamma}(x,t;H_{+}), V_{i,\gamma}(x,t;V_{+}))$ satisfies
\begin{equation*}
\left\{
\begin{aligned}
 &\tau \partial_{t} H_{i,\gamma}
=d_{1} \mathcal{K}^{1}_{\Omega,\sigma,2}[H_{i,\gamma}](x,t)-\rho(x,t)H_{i,\gamma}+\sigma_{1}(x,t)H_{u}(x)V_{i,\gamma},\quad &&x\in\bar\Omega, t>0,
\\
 &\tau \partial_{t} V_{i,\gamma}
= d_{2}\mathcal{K}^{2}_{\Omega,\sigma,2}[V_{i,\gamma}](x,t)+\sigma_{2}(x,t)(V_{*}^{\sigma,n_{0}}(x,t)-\gamma-V_{i,\gamma})^{+}H_{i,\gamma}
\\
&\hspace{1.5cm} -\mu(x,t)(V_{*}^{\sigma,n_{0}}(x,t)+\gamma)V_{i,\gamma},\quad &&x\in\bar\Omega, t>0.
\\
&(H_{i,\gamma}(x,0), V_{i,\gamma}(x,0)) = (H_{+},V_{+}) && x\in \bar{\Omega}.
\end{aligned}
\right.
\end{equation*}
Similar to the proof of Theorem \ref{th6.11}, we can obtain that there exists a unique positive 1-periodic function
$\left(H_{i,\gamma}^{*}, V_{i,\gamma}^{*}\right)$ such that
\begin{equation*}
\lim_{T \to \infty} (H_{i,\gamma}(x,t+T;u_{0}), V_{i,\gamma}(x,t+T;v_{0}))
=
\left(H_{i,\gamma}^{*}(x,t), V_{i,\gamma}^{*}(x,t)\right)
\enspace \text{uniformly in} \enspace \bar{\Omega}\times[0,1]
\end{equation*}
for any $(u_{0},v_{0})\in X^{+}$.
It then follows from \eqref{4-23} and the comparison principle that
\begin{equation}
(H_{i,\gamma}(x,t+T_{1};H_{+}), V_{i,\gamma}(x,t+T_{1};V_{+}))
\geq
(H_{i,\gamma}(x,t;H_{i,\gamma}^{*}), V_{i,\gamma}(x,t;V_{i,\gamma}^{*}))
=
\left( H_{i,\gamma}^{*}(x,t), V_{i,\gamma}^{*}(x,t) \right).
\end{equation}
for all $(x,t)\in\bar{\Omega}_{n_{0}}\times[0,1]$.
Moreover, by taking $n_{0}$ larger and $\tau$ smaller if necessary and using Lemma \ref{LEM8.8}, we have
\begin{equation}\label{4-29}
\left( H_{i,\gamma}^{*}(x,t), V_{i,\gamma}^{*}(x,t) \right)
\geq
\left(
H_{i,*}^{\sigma}(x,t) - \frac{\varepsilon}{4}, V_{i,*}^{\sigma}(x,t) -\frac{\varepsilon}{4} \right)
\enspace \text{for all}\enspace (x,t)\in\bar{\Omega}\times[0,1].
\end{equation}
Finally, by combining \eqref{4-22} and \eqref{4-24}-\eqref{4-29}, we get \eqref{4-21} holds.
\end{proof}

\subsection{Stem cell model}
In this subsection, motivated by \cite{Lei1,Lei2,SLLW}, we consider the following
multigenotype stem cell regeneration model with genetic mutations:

\begin{equation}\label{4-30}
\left\{
\begin{aligned}
&\tau \partial_t Q_i(x,t)
= 2\Bigl(1-\sum_{j\ne i} p_{j,i}\Bigr)\beta_i(x,t)
   \int_{\Omega} P_i(x,y,t)\,Q_i(y,t)\,e^{-\mu_i(y,t)}\,dy \\
&\hspace{2.1cm} + 2\sum_{j\ne i} p_{i,j}\,\beta_j(x,t)
   \int_{\Omega} P_j(x,y,t)\,Q_j(y,t)\,e^{-\mu_j(y,t)}\,dy \\
&\hspace{2.1cm} - Q_i(x,t)\Bigl(\beta_i(x,t)+\kappa_i(x,t)\,Q_i^{n}(x,t)\Bigr),
\qquad (x,t)\in \bar{\Omega}\times(0,\infty),\ 1\le i\le l,\\
& Q_i(x,0) = Q_{0i}(x), \hspace{6.3cm} x\in\bar{\Omega},\ 1\le i\le l.
\end{aligned}
\right.
\end{equation}
Here $Q_{i}(x,t)$ denotes the density of stem cell population with genotype $i$ and epigenetic state $x$ at time $t$;
$p_{ij}$ denotes the mutation rate from genetic type $j$ to type $i$;
$\beta_{i}$ denotes the proliferation rate;
$P_{i}$ represents the transition kernel from state $y$ to $x$;
$\mu_{i}$ denotes the apoptosis rate in proliferating
phase and cell cycle;
$\kappa_{i}$ denotes the removal rate, and $n$ characterizes the density dependence of the removal term.

Consider the following generalization of \eqref{4-30}:
\begin{equation}\label{4-31}
\left\{
\begin{aligned}
&\tau\partial_t Q_i =
\sum_{j=1}^{l} c_{ij}\beta_{j}(x,t)\int_{\Omega}k_{j}(x,y,t)Q_{j}(y,t)dy  -Q_i(\beta_i(x,t) + \kappa_i(x,t)Q_{i}^{n}), &&x\in\bar\Omega, t>0,\\
&Q_{i}(x,0) = Q_{0i}(x), &&x\in\bar\Omega, \\
&i=1,2, \cdots, l,
\end{aligned}
\right.
\end{equation}
We make the following assumption:
$c_{ij}\geq 0$, $\beta_{i}(x,t)$ and $\kappa_{i}(x,t)$ are positive continuous functions and $1-$periodic in $t$. $k_{i}$ satisfies (H3).
$D(x,t)$ is irreducible and satisfies (H1), where
$D(x,t)= (d_{ij}(x,t))_{l\times l}$ with $d_{ij}(x,t)=c_{ij}\beta_{j}(x,t)$.

By general semigroup theory in \cite{Pazy}, for any $Q_{0} = (Q_{01},Q_{02},\cdots,Q_{0l}) \in X$,
\eqref{4-31} admits a unique solution, denoted by $Q(x,t;Q_{0})$.
In what follows, we assume $Q_{0}\in X^{+}\backslash\{ 0 \}$.
Consider the time-periodic problem associated with \eqref{4-31}:
\begin{equation}\label{4-32}
\left\{
\begin{aligned}
&\tau\partial_t Q_i =
\sum_{j=1}^{l} c_{ij}\beta_{j}(x,t)\int_{\Omega}k_{j}(x,y,t)Q_{j}(y,t)dy  -Q_i(\beta_i(x,t) + \kappa_i(x,t)Q_{i}^{n}),&&x\in\bar\Omega, t\in\mathbb{R},\\
&Q_{i}(x,t)= Q_{i}(x,t+1)&&x\in\bar\Omega, t\in\mathbb{R},\\
&i=1,2, \cdots, l.
\end{aligned}
\right.
\end{equation}
Define the linear operator associated to the equation \eqref{4-32} linearized at zero by
\begin{equation*}
L_{n}[Q] = -\tau Q_{t}(x,t) + D(x,t)\mathcal{P}[Q](x,t) - (\beta(x,t)+ \kappa(n)(x,t))Q(x,t).
\end{equation*}
Here, $\beta(x,t) = \text{diag}(\beta_{1}(x,t),\cdots\beta_{l}(x,t))$, and
\begin{equation*}
\kappa(n)(x,t) =
\begin{cases}
&\!\!\!\!\!diag(\kappa_1(x,t),\cdots \kappa_l(x,t) ) \text{ if } n=0,
\\
&\qquad 0 \qquad\qquad\quad\,\qquad \text{ if } n>0.
\end{cases}
\end{equation*}
Let $\tilde{Q}(t,0),t>0$ be the evolution operator determined by equation $L_{n}[u]=0$.
Set $$\alpha(x,t) = \beta(x,t)+ \kappa(n)(x,t).$$

\begin{theorem}\label{THM4.8}
Suppose $n=0$. The following conclusions hold:

\begin{itemize}[leftmargin=0.9cm]
\item[(i)] If $s(L_{0}) >0$, then
\begin{equation*}
\lim\limits_{t\to+\infty} \| Q(\cdot,t;Q_{0}) \|_{X} = +\infty.
\end{equation*}

\item[(ii)] If $s(L_{0}) = 0$ and $s(L_{0})$ is the principal eigenvalue of $L_{0}$, then
\begin{equation*}
\lim\limits_{t\to+\infty}  \| Q(\cdot,t;Q_{0}) -  c\varphi(\cdot,t) \|_{X} = 0,
\end{equation*}
where
$c$ is a constant depending on $Q_{0}$ and $\varphi$ is the principal eigenfunction of $s(L_{0})$.

\item[(iii)] If $s(L_{0}) < 0$, then
\begin{equation}\label{4-33}
\lim\limits_{t\to+\infty} \|Q(\cdot,t;Q_{0})\|_{X} = 0.
\end{equation}

\end{itemize}
\end{theorem}
\begin{proof}
(i)
If $s(L_{0}) >0$, by the approximation arguments (Theorem \ref{THM1.5}), there exists $\alpha_{\varepsilon}$ and $(\lambda_{\varepsilon},\varphi_{\varepsilon})$ such that $\alpha_{\varepsilon} \geq \alpha$, $\lambda_{\varepsilon}>0$ and
\begin{equation*}
-\tau (\varphi_{\varepsilon})_{t} + D(x,t)\mathcal{K}[\varphi_{\varepsilon}](x,t) - \alpha_{\varepsilon}(x,t)\varphi_{\varepsilon}(x,t) = \lambda_{\varepsilon}\varphi_{\varepsilon}(x,t)
\enspace \text{in} \enspace
\bar{\Omega} \times [0,1].
\end{equation*}
Choose $\delta>0$ small enough such that
$\delta\varphi_{\varepsilon}(x,0) \leq Q_{0}(x)$, and let $\phi(x,t) = \delta e^{\frac{\lambda_{\varepsilon}}{\tau}t}\varphi_{\varepsilon}(x,t) $.
It is easy to check that $\phi$ is the lower-solution of \eqref{4-31}. As a result, we get the desired conclusion.

(ii) By Proposition \ref{PP2.10}, $r(\tilde{Q}(1,0))=e^{ \frac{s(L_{0})}{\tau}}=1$ is an isolated algebraically simple principal eigenvalue of $\tilde{Q}(1,0)$. Clearly, $\varphi_{0} = \varphi(x,0)$ is the corresponding eigenfunction of $\tilde{Q}(1,0)$.
Let $T$ be the dual operator of $\tilde{Q}(1,0)$ with the principal eigenfunction $\psi$.
Define the spectral projection $P$ on $X$ by
$Pu = \frac{\langle u , \psi \rangle}{\langle \varphi_{0}, \psi \rangle} \varphi_{0}$, where $\langle \cdot , \cdot \rangle$ denotes the duality pair between $X$ and its dual space.
Let $X_{1} = PX$ and $X_{2} = (I - P)X$. It easy to check that
$X = X_{1} \oplus X_{2}$
and $X_{2} = ker P = \{u\in X \,|\, \langle u , \psi \rangle =0 \}$. Using this decomposition, for any $Q_{0} \in X$, there exist
constant $c$ and $v \in X_{2}$ such that
\begin{equation*}
Q_{0} = c\varphi_{0} + v.
\end{equation*}
Hence,
\begin{equation}\label{4-34}
Q(x,t;Q_{0}) = \tilde{Q}(t,0)Q_{0}(x) = c\varphi(x,t) + \tilde{Q}(t,0)v.
\end{equation}
Note that $r(\tilde{Q}(1,0) |_{X_{2}} ) < 1$.
Moreover, by the Gelfand's formula
\begin{equation*}
\lim\limits_{n\to\infty} \| (\tilde{Q}(1,0)|_{X_{2}} )^{n}\|^{\frac{1}{n}} = r(\tilde{Q}(1,0)|_{X_{2}}) < 1.
\end{equation*}
This together with \eqref{4-34} implies that
\begin{equation*}
\lim\limits_{t \to \infty} \| Q(x,t;Q_{0}) - c\varphi(x,t)  \|_{X}
\leq
\lim\limits_{t \to \infty} \|\tilde{Q}(t,0)v \|_{X}
=0.
\end{equation*}
On the other hand, $\langle Q_{0} , \psi \rangle = \langle c\varphi_{0} + v , \psi \rangle = c\langle \varphi_{0} , \psi \rangle$. Hence, $c = \frac{\langle Q_{0} , \psi \rangle  }{ \langle \varphi_{0} , \psi \rangle }$.

(iii) This case can be obtained by the similar arguments in (ii),  thus is omitted here.
\end{proof}

\begin{theorem}\label{THM4.9}
Suppose $n>1$. The following conclusions hold:
\begin{itemize}[leftmargin=0.9cm]
\item[(i)] If $s(L_{n}) >0$, then \eqref{4-32} admits a unique positive solution $\hat{Q}(x,t)$, and there holds
\begin{equation*}
\lim\limits_{t\to+\infty} \| Q(\cdot,t;Q_{0}) - \hat{Q}(\cdot,t)\|_{X} = 0.
\end{equation*}

\item[(ii)] If $s(L_{n}) = 0$ is the principal eigenvalue, then \eqref{4-32} admits no solution in $X^{+}\backslash \{ 0\}$.

\item[(iii)] If $s(L_{n}) < 0$, then \eqref{4-32} admits no solution in $X^{+}\backslash \{ 0\}$, and there holds
\begin{equation}\label{4-35}
\lim\limits_{t\to+\infty} \| Q(\cdot,t;Q_{0}) \|_{X} = 0.
\end{equation}

\end{itemize}
\end{theorem}
\begin{proof}
The proofs of (i) and (iii) are similar to the proof of Theorem \ref{th6.11}, and the proof of (ii) follows the argument in
Step 2 of Theorem \ref{T Zika 1}. We omit the details.
\end{proof}

To highlight the dependence on $n$,
We write $\alpha$ as $\alpha(n)$.
It is easy to check that
\begin{equation*}
\max_{\bar{\Omega}}\lambda_{\alpha(0)}(x)= \max_{1\le i\le l,x\in\bar{\Omega}} - \int_{0}^{1} \left( \kappa_{i}(x,t) + \beta_{i}(x,t) \right) dt < 0,
\end{equation*}
and for $n>0$,
\begin{equation*}
\max_{\bar{\Omega}}\lambda_{\alpha(n)}(x)= \max_{1 \leq i\leq l,x\in\bar{\Omega}} - \int_{0}^{1} \kappa_{i}(x,t) dt < 0,
\end{equation*}

Define $\bar{\beta} = \max_{1\le i\le l,(x,t)\in\bar{\Omega}} \beta_{i}(x,t)$.
Combining Theorem \ref{THM4.8}, Theorem \ref{THM1.7} and Theorem \ref{THM1.8} (i), we have the following result.

\begin{corollary}
Suppose $n=0$. The following results hold.
\begin{itemize}[leftmargin=0.8cm]

\item[(i)]
There exists $0< \beta_{2}\ll 1 $ such that
\eqref{4-33} holds for all $0< \bar\beta < \beta_{2}$.

\item[(ii)]
Assume that $k_{i}$ satisfies ${ \rm (\tilde{H}3) }$ and has compact support for $1 \leq i\leq l$. Then there exists  $\sigma_{1}\gg 1$
such that  \eqref{4-33} holds for all $\sigma> \sigma_{1}$.

\end{itemize}
\end{corollary}

Combining Theorem \ref{THM4.9} and Theorem \ref{THM1.8} (i), we have the following result.
\begin{corollary}
Assume that $n>1$, and $k_{i}$ satisfies ${ \rm (\tilde{H}3) }$ and has compact support for $1 \leq i\leq l$.
Then there exists  $\sigma_{1}\gg 1$
such that \eqref{4-35} holds for all $\sigma> \sigma_{1}$.
\end{corollary}

\section*{Acknowledgments}
\noindent

Research of W.-T. Li was partially supported by NSF of China (12531008; 12271226). Research of J.-W. Sun was partially supported by NSF of China (12371170).

\end{document}